[1994/12/01]
\documentclass[10pt,preprint]{amsart}
\usepackage{mathptmx}

\usepackage[english]{babel}
\usepackage[utf8]{inputenc}
\usepackage{textcomp}  
\usepackage[T1]{fontenc}
\usepackage{graphicx}
\usepackage[automark]{scrpage2}
\usepackage[a4paper,left=4cm, right=2cm,top=3cm, bottom=3cm]{geometry} 
\usepackage{comment}
\reversemarginpar

\usepackage{color}
\usepackage{hyperref}

\usepackage{enumitem}

\usepackage{amsthm,amsmath,amssymb,amsfonts}
\usepackage{mathrsfs}
\usepackage{float}
\usepackage{tikz}

\newtheorem{Theorem}{Theorem}[section]
\newtheorem{Lemma}[Theorem]{Lemma}
\newtheorem{Proposition}[Theorem]{Proposition}

\theoremstyle{definition} 

\newtheorem{Remark}[Theorem]{Remark}
\newtheorem{Definition}[Theorem]{Definition}

\newcommand{\R}{\mathbb{R}}

\newcommand{\N}{\mathbb{N}}

\newcommand{\eps}{\varepsilon}
\newcommand{\cal}{\mathcal}

\newcommand{\maxindex}{\hat{\iota}}

\makeatletter
\providecommand\@dotsep{5}
\def\listtodoname{List of Todos}
\def\listoftodos{\@starttoc{tdo}\listtodoname}
\makeatother

\begin{document}


\title{Sparse Control of alignment models in high dimension}

\author[Mattia Bongini]{Mattia Bongini}
\author[Massimo Fornasier]{Massimo Fornasier}
\author[Oliver Junge]{Oliver Junge}
\author[Benjamin Scharf]{Benjamin Scharf}




\address{Mattia Bongini. Technische Universit\"at M\"unchen, Fakult\"at Mathematik, Boltzmannstrasse 3, D-85748 Garching, Germany.  E-mail {\sf mattia.bongini@ma.tum.de}}
\address{Massimo Fornasier. Technische Universit\"at M\"unchen, Fakult\"at Mathematik, Boltzmannstrasse 3, D-85748 Garching, Germany.  E-mail {\sf massimo.fornasier@ma.tum.de}}
\address{Oliver Junge. Technische Universit\"at M\"unchen, Fakult\"at Mathematik, Boltzmannstrasse 3, D-85748 Garching, Germany.  E-mail {\sf oj@tum.de}}
\address{Benjamin Scharf. Technische Universit\"at M\"unchen, Fakult\"at Mathematik, Boltzmannstrasse 3, D-85748 Garching, Germany.  E-mail {\sf benjamin.scharf@ma.tum.de}}


\begin{abstract}For high dimensional particle systems, governed by smooth nonlinearities depending on mutual distances between particles, one can construct low-dimensional representations of the dynamical system,  which allow the learning of nearly optimal control strategies in high dimension with overwhelming confidence. In this paper we present an instance of this general statement tailored to the sparse control of models of consensus emergence in high dimension, projected to lower dimensions by means of random linear maps. We show that one can steer, nearly optimally and  with high probability, a high-dimensional alignment model to consensus by acting at each switching time on one agent of the system only, with a control rule chosen essentially exclusively according to information gathered from a randomly drawn low-dimensional representation of the control system.
\end{abstract}

\maketitle


\section*{Introduction}
In view of the increasing technical ability of collection of large amounts of time-evolving data and of potentially modeling them into  high-dimensional dynamical systems,  the controllability of complex multi-agent interactions has become an actual challenge of paramount importance due to its social and economical impact. In this paper, we shall investigate the applicability of the following
\\

\noindent {\bf Meta-theorem.} \emph{For high dimensional particle systems, governed by smooth nonlinearities depending on mutual distances between particles, one can construct low-dimensional representations of the dynamical system,  which allow the learning of nearly optimal control strategies in high dimension with overwhelming confidence.}\\

\noindent

As control is usually goal-oriented, hence highly dependent on the specific dynamical system, investigating the qualitative applicability of this statement in its full generality may risk to dilute its quantitative understanding. Thus we shall prove in this paper a specific instance of it, which conveys nonetheless all the relevant aspects and technical issues potentially encountered in other situations. In particular we shall focus on alignment models inspired by the seminal work of Cucker and Smale \cite{CS07,cusm07}.
In this class of dynamical systems the particles influence each other according to a positive rate of communication $a\left(\left\Vert x_{i}(t)-x_{j}(t)\right\Vert\right)$ depending on the mutual distance towards the alignment of the entire group to a common conduct, and they read
$$
\left \{\begin{array}{lll}
\dot{x}_{i}(t) & = & v_{i}(t)\\
\dot{v}_{i}(t) & = & \frac{1}{N}\underset{j=1}{\overset{N}{\sum}}a\left(\left\Vert x_{i}(t)-x_{j}(t)\right\Vert\right)\left(v_{j}(t)-v_{i}(t)\right), \quad i=1,\dots, N.
\end{array}
\right .
$$
The classically mentioned inspiring application is the modeling of the emergence of a flock moving with the same velocity in a group of migrating birds.  However, the emergence of a common direction may be depending on whether the initial conditions lay within a corresponding basin of attraction and such conditional pattern formation has been fully explored, for instance, in  \cite{CFPT12,CFPT14,HHK10}: for 
\begin{align*}
X(t)&:=\overline{x^2(t)}-\overline{x}^2(t), \\
V(t)&:=\overline{v^2(t)}-\overline{v}^2(t)
\end{align*}
the following result holds.
\begin{Theorem}[\cite{HHK10}]\label{hhktm} If $\int_{\sqrt{X(0)}}^{\infty} a\left(\sqrt{2N}r\right) \ dr \geq \sqrt{V(0)}$, then $\lim_{t\rightarrow \infty} V(t)=0$, meaning that $\lim_{t\rightarrow \infty} v_i(t) = \overline v$, for all $i=1,\dots,N$.
\end{Theorem}
In those initial conditions where 
$$
\int_{\sqrt{X(0)}}^{\infty} a\left(\sqrt{2N}r\right) \ dr < \sqrt{V(0)},
$$
and the convergence towards alignment is not anymore guaranteed, despite being desirable, for instance when it comes to unanimous decisions in assemblies, one may wonder whether the application of a parsimonious external control can lead nevertheless to consensus emergence. 

\begin{figure}[H]
\centering
\includegraphics[scale=0.8]{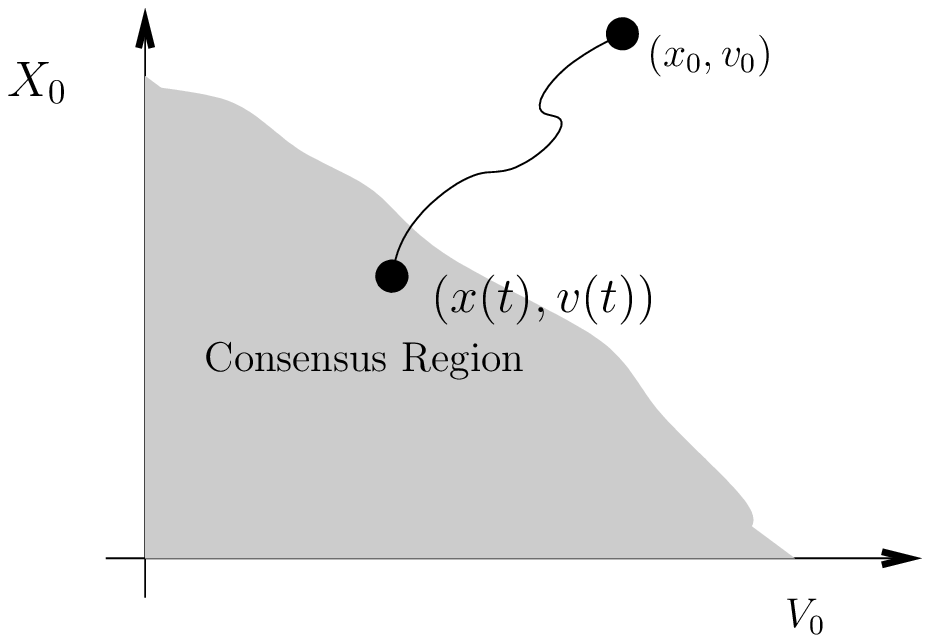}
\caption{Steering the alignment system to a point fulfilling the conditions of Theorem \ref{hhktm} towards consensus formation.} 
\label{consreg}
\end{figure}

This issue has been recently explored in the series of papers \cite{CFPT12,CFPT14}, where the {\it sparse} controllability of alignment models towards consensus have been established (see Figure \ref{consreg}) regardless of the dimensionality of the problem, see also \cite{CD14,bkf14} for extensions and generalizations.
In particular, alignment should not be interpreted exclusively relative to motion in the three dimensional Euclidean space, but there are several instances of ``abstract alignment'' which may occur in high-dimension, for instance in \cite{ahn13} the authors consider an application of alignment models to predict the collective phenomenon of asset pricing and volatilities in financial markets. Therefore, in those circumstances where the dimensionality of the dynamics is very high, it becomes a relevant question whether it is possible to define control strategies of the dynamics by observing instances of the system in lower dimension.
\\
In recent years, several techniques have been developed in order to reduce the dimensionality of time-evolving point  clouds, such as {\it diffusion maps} applied to networks changing in time \cite{coifman14} and geometric multiscale dimensionality reductions \cite{maggioni12}, just to mention a few. Besides these perhaps involved methods based on computationally demanding nonlinear embeddings of the high-dimensional clouds in lower dimension, Johnson--Lindenstrauss embeddings, introduced in the seminal work \cite{JL84}, have the remarkable property of being simple {\it linear operators} $M \in \mathbb R^{k\times d}$ preserving the distances between points in the cloud $\mathcal P \subset \mathbb R^d$ up to an $\varepsilon$-distortion:
$$
(1-\varepsilon) \| x-x' \| \leq \| M x- M x' \| \leq (1+\varepsilon) \| x-x' \|, \quad \text{for all} \quad x,x' \in \mathcal P,
$$
where 
$$
k \sim \varepsilon^{-2} \log (\# \mathcal P).
$$
 As Johnson--Lindenstrauss embeddings with such scaling of the low-dimension are constructed by generating random projections, the quasi-isometry property on the point cloud is usually stated with a certain (high) probability. \\
The random linear projection of high-dimensional systems governed by smooth nonlinearities depending on mutual distances has been investigated in \cite{FHV13}: roughly speaking, given a dynamical system in high-dimension  $d \gg 1$ governed by locally Lipschitz functions $f_i:\mathbb R_+^{N\times N} \to \mathbb R^d$
$$
 \dot z_i = f_i ((\|z_j - z_{\ell}\|)_{j{\ell}}) \in \mathbb R^d, \quad i=1,\dots,N
$$
 and its lower-dimensional counterpart
 $$
 \dot \zeta_i = M f_i ((\|\zeta_j - \zeta_{\ell}\|)_{j{\ell}}) \in \mathbb R^k, \quad i=1,\dots,N,
 $$
where $M:\mathbb R^d \to \mathbb R^k$ is a Johnson--Lindenstrauss linear embedding for  $k \sim \varepsilon^{-2} \log(N)$, the following finite time approximation holds
 $$
\| \zeta_i(t)- M z_i(t)\| \leq C_T \varepsilon, \quad \mbox{ for all } t \in [0,T],
 $$
with high probability.  If we applied such linear projections verbatim to each equation of a Cucker--Smale system, we would obtain the following approximation 
\begin{align*}
\frac{d}{d t}Mv_i(t)&=\frac{1}{N} \sum\limits_{j=1}^N a\left(\left\Vert x_j(t)-x_i(t)\right\Vert\right) \left(Mv_j(t)-Mv_i(t)\right) \\
&\overset{\text{wish}}{\sim} \frac{1}{N} \sum\limits_{j=1}^N a\left(\left\Vert Mx_j(t)-Mx_i(t)\right\Vert\right) \left(Mv_j(t)-Mv_i(t)\right),
\end{align*} 
leading to the formulation of the low-dimensional system in $\mathbb R^k$
$$
\left \{\begin{array}{lll}
\dot{y}_{i}(t) & = & w_{i}(t)\\
\dot{w}_{i}(t) & = & \frac{1}{N}\underset{j=1}{\overset{N}{\sum}}a\left(\left\Vert y_{i}(t)-y_{j}(t)\right\Vert\right)\left(w_{j}(t)-w_{i}(t)\right), \quad i=1,\dots, N,
\end{array}
\right .
$$
with initial conditions $(y(0),w(0)) = (M x(0),Mv(0))$. The first result of this paper, refining and generali\-zing those in \cite{FHV13}, is roughly summarized as follows.
\begin{Theorem} Let $(x,v)$ be a solution of the $d$-dimensional Cucker--Smale system for given initial values $x(0),v(0) \in \R^{N\times d}$, and let  $M \in \R^{k \times d}$ be a Johnson--Lindenstrauss matrix
for a suitable $\eps>0$ distortion parameter and low dimension $k$ depending on the logarithm of the number of agents $N$.  Then the $k$-dimensional solution $(y,w)$ with initial values $(y(0),w(0))=(Mx(0),Mv(0))$ stays close to the projected $d$-dimensional trajectory $(Mx,Mv)$, i.e.,
\begin{align}\label{fintim}
\|y(t)-Mx(t)\| + \|w(t)-Mv(t)\| \lesssim \eps e^{Ct}, \quad t \leq T. 	
\end{align}
 \end{Theorem}
As we highlight in details in Section \ref{sec:withoutControl}, not only the approximation \eqref{fintim} holds for finite time, but, remarka\-bly, the lower dimensional representation also shows also a rather impressive faithfulness in terms of the asymptotic (long time) detection of collective behavior emergence, i.e., global alignment occurs in lower dimension $k$ if and only if it occurs in high dimension $d$ with high probability. The key technical tool for proving this result and the ones following is a weak form of the Johnson--Lindenstrauss Lemma, formulated below in Lemma \ref{le:JLL_cont}, valid for continuous trajectories and not only for clouds of points. Similar results appear, to some extent in greater generality in \cite{MR2472287,dirk}, but not in the weak form we consider here. \\

Additionally we combine the analysis of \cite{FHV13} with the sparse controllability results in \cite{CFPT12} and show that a high-dimensional dynamical systems of Cucker--Smale type can be {\it nearly optimally} stabilized towards consensus by means of a control strategy completely identified by the {\it optimal} control strategy in low-dimension with high probability. 
More formally we consider for a  given $(x(0),v(0))$ the high-dimensional controlled system
\begin{align*}
\begin{cases} 
\dot{x_i}(t)=v_i(t) \\
\dot{v_i}(t)=\frac{1}{N} \sum\limits_{j=1}^N a\left(\left\Vert x_j(t)-x_i(t)\right\Vert\right) \left(v_j(t)-v_i(t)\right)+{u_i^h(t)}
\end{cases}
\end{align*}
and its low-dimensional system counterpart with initial data $(y(0),w(0))=(Mx(0),Mv(0))$,
\begin{align*}
\begin{cases} 
\dot{y_i}(t)=y_i(t) \\
\dot{w_i}(t)=\frac{1}{N} \sum\limits_{j=1}^N a\left(\left\Vert y_j(t)-y_i(t)\right\Vert\right) \left(w_j(t)-w_i(t)\right)+{u_i^{\ell}(t)}.
\end{cases}
\end{align*}
The {\it sparse} control strategies applied to the systems are defined as follows: fix $\theta>0$ and define for $w_{i}^{\perp}= w_i - \overline{w}$ as well as $v_i^\perp= v_i - \overline{v}$
\begin{align*}
u_i^{\ell}&=
\begin{cases} 
-\theta\frac{w_{\maxindex}^{\perp}}{\|w_{\maxindex}^{\perp}\|} &\text{ if } i=\maxindex \text{ is the smallest index such that } \|w_{\maxindex}^{\perp}\|=  \max\limits_{j=1,\ldots,N} \|w_j^{\perp}\|, \\
0 &\text{ otherwise, }
\end{cases} \\
u_i^{h}&=
\begin{cases} 
-\theta\frac{v_{\maxindex}^{\perp}}{\|v_{\maxindex}^{\perp}\|} &\text{ if } i=\maxindex, \\
0 &\text{ otherwise. }
\end{cases}
\end{align*}
Notice that the control $u^{h}$ is {\it sparse} (all the components are zero except one) and defined exclusively through the following information: the index $\maxindex$ which is computed from the low-dimensional control problem, the consensus parameter $v_{\maxindex}$, which is actually the only information to be observed in high-dimension, and the mean consensus parameter $\overline{v}(t) = \overline{v}(0) + \frac{1}{N} \sum_{i=1}^N\int_{0}^t u_i^h(s) ds$, which one does compute by integration and sums of previous controls.
Our main result reads as follows.
\begin{Theorem}
Let $M \in \R^{k \times d}$ and $\Theta>0$. Assume that $(x,v)$ and $(y,w)$ are solutions of the $d$-dimensional and $k$-dimensional so controlled Cucker--Smale systems with  initial values $(x(0),v(0))$ and $(Mx(0),Mv(0))$, respectively.
Further assume that $M$ is a Johnson--Lindenstrauss matrix for a certain distortion $\eps>0$ and low dimension $k$, which depends exponentially on the number of agents, but not on the dimension $d$.
Then both controlled Cucker--Smale systems 
\begin{itemize}
\item[(a)] stay close to each other after the projection of the high-dimensional trajectories;
\item[(b)] reach the consensus region of Theorem \ref{hhktm}  in finite time, and
\item[(c)] reach the  consensus region, when a certain parameter of the low-dimensional systems falls below a known threshold.
\end{itemize}
\end{Theorem}

We consciously do not wish to be more detailed at this point than this rather general and perhaps rough explanation because the precise statements appear in the rest of the paper in a rather technical form and we wish here, in the introduction, mainly to convey their fundamental message.  Let us stress again that in our view the content of this paper is of technical nature towards a proof of concept and we expect our main results actually to extend similarly to other high-dimensional dynamical systems whose nonlinearities depend smoothly on mutual Euclidean distances. We refer to \cite{FHV13} for more examples of relevant dynamical systems of this type. While in this paper we consider the sparse controllability of alignment systems for $d \rightarrow \infty$, we mention also the related investigations towards a sparse mean-field optimal control for $N \rightarrow \infty$ in \cite{FPR14,FS14}.\\

The paper is organized as follows: Section \ref{sec:enteringCucker} presents the Cucker--Smale model and some of its main features. Section \ref{sec:enteringJohnson} deals with Johnson--Lindenstrauss embeddings, which shall be used extensively to obtain  low-dimensional counterparts of  Cucker--Smale models retaining all the information about the asymptotic behavior of the system for large times. Section \ref{sec:withoutControl} studies the interplay between a high-dimensional Cucker--Smale model and the low-dimensional system obtained via Johnson--Lindenstrauss embeddings: in particular, in Theorem \ref{th:Reduction_uncontrolled} we derive an error estimate for the approximation of the projected high-dimensional system by the low-dimensional one. Section \ref{sec:withControl} introduces the sparse control strategy we shall exploit to enforce alignment in the high-dimensional system using only information gathered from the low-dimensional system and presents  Theorem \ref{eq:convergence}, the main result of this paper. In Section \ref{sec:JLL_matrix} we discuss about the appropriate size of the dimension onto which we should project a given high-dimensional system and the construction of suitable Johnson--Lindenstrauss embeddings fulfilling the conditions stated in the main result. Finally, Section \ref{sec:numerics} shows a series of numerical experiments and compares the sparse control strategy to several other possible stabilization procedures.

\noindent

\noindent

\section{The Cucker--Smale model} \label{sec:enteringCucker}

In the following, we shall work in the ambient space $\R^d$ equipped with the $\ell^d_2$-Euclidean norm $\Vert \cdot \Vert_{\ell^d_2}$, omitting the subscript if the dimensionality of the norm can be retrieved from the context. Consider a system of $N$ agents, whose state is described by a pair $(x_i, v_i)$ of vectors of $\R^d$, where $x_i$ represents the \emph{main state} of the agent and $v_i$ its \emph{consensus parameter}. The \emph{alignment} model as presented in \cite{HHK10} assumes that the dynamics of the $i$-th agent of the group evolves according to  the following system of ordinary differential equations
\begin{align} \label{eq:freesystemhi}
\left \{\begin{array}{lll}
\dot{x}_{i}(t) & = & v_{i}(t)\\
\dot{v}_{i}(t) & = & \frac{1}{N}\underset{j=1}{\overset{N}{\sum}}a\left(\left\Vert x_{i}(t)-x_{j}(t)\right\Vert\right)\left(v_{j}(t)-v_{i}(t)\right)
\end{array}
\right .
\end{align}
for every $i=1,\ldots,N$, where $a$ is a \emph{non-increasing positive Lipschitz function} on $[0,\infty)$. In this model, at any time every agent adjusts its consensus parameter to match those of the other agents according to a weighted average of the differences: how much the $i$-th agent will align with the $j$-th agent depends on the Euclidean distance, meaning that the $i$-th agent is more influenced by those which are near to him than to those which are far away from him.

As a prominent example, the Cucker--Smale models considered in the seminal paper \cite{CS07} are governed by a function $a$ of the form
\begin{align}\label{csmod}
a(r) = \frac{K}{(\sigma^2 + r^2)^{\beta}},
\end{align}
where the parameters $K > 0$, $\sigma > 0$, and $\beta \geq 0$ tune the social interaction in the group of agents.

\begin{Definition}
We say that a solution $(x(t), v(t))$ of system \eqref{eq:freesystemhi} \emph{tends to consensus} if the consensus para\-meter vectors tend to the mean $\overline{v} = \frac{1}{N}\sum^N_{i = 1} v_i$, namely if 
\begin{align*}
\lim_{t \rightarrow + \infty}  \left\|v_i(t) - \overline{v}(t)\right\|_{\ell^d_2}=\lim_{t \rightarrow + \infty} \left\|v_i^{\perp}(t)\right\|_{\ell^d_2}= 0
\end{align*}
for every $i = 1, \ldots, N$. Notice that $\overline{v}(t) = \overline{v}(0)$ is a conserved quantity for a system of the type \eqref{eq:freesystemhi}, but later we shall consider below controlled systems for which $\overline{v}(t)$ is eventually time dependent.
\end{Definition}
Given a solution $(x(t), v(t))$ of system \eqref{eq:freesystemhi}, we reformulate the convergence to consensus by means of the following quantities
\begin{align}
\label{Def:XY}
X(t):= B(x(t), x(t))= \frac{1}{2N^2} \sum^N_{i,j = 1} \Vert x_i(t)-x_j(t) \Vert^2, \qquad V(t) := B(v(t), v(t)) = \frac{1}{N} \sum^N_{i = 1} \Vert v_i^{\perp}(t) \Vert^2,
\end{align}
where for $u=(u_1,\ldots,u_N), \tilde{u}=(\tilde{u}_1,\ldots,\tilde{u}_N) \in (\R^d)^N$
\begin{align*}
B(u,\tilde{u}) = \frac{1}{2N^2} \sum^N_{i,j = 1} \left\langle u_i - u_j, \tilde{u}_i - \tilde{u}_j \right\rangle
\end{align*}
is a bilinear form on  the space $(\R^d)^N$, and  $\langle \cdot, \cdot \rangle$ denotes the usual scalar product on $\R^d$. 

If we denote with
\begin{align*}
\mathcal{V} & = \left\{ v \in (\R^d)^N \mid v_1 = \ldots = v_N \in \R^d \right\}, \\
\mathcal{V}^{\perp} & = \left\{ v \in (\R^d)^N \mid \sum^N_{i = 1}v_i = 0 \right\},
\end{align*}
then $(\R^d)^N = \mathcal{V} \oplus \mathcal{V}^{\perp}$ with respect to the scalar product $B$, hence every $v \in (\R^d)^N$ can be written uniquely as $v = v_0 + v^{\perp}$ where $v_0 \in \mathcal{V}$ and $v^{\perp} \in \mathcal{V}^{\perp}$.

\begin{Proposition}
For a solution $(x(t), v(t))$ of system \eqref{eq:freesystemhi} the following are equivalent:
\begin{enumerate}
\item $\lim_{t \rightarrow + \infty}  \left\|v_i(t) - \overline{v}(t)\right\|_{\ell^d_2} = 0$ for every $i = 1, \ldots, N$;
\item $\lim_{t \rightarrow + \infty} v^{\perp}_i(t) = 0$ for every $i = 1, \ldots, N$;
\item $\lim_{t \rightarrow + \infty} V(t) = 0$.
\end{enumerate}
\end{Proposition}

A sufficient condition for a solution of system \eqref{eq:freesystemhi} to converge to consensus can be given using the following functional
\begin{align*}
\gamma(X_0):=\int^{+\infty}_{\sqrt{X_0}} a(\sqrt{2N} r) \ dr.
\end{align*}

\begin{Lemma}[\cite{HHK10}, Corollary 3.1]
\label{le:HaHaKim}
Let $(x(t), v(t))$ be a solution of system \eqref{eq:freesystemhi}. Then $X(t)$ and $V(t)$ satisfy
\begin{align*}
\frac{d}{dt} V(t) \leq -2a\left(\sqrt{2NX(t)}\right) V(t).
\end{align*}
In particular, if the initial datum $(x(0),v(0))=(x_0, v_0) \in (\R^d)^N \times (\R^d)^N$ is such that the quantities $X_0=X(0)$ and $V_0=V(0)$ are fulfilling
\begin{align*}
\gamma(X_0) \geq \sqrt{V_0},
\end{align*}
then the solution of \eqref{eq:freesystemhi} with initial data $(x_0, v_0)$ tends to consensus.
\end{Lemma}

\begin{Remark}
A simple proof of this crucial observation can be found in the Appendix of \cite{CFPT14}. Notice that it follows immediately that $V$ is decreasing.
\end{Remark}

\begin{Definition}[Consensus region]
\label{def:consensus}
If $(x(t),v(t))$ fulfills the condition 
\begin{align*}
\gamma(X(t)) \geq \sqrt{V(t)},
\end{align*}
we say that the system is in the \emph{consensus region} at the time $t$.
\end{Definition}

\section{A continuous Johnson--Lindenstrauss Lemma} \label{sec:enteringJohnson}
As it will be made clear below, we indent to reduce the computational effort of extracting fundamental features of the dynamical system \eqref{eq:freesystemhi}, for instance about its asymptotic behavior, by projecting it to a $k$-dimensional space for $k \ll d$ by a linear mapping $M\in \R^{k\times d}$. In particular, we apply such a matrix $M$ to each equation of \eqref{eq:freesystemhi} and by setting $y_i = Mx_i$ as well as $w_i = Mv_i$ for $i = 1, \ldots, N$, we  obtain the system
\begin{align*}
\left \{\begin{array}{lll}
\dot{y}_{i}(t) & = & w_{i}(t)\\
\dot{w}_{i}(t) & = & \frac{1}{N}\underset{j=1}{\overset{N}{\sum}}a\left(\left\Vert y_{i}(t)-y_{j}(t)\right\Vert\right)\left(w_{j}(t)-w_{i}(t)\right),
\end{array}
\right .
\end{align*}
where we formally applied the equivalences 
\begin{align} \label{eq:approx_space}
\left\Vert x_{i}(t)-x_{j}(t)\right\Vert_{\ell^d_2} \equiv \left\Vert Mx_{i}(t)-Mx_{j}(t)\right\Vert_{\ell^k_2} \equiv \left\Vert y_{i}(t)-y_{j}(t)\right\Vert_{\ell^k_2}.
\end{align}
For \eqref{eq:approx_space} to hold, at least approximately, we need that $M$ is nearly an isometry (here we further refine and extend results from \cite[Section 3]{FHV13}).
\begin{Definition}
\label{def:JLL_property}
Let $M \in \R^{k \times d}$, $\delta>0$, and $\eps\in (0,1)$. Then we say, that $M$ is fulfilling the \textit{weak Johnson--Lindenstrauss property} of parameters $\eps$ and $\delta$ at $x \in \R^d$  if either
\begin{align} 
\label{eq:JLL_xquasi}
(1-\eps) \|x\| \leq \|Mx\| \leq (1+\eps) \|x\|
\end{align} 
or
\begin{align} 
\label{eq:JLL_xsmall}
\|x\| \leq \delta \text{ and } \|Mx\| \leq \delta.
\end{align}
We say that $M$ is fulfilling the \textit{(strong) Johnson--Lindenstrauss property} of parameter $\eps$ at $x \in \R^d$  if exclusively \eqref{eq:JLL_xquasi} holds at $x \in \R^d$.
\end{Definition} 
\begin{Remark} \label{rm:numberofpoints}
The earliest result providing the existence of matrices $M$ for which \eqref{eq:JLL_xquasi} holds for every $x \in \mathcal{P}$, $\mathcal{P} \subseteq \R^d$ such that $N=\# \mathcal{P}$ for the dimensionality $k$ scaling as
\begin{align}
\label{eq:scaling}
k \sim \eps^{-2} \log N
\end{align}
is the celebrated Johnson--Lindenstrauss Lemma from the seminal paper \cite{JL84}. We refer to \cite{dirk} for a rather general version of this result and to the references therein for an extended literature.

The only construction of a matrix $M$ fulfilling the (strong) Johnson--Lindenstrauss property with scaling \eqref{eq:scaling} known up to now is stochastic, i.e., the matrix is randomly generated and satisfies \eqref{eq:JLL_xquasi} with high probability. One of the remarkable features of these embeddings, which we exploit extensively in this paper, is that for their construction there is no need to know the specific points in advance: given a fixed cloud of points (not necessarily explicitely given!) a random matrix drawn according to certain distributions will fulfill the (strong) Johnson--Lindenstrauss property with high probability. Let us recall briefly some well-known instances of such distributions:
\begin{itemize}
\item[(S1)] $k \times d$ matrices $M$ whose entries $m_{ij}$ are independent realizations of Gaussian random variables, i.e.,
\begin{align*}
m_{ij} \sim \cal{N}\left(0,\frac{1}{k}\right);
\end{align*}
\item[(S2)] \label{Bern}$k \times d$ matrices $M$ whose entries $m_{ij}$ are independent realizations of scaled Bernoulli random variables, i.e.,
\begin{align*}
m_{ij} = \left\{
\begin{array}{lr}
+\frac{1}{\sqrt{k}}, & \text{ with probability } \frac{1}{2} \\
-\frac{1}{\sqrt{k}}, & \text{ with probability } \frac{1}{2} \\
\end{array}
\right. .
\end{align*}
It holds $1 \leq \|M\|_{\ell_2^d \rightarrow \ell_2^k} \leq \sqrt{d}$.

\item[(S3)] \label{DG99} $k \times d$ matrices $M$ which are random projections and are scaled by a factor $\sqrt{d/k}$, see \cite{DG03}. In particular, it holds $\|M\|_{\ell_2^d \rightarrow \ell_2^k}  = \sqrt{d/k}$.
\end{itemize}

\end{Remark}

\begin{Remark}
While the Johnson--Lindenstrauss Lemma is a result for a finite number of points, we need an analogous continuous result for projecting trajectories of dynamical systems. A result in this direction was given in \cite[Theorem 3.3]{FHV13}: Given $\eps>0$ and a ${\cal C}^1$-curve $\varphi:[0,1] \rightarrow \R^d$, if
\begin{align}
\label{eq:JLL_gamma}
 \rho:= \max_{t \in [0,1]} \frac{\|\varphi'(t)\|}{\|\varphi(t)\|} < \infty, 
\end{align}
then there exists a matrix $M \in \mathbb R^{k\times d}$ for $k \sim \eps^{-2} \log(d\cdot \rho \cdot \eps^{-1})$ such that
\begin{equation}\label{contJL}
(1-\eps) \|\varphi(t)\| \leq \|M \varphi(t)\| \leq (1+\eps) \|\varphi(t)\|,
\end{equation}  for all $t \in [0,1]$. As already announced at the beginning of this section, we would like to use \eqref{contJL} for
\begin{align*}
\varphi(t):=x_i(t)-x_j(t) \text{ or } \varphi(t):=v_i(t)-v_j(t)
\end{align*}
being $(x_i(t),v_i(t))$ the trajectory of the $i$-th agent in \eqref{eq:freesystemhi}. Unfortunately, \eqref{eq:JLL_gamma} does not hold in this case even if we assume that $\|x_i(0)-x_j(0)\|\geq c>0$ for all $i \not = j$: Let us consider, for instance, Example 1 from \cite{CFPT12} of a Cucker--Smale system of the type \eqref{eq:freesystemhi} with communication function \eqref{csmod} of two agents moving on the real line with positions and velocities at time $t$ given by $(x_1(t), v_1(t))$ and $(x_2(t), v_2(t))$. Let us assume that $\beta = 1$, $K = 2$ as well as $\sigma = 1$. We indicate by $x(t) = x_1(t) - x_2(t)$ the relative main state and by $v(t) = v_1(t) - v_2(t)$ the relative consensus parameter. The system can be reformulated in terms of relative variables
\begin{align*}
	  \left\{
		\begin{array}{l}
		\dot{x} = v, \\
		\dot{v} = - \frac{v}{1 + x^2}
		\end{array}
		\right.
\end{align*}
with initial conditions given by $x(0) = x_1(0) - x_2(0)$ and $v(0) = v_1(0) - v_2(0)$. Its solution can be characterized by integration by the following differential equation
\begin{align*}
x'(t)=v(t)=-\arctan(x(t)) + \arctan(x(0))+v(0).
\end{align*} 
Now, if $x(0)<0$ and $v(0)+\arctan(x(0))=:c(0)>0$, then $v(t)>c(0)$ as long as $x(t)<0$. Hence there has to be a $T>0$ with $x(T)=0$ and $v(T)=c(0)$. Thus $\eqref{eq:JLL_gamma}$ is violated for $\varphi(t) = x(t)$.

 Let us stress that \eqref{eq:JLL_gamma} is a necessary condition for \eqref{contJL} to hold (see \cite[Remark 1]{FHV13}). This motivates the relaxation of the strong Johnson--Lindenstrauss property to its weak version in Definition \ref{def:JLL_property}. Hence we prove a result based on the more general weak Johnson--Lindenstrauss property which will be sufficient for us in the following.
\end{Remark}

In the rest of the paper, given a Lipschitz function $\varphi: [a,b] \rightarrow \R^d$, we indicate with $L_{\varphi}(a,b)$ its Lipschitz constant on $[a,b]$, i.e.,
\begin{align*}
L_{\varphi}(a,b) := \sup_{\stackrel{t,s \in [a,b]}{t\neq s}} \frac{\| \varphi(t) - \varphi(s) \|}{|t - s|}.
\end{align*}

\begin{Lemma} 
\label{le:JLL_cont}
Let $\varphi:[0,1] \rightarrow \R^d$ be a Lipschitz function with Lipschitz constant $L_{\varphi}=L_{\varphi}(0,1)$,
let $\delta>0$, and $0<\eps< 1$. Let $k$ be such that a matrix $M \in \R^{k\times d}$ - stochastically generated as in (S2) or (S3) of Remark \ref{rm:numberofpoints}  - satisfies the (strong) Johnson--Lindenstrauss property of parameter $\tilde \eps=\eps/2$ at $\cal{N}$ arbitrary points with some (high) probability, where
\begin{align}
\label{eq:JLL_dN2}
\cal{N} \geq 4 \cdot \frac{L_{\varphi}  \cdot (\sqrt{d}+2)}{\delta \eps}.
\end{align}
Then the matrix $M$ fulfills the weak Johnson--Lindenstrauss property of parameters $\eps$ and $\delta$ at $\varphi(t)$ for every $t \in [0,1]$ with the same high probability, i.e.,
either 
\begin{align*} 
(1-\eps) \|\varphi(t)\| \leq \|M\varphi(t)\| \leq (1+\eps) \|\varphi(t)\|
\end{align*} 
or
\begin{align*} 
\|\varphi(t)\| \leq \delta \text{ and } \|M\varphi(t)\| \leq \delta
\end{align*}
holds for all $t \in [0,1]$.
\end{Lemma}
\begin{proof}
We shall adapt the arguments from the proof of \cite[Theorem 3.3]{FHV13}: Let $t_i:=i/{\cal N}$ for $i=0,\ldots,{\cal N}-1$ and assume that $M:\R^d \rightarrow \R^k$ fulfills the (strong) Johnson--Lindenstrauss property with parameter $\tilde{\eps}=\eps/2$ at the points $\{\varphi(t_i)\}_{i = 0}^{\cal{N}-1}$, i.e., we have
\begin{align*}
(1-\tilde{\eps}) \|\varphi(t_i)\| \leq \|M\varphi(t_i)\| \leq (1+\tilde{\eps}) \|\varphi(t_i)\|
\end{align*}
for all $i \in \{0,\ldots,{\cal N} - 1\}$. Furthermore, we may assume $1 \leq \|M\| \leq \sqrt{d}$, see (S2) and (S3) of Remark \ref{rm:numberofpoints}.

Let $t \in [0,1]$ and choose $j \in \{0,\ldots,\cal{N}-1\}$ such that $t \in [t_j,t_{j+1}]$. Let us at first assume
\begin{align*}
\|\varphi(t_j)\| \leq \delta/2.
\end{align*}
Since $\eps \in (0,1)$, by \eqref{eq:JLL_dN2} we have that
\begin{align*}
\cal{N} \geq 4 \cdot \frac{L_{\varphi}  \cdot \sqrt{d}}{\delta}.
\end{align*}
Using this latter inequality and the Lipschitz continuity of $\varphi$ we obtain
\begin{align*}
\|\varphi(t)\| &\leq \|\varphi(t)-\varphi(t_j)\| + \|\varphi(t_j)\| \\
&\leq L_{\varphi}/{\cal N} + \delta/2 \\
&\leq \delta
\end{align*}
and also
\begin{align*}
\|M\varphi(t)\| &\leq \|M\| \|\varphi(t)-\varphi(t_j)\| + \|M\varphi(t_j)\| \\
&\leq \sqrt{d} \cdot L_{\varphi}/{\cal N} + (1+\eps')\|\varphi(t_j)\| \\
&\leq \delta/4 + 3/2 \cdot \delta/2 \\
&\leq \delta.
\end{align*}
Let us now assume 
\begin{align*}
\|\varphi(t_j)\| >  \delta/2.
\end{align*}
Using again the Lipschitz continuity of $\varphi$ we obtain the estimate
\begin{align*}
\|\varphi(t)-\varphi(t_j)\| &\leq L_{\varphi}/{\cal N} \\
&\leq L_{\varphi}/{\cal N}  \cdot \frac{2\|\varphi(t_j)\|}{\delta} \\
&\leq  \frac{\delta \eps}{4(\sqrt{d}+2)}  \cdot \frac{2\|\varphi(t_j)\|}{\delta}  \\
&\leq \frac{\|\varphi(t_j)\| \cdot (\eps - \tilde{\eps})}{\|M\|+1+\eps},
\end{align*}
where in the last inequality we used that
\begin{align*}
\|M\|+1+\eps \leq \sqrt{d}+2.
\end{align*}
This estimate of the distance $\|\varphi(t)-\varphi(t_j)\|$ and the (strong) Johnson--Lindenstrauss property at $\varphi(t_j)$ enable us to extend the (strong) Johnson--Lindenstrauss property at $\varphi(t)$ as well, as a direct application of \cite[Lemma 3.2]{FHV13}, i.e., 
\begin{align*}
(1-\eps) \|\varphi(t)\| \leq \|M\varphi(t)\| \leq (1+\eps) \|\varphi(t)\|.
\end{align*}
Both cases together show the (weak) Johnson--Lindenstrauss property  at $\varphi(t)$ for every $t \in [0,1]$.
\end{proof}


We show in the following lemma that the mean-square norm and the relative order of the magnitudes of points in a cloud in high dimension are nearly preserved when projected in lower dimension by a weak Johnson--Lindenstrauss embedding.

\begin{Lemma}
\label{le:CSCS_technical}
Let $a_1, \ldots, a_N \in \R^d, b_1,\ldots,b_N \in \R^k$ and $M \in \R^{k \times d}$ such that there is  $\Delta> 0$ with the following properties:
\begin{itemize}
\item[(i)] The matrix $M$ fulfills the weak Johnson--Lindenstrauss property with $\eps=1/2$  and $\delta=\Delta$ for the points $a_i$, i.e.,
either 
\begin{align}
\label{eq:CSCS_JLL}
1/2 \cdot \|a_i\| \leq \|Ma_i\| \leq 3/2 \cdot  \|a_i\|,
\end{align} 
or
\begin{align}
\label{eq:CSCS_JLL_2}
\|a_i\| \leq \Delta \text{ and } \|Ma_i\| \leq \Delta,
\end{align}
for all $i \in 1,\ldots,N$.
\item[(ii)] We have the following approximation bound
\begin{align*}
\|Ma_i-b_i\| \leq \Delta,
\end{align*}
for all $i \in 1,\ldots,N$.
\end{itemize}
Let $\maxindex$ be the smallest index such that $\|b_{\maxindex}\| \geq \|b_j\|$ for all $j = 1, \ldots, N$ and let
\begin{align*}
A := \frac{1}{N} \sum^N_{j = 1} \|a_j \|^2 \text{ and } B := \frac{1}{N} \sum^N_{j = 1} \|b_j \|^2.
\end{align*}
If $\sqrt{B} \geq 2 \Delta$, then, for $c=1/\sqrt{289}$, it holds
\begin{align*}
 \|a_{\maxindex}\| \geq \|b_{\maxindex}\|/4, \quad \|a_{\maxindex}\| \geq c \sqrt{A}, \quad  \text{ and } \quad B \leq 16 NA.
\end{align*}
If $\sqrt{B} \leq 2 \Delta$, then, for $C=\sqrt{72}$, it holds
\begin{align*}
\sqrt{A} \leq C \Delta.
\end{align*}
\end{Lemma}
\begin{proof} First suppose that $\sqrt{B} \geq 2 \Delta$: since $\|b_{\maxindex}\|$ is maximal, we have $\|b_{\maxindex}\|\geq \sqrt{B} \geq 2\Delta$. By (ii) it holds $\|Ma_{\maxindex}\| \geq \|b_{\maxindex}\|-\Delta \geq 2 \Delta - \Delta \geq \Delta$ and hence using \eqref{eq:CSCS_JLL} we get
\begin{align*}
\|a_{\maxindex}\| \geq  \|Ma_{\maxindex}\|/2 \geq (\|b_{\maxindex}\|-\Delta)/2 \geq \|b_{\maxindex}\|/4.
\end{align*}
This shows the first estimate of the first part of the lemma. Let us address the second estimate. Let $j \in \{1,\ldots,N\}$ for $j \neq {\maxindex}$. If 
$\|b_j\| \geq 2 \Delta$, then, using the same argument as above, we have $\|Ma_j\|\geq \Delta$ and thus by \eqref{eq:CSCS_JLL}  we get
\begin{align}
\label{eq:CSCS_est_1}\
\|a_j\| \leq 2\|Ma_j\| \leq 2(\|b_j\|+\Delta)\leq 2 \cdot 3/2 \cdot \|b_j\| = 3\|b_j\|.
\end{align}
On the other hand, if $\|b_j\| <2 \Delta$, then $\|Ma_j\| \leq 3 \Delta$. Then either \eqref{eq:CSCS_JLL} holds and we have
\begin{align}
\label{eq:CSCS_est_2}
\|a_j\|\leq 2 \|Ma_j\| \leq 6 \Delta,
\end{align}
or \eqref{eq:CSCS_JLL_2} holds and automatically  $\|a_j\|\leq \Delta$. Now we can estimate the {mean-square norm} $A$. We obtain
\begin{align*}
NA= \sum^N_{j = 1} \|a_j \|^2 = \|a_{\maxindex}\|^2 + \sum_{j \in A_1} \|a_j\|^2+\sum_{j \in A_2} \|a_j\|^2,
\end{align*}
where $A_1$ is the index set of all $j \in \{1,\ldots,N\}\setminus\{\maxindex\}$ such that $\|b_j\| \geq 2 \Delta$ and  $A_2$ is the index set of all $j \in \{1,\ldots,N\}$ for which $\|b_j\| < 2 \Delta$. Using \eqref{eq:CSCS_est_1} and \eqref{eq:CSCS_est_2} we {obtain}
\begin{align*}
NA & \leq \|a_{\maxindex}\|^2 + {9}  \sum_{j \in A_1} \|b_j\|^2 + |A_2| \cdot 36 \Delta^2 \\
  & \leq \|a_{\maxindex}\|^2 + {9}  NB + 9 N \|b_{\maxindex}\|^2 \\
  & \leq \|a_{\maxindex}\|^2 + N  \|a_{\maxindex}\|^2 (9\cdot 16 + 9\cdot 16) \\
  & \leq 289 \|a_{\maxindex}\|^2 N
\end{align*}
using the maximality of $\|b_{\maxindex}\|$ and the first part of the lemma. Furthermore, we have
\begin{align*}
NB&= \sum_{j=1}^N \|b_j\|^2 \leq N\|b_{\maxindex}\|^2\leq 16N\|a_{\maxindex}\|^2 \leq 16N^2A.
\end{align*}
Hence
\begin{align*}
B \leq 16 NA.
\end{align*}
Let now $\sqrt{B} \leq 2 \Delta$. We can argue in the same way as for the second estimate of the first part: If $\|b_j\| \geq 2\Delta$, then as in \eqref{eq:CSCS_est_1}
\begin{align*}
\|a_j\| \leq {3}  \|b_j\|.
\end{align*}
If $\|b_j\| \leq 2\Delta$, then by \eqref{eq:CSCS_est_2} and the arguments thereafter we get
\begin{align*}
\|a_j\| \leq 6 \Delta. 
\end{align*}
Putting both estimates together and using the notation $\tilde{A}_1$ for the index set of all $j \in \{1,\ldots,N\}$  {such that} $\|b_j\| \geq 2 \Delta$ as well as  $\tilde{A}_2$ for the index set of all $j \in \{1,\ldots,N\}$ {such that} $\|b_j\| < 2 \Delta$ {yield}
\begin{align*}
NA&= \sum^N_{j = 1} \|a_j \|^2 = \sum_{j \in \tilde{A}_1} \|a_j\|^2+\sum_{j \in \tilde{A}_{2}} \|a_j\|^2 \\
&\leq 9\sum_{j \in \tilde{A}_1} \|b_j\|^2 + |\tilde{A}_2| \cdot 36 \Delta^2 \\
&\leq 9 NB + 36 N\Delta^2   \\
&\leq N({9} B + 36 \Delta^2 ) \\
&\leq N({36}\Delta^2 + 36 \Delta^2) \\
&\leq 72N \Delta^2.
\end{align*}
Taking the square root on both sides finishes the proof.
\end{proof}

\section{Dimension reduction of the Cucker--Smale model without control} \label{sec:withoutControl}
In this section we consider the projection of the Cucker--Smale system without control. We compare two quantities: First, we calculate the trajectory of the high-dimensional Cucker--Smale system and then project the agents' parameters by $M \in \R^{k\times d}$. Second, we project the initial configurations to dimension $k$  by applications of $M$. Then we compute from these initial values the trajectories of the corresponding low-dimensional Cucker--Smale system. What we shall do in the upcoming Theorem \ref{th:Reduction_uncontrolled} is to give a precise bound from above to the distance between the the two $k$-dimensional trajectories, computed as described above.

More formally, given $M \in \R^{k\times d}$ (where $k \leq d$) and initial conditions $(x(0),v(0))$ for \eqref{eq:freesystemhi}, we indicate with $(y(t),w(t))$ the  solution of the $\R^k$-projected Cucker--Smale system 
\begin{align*}
\left \{\begin{array}{lll}
\dot{y}_{i}(t) & = & w_{i}(t), \\	
\dot{w}_{i}(t) & = & \frac{1}{N}\underset{j=1}{\overset{N}{\sum}}a\left(\left\Vert y_{i}(t)-y_{j}(t)\right\Vert\right)\left(w_{j}(t)-w_{i}(t)\right), \quad i=1,\ldots,N
\end{array}
\right.
\end{align*}
with initial conditions $y(0)=\left(Mx_1(0),\ldots,Mx_N(0)\right) \in (\R^k)^N$ and $w(0)=\left(Mv_1(0),\ldots,Mv_N(0)\right) \in (\R^k)^N$. 

We introduce the low-dimensional analogues of $X$ and $V$ by
\begin{align}
\label{eq:Wintro}
Y(t):= B(y(t), y(t)), \qquad W(t) := B(w(t), w(t)).
\end{align} 
Here the bilinear form $B$ is intended to act on $\mathbb R^k$ instead of $\mathbb R^d$, but with the same meaning
of the symbol as before.

\begin{Remark}
By Lemma \ref{le:HaHaKim} we know that $V$ and $W$ are decreasing. Hence for all $i,j\in\{1,\ldots,N\}$
\begin{align*}
\|w_i(t)-w_j(t)\|^2 &\leq 2\left(\|w_i(t)-\overline{w}(t)\|^2+ \|w_j(t)-\overline{w}(t)\|^2\right) \\
					&\leq 2 \sum_{\ell=1}^{N} \|w_{\ell}(t)-\overline{w}(t)\|^2 \\
					&\leq 2NW(t) \\
					&\leq 2NW(0), 
\end{align*}
thus
\begin{align*}
\|w_i(t)-w_j(t)\| \leq \sqrt{2NW(0)}.
\end{align*}
An analogous estimate holds for $V$ and $v$. Furthermore, we have
\begin{align*}
\|x_i(0)-x_j(0)\| \leq \sqrt{2NX(0)}.
\end{align*}
\end{Remark}

\begin{Theorem}
\label{th:Reduction_uncontrolled}
Let $\delta>0$, let $\eps \in (0,1)$, and let $M \in \R^{k \times d}$ be a matrix with the weak Johnson--Lindenstrauss property of parameters $\eps$ and $\delta$ at the vectors $x_i(t)-x_j(t)$ for all $t \in [0,T]$ and all $i,j \in \{1,\ldots,N\}$.

Define the following errors:
\begin{align*}
	e^x_i(t) &= \| y_i(t) - Mx_i(t) \|, \quad\phantom{XX'''} e^v_i(t) = \| w_i(t) - Mv_i(t) \|, \\
\cal{E}^x (t) &= \max_{i = 1, \ldots, N} e^x_i(t), \quad\phantom{XXXXX} \cal{E}^v (t) = \max_{i = 1, \ldots, N} e^v_i(t), \\
	\cal{E}_2^x (t) &= \left(\frac{1}{N}\sum_{i=1}^N (e^x_i(t))^2\right)^{1/2}, \quad  \cal{E}_2^v(t)=\left(\frac{1}{N}\sum_{i=1}^N (e^v_i(t))^2\right)^{1/2}.
\end{align*}

  Furthermore, let $L_a$ be the Lipschitz constant of the function $a$, set $K_1:= L_a \sqrt{NW(0)} \sqrt{2X(0)}$, $K_2:=2L_a\sqrt{NW(0)}$, and $K_3=1/2\cdot L_a\sqrt{NW(0)}\sqrt{2V(0)}$. Then for all $t \in [0,T]$ the estimates
\begin{align*}
\cal{E}^x(t)+ \cal{E}^v(t) \leq \sqrt{N} ((\eps K_1+\delta K_2) t+ \eps K_3 t^2) \cdot e^{t \left\| \cal{K} \right\|_{\ell_1 \rightarrow \ell_1} }
\end{align*}
 and 
\begin{align*}
\cal{E}_2^x(t)+ \cal{E}_2^v(t) \leq ((\eps K_1+\delta K_2) t+ \eps K_3 t^2) \cdot e^{t \left\| \cal{K} \right\|_{\ell_1 \rightarrow \ell_1} },
\end{align*}
hold, where
\begin{align*}
	  \cal{K}=\left[
		\begin{array}{cc}
		2 a(0) & 2L_a\sqrt{NW(0)} \\
		1 & 0
		\end{array}
		\right].
	\end{align*}
Moreover
\begin{align*}
\cal{E}^v(t) & \leq \sqrt{N} \min \left\{ ((\eps K_1+\delta K_2) t+ \eps K_3 t^2) \cdot e^{t \left\| \cal{K} \right\|_{\ell_1 \rightarrow \ell_1}}, \left(\|M\|\sqrt{V(t)} + \sqrt{W(t)}\right) \right\}.
\end{align*}
\end{Theorem}
\begin{proof}
We estimate the decay of $\cal{E}_2^x (t)$ and $\cal{E}_2^v (t)$ in order to use Gronwall's Lemma. For the following estimates we may assume that -- without loss of generality -- $e^v_i(t)\neq 0$ for $t\in[0,T]$ and for every $i=1,\ldots,N$: if this is not the case, either $e^v_i \equiv 0$ in a neighborhood of $t$ or, by continuity, the estimates will also hold true at $t$. Hence we may assume that $e^v_i$ is differentiable at $t \in [0,T]$. By Cauchy--Schwarz inequality it holds
\begin{align*}
	\frac{d}{dt}  e^v_i(t) & = \frac{\langle w_i(t) - Mv_i(t), \frac{d}{dt}(w_i(t) - Mv_i(t)) \rangle}{\|w_i(t) - Mv_i(t) \|} \\
	& \leq  \left\| \dot{w}_i(t) - M\dot{v}_i(t) \right\| \\
	& \leq  \frac{1}{N} \sum^N_{j = 1} \left\| a(\|y_i(t) - y_j(t) \|)(w_j(t) - w_i(t)) - a(\|x_i(t) - x_j(t) \|)(Mv_j(t) - Mv_i(t)) \right\|.
\end{align*}
Using triangle inequality, the Lipschitz continuity of $a$ and its monotonicity, we obtain 
\begin{align}
\label{eq:jtriangleo}
\begin{split}
	\frac{d}{dt}  e^v_i(t) & \leq \frac{1}{N} \sum^N_{j = 1} \Big[ \left|a(\|y_i(t) - y_j(t) \|) - a(\|x_i(t) - x_j(t) \|) \right| \|w_j(t) - w_i(t)\| +  \\
	&  \quad + a(\|x_i(t) - x_j(t) \|) \left\|(w_j(t) - w_i(t)) - (Mv_j(t) - Mv_i(t)) \right\| \Big] \\
	& \leq \frac{1}{N} \sum^N_{j = 1}  \Big[ L_a  \left|\|y_i(t) - y_j(t) \| - \|x_i(t) - x_j(t) \| \right| \|w_j(t) - w_i(t)\| +  \\
	&  \quad + a(0) \left\|(w_j(t) - w_i(t)) - (Mv_j(t) - Mv_i(t)) \right\| \Big].
\end{split}
\end{align}
We now estimate the derivative of $\cal{E}_2^v$. First of all, again by Cauchy--Schwarz inequality it follows
\begin{align}
\label{eq:Norminsert}
\begin{split}
	\frac{d}{dt}\cal{E}^v_2(t) &=	\frac{d}{dt} \left(\frac{1}{N} \sum_{i=1}^N\left\|(w_i(t) - Mv_i(t)) \right\|^2\right)^{1/2} \\
	&= \frac{1}{\left(\frac{1}{N} \sum_{i=1}^N\left\|(w_i(t) - Mv_i(t)) \right\|^2\right)^{1/2}} \cdot \left(\frac{1}{N} \sum_{i=1}^N\left\|(w_i(t) - Mv_i(t)) \right\| \cdot \frac{d}{dt} \left\|(w_i(t) - Mv_i(t)) \right\|\right)^{1/2}  \\
	&\leq \left(\frac{1}{N} \sum_{i=1}^N \left(\frac{d}{dt} \left\|(w_i(t) - Mv_i(t)) \right\|\right)^2\right)^{1/2} \\
&= \left(\frac{1}{N}\sum_{i=1}^N\left(\frac{d}{dt}  e^v_i(t)\right)^2\right)^{1/2}.
\end{split}
	\end{align}
If we insert \eqref{eq:jtriangleo} into the last inequality and we use triangle as well as Cauchy--Schwarz inequality in sequence, we get
\begin{align*}
\frac{d}{dt}\cal{E}_2^v(t)& \leq \left(\frac{1}{N}\sum_{i=1}^N\left(\frac{1}{N} \sum^N_{j = 1} L_a\left|\|y_i(t) - y_j(t) \| - \|x_i(t) - x_j(t) \right\| \|w_j(t) - w_i(t)\|\right)^{2}\right)^{1/2}  \\
	&  \quad + \left(\frac{1}{N}\sum_{i=1}^N\left(\frac{1}{N} \sum^N_{j = 1} a(0)\left\|(w_j(t) - w_i(t)) - (Mv_j(t) - Mv_i(t)) \right\|\right)^{2}\right)^{1/2} \\
	&\leq L_a\left(\frac{1}{N} \sum_{i=1}^N \left(\frac{1}{N}\sum_{j=1}^N \left|\|y_i(t) - y_j(t) \|-\|x_i(t) - x_j(t) \| \right|^2\right)  \cdot \left(\frac{1}{N} \sum_{j=1}^N\|w_j(t) - w_i(t)\|^2 \right)\right)^{1/2}\\
	&  \quad +a(0) \left(\frac{1}{N^2} \sum_{i,j=1}^N\left\|(w_j(t) - w_i(t)) - (Mv_j(t) - Mv_i(t)) \right\|^2\right)^{1/2} \\
	& \leq L_a \left(\frac{1}{N^2} \sum_{i,j=1}^N \left|\|y_i(t) - y_j(t) \| - \|x_i(t) - x_j(t) \|\right|^2 \right)^{1/2} \cdot \max_{i=1,\ldots,N} \left(\frac{1}{N} \sum_{j=1}^N\|w_j(t) - w_i(t)\|^2 \right)^{1/2}  \\
& \quad +a(0) \left(\frac{1}{N^2} \sum_{i,j=1}^N\left\|(w_j(t) - w_i(t)) - (Mv_j(t) - Mv_i(t)) \right\|^2\right)^{1/2}.
\end{align*}
Let us now estimate now the first term of the sum. It holds
\begin{align}
\label{eq:xy_estimate}
\begin{split}
 \left|\|y_i - y_j \| - \|x_i - x_j \| \right| & =  \left|\|y_i - y_j \| - \|Mx_i - Mx_j\| +  \|Mx_i - Mx_j\| - \|x_i - x_j \| \right| \\
& \leq  \left|\|y_i - y_j \| - \|Mx_i - Mx_j\| \right| +  \left| \|Mx_i - Mx_j\| - \|x_i - x_j \| \right| \\
& \leq  \|y_i - Mx_i\|+  \|y_j - Mx_j\| + \left| \| Mx_i - Mx_j \| - \|x_i - x_j \| \right| \\
& \leq e_i^x+e_j^x + \left| \| Mx_i - Mx_j \| - \|x_i - x_j \| \right|
\end{split}
\end{align}
and for all $i=1,\ldots,N$ 
\begin{align*}
\frac{1}{N} \sum_{j=1}^N\|w_j(t) - w_i(t)\|^2 \leq \frac{1}{N} \sum_{i=1}^N\sum_{j=1}^{i-1}\|w_j(t) - w_i(t)\|^2=\frac{1}{2N} \sum_{i,j=1}^N\|w_j(t) - w_i(t)\|^2=NW(t).
\end{align*}
Furthermore, for the second sum we have
\begin{align*} 
&\left(\frac{1}{N^2} \sum_{i,j=1}^N\left\|(w_j(t) - w_i(t)) - (Mv_j(t) - Mv_i(t)) \right\|^2\right)^{1/2} \\
&\phantom{XXXXXXX}\leq \left(\frac{1}{N} \sum_{i=1}^N\left\|(w_i(t) - Mv_i(t)) \right\|^2\right)^{1/2} + \left(\frac{1}{N} \sum_{j=1}^N\left\|(w_j(t) - Mv_j(t)) \right\|^2\right)^{1/2} \\
&\phantom{XXXXXXX}\leq 2 \cal{E}_2^v(t).
\end{align*}
Hence our computation yields
\begin{align*}
	\frac{d}{dt} \cal{E}_2^v(t) & \leq L_a\sqrt{NW(t)} \left(\frac{1}{N^2} \sum_{i,j=1}^N  \left| \| Mx_i - Mx_j \| - \|x_i - x_j \| \right|^2 \right)^{1/2} +  2 L_a\sqrt{NW(t)} \cal{E}_2^x(t) + 2 a(0) \cal{E}_2^v(t). 
\end{align*}
Now we apply the assumptions on the matrix $M$: For every $i,j \in \{1,\ldots,N\}$ and $t \in [0,T]$ either \eqref{eq:JLL_xquasi} holds, and then
\begin{align}
\label{eq:Mx_error}
\left| \| Mx_i(t) - Mx_j(t) \| - \|x_i(t) - x_j(t) \| \right| &\leq \eps \|x_i(t) - x_j(t) \| \leq \eps \left( \|x_i(0) - x_j(0)\|+  \int_{0}^t \|v_i(s)-v_j(s)\| \ ds \right),
\end{align}
or \eqref{eq:JLL_xsmall} holds, and then
\begin{align*}
\left| \| Mx_i(t) - Mx_j(t) \| - \|x_i(t) - x_j(t) \| \right| \leq 2 \delta,
\end{align*}
so we always have
\begin{align*}
\left| \| Mx_i(t) - Mx_j(t) \| - \|x_i(t) - x_j(t) \| \right| \leq \eps \left( \|x_i(0) - x_j(0)\|+  \int_{0}^t \|v_i(s)-v_j(s)\| \ ds \right) + 2 \delta.
\end{align*}
Using (vector-valued) Minkowski inequality and observing that, by Lemma \ref{le:HaHaKim}, $V$ and $W$ are decreasing, we derive
\begin{align} 
	\frac{d}{dt} \cal{E}_2^v(t)  & \leq \eps L_a\sqrt{NW(t)}  \left(\left(\frac{1}{N^2} \sum_{i,j=1}^N \|x_i(0) - x_j(0)\|^2\right)^{1/2} +\left(\frac{1}{N^2} \sum_{i,j=1}^N  \left(\int_{0}^t \|v_i(s)-v_j(s)\| \ ds \right)^2\right)^{1/2}\right) \notag \\
	&\quad +   2 L_a\sqrt{NW(t)} (\delta+ \cal{E}_2^x(t)) + 2 a(0) \cal{E}_2^v(t) \notag \\
	&\leq  \eps L_a \sqrt{NW(0)} \left(\sqrt{2X(0)} +   \int_0^t \left(\frac{1}{N^2} \sum_{i,j=1}^N   \|v_i(s)-v_j(s)\|^2 \right)^{1/2} \ ds\right) \notag \\
	&\quad +   2 L_a\sqrt{NW(0)} (\delta+ \cal{E}_2^x(t)) + 2 a(0) \cal{E}_2^v(t) \notag \\
	&\leq   \eps L_a \sqrt{NW(0)} \left(\sqrt{2X(0)} +  \int_0^t \sqrt{2V(s)} \ ds \right)+   2 L_a\sqrt{NW(0)} (\delta+ \cal{E}_2^x(t)) + 2 a(0) \cal{E}_2^v(t), \label{eq:CSU_derivativeerrorv_partial} \\
	&\leq  \eps  L_a \sqrt{NW(0)} \left(\sqrt{2X(0)} + t\sqrt{2V(0)}\right)+ 2\delta L_a\sqrt{NW(0)} + 2 L_a\sqrt{NW(0)}\cal{E}_2^x(t) + 2 a(0) \cal{E}_2^v(t). \label{eq:CSU_derivativeerrorv}
\end{align}
On the other hand, in the same way as in \eqref{eq:Norminsert}, we obtain
\begin{align*}
	\frac{d}{dt}\cal{E}^x_2(t) 
	&\leq \left(\frac{1}{N} \sum_{i=1}^N \left(\frac{d}{dt} \left\|(y_i(t) - Mx_i(t)) \right\|\right)^2\right)^{1/2} \\
	&\leq \left(\frac{1}{N} \sum_{i=1}^N \left\|\frac{d}{dt}(y_i(t) - Mx_i(t)) \right\|^2\right)^{1/2} \\
	&= \left(\frac{1}{N} \sum_{i=1}^N \left\|w_i(t) - Mv_i(t) \right\|^2\right)^{1/2} \\
       &=\cal{E}^v_2(t).
	\end{align*}
Let $K_1=L_a\sqrt{NW(0)}\sqrt{2X(0)}$, $K_2=2 L_a\sqrt{NW(0)}$, $K_3=1/2 \cdot L_a\sqrt{NW(0)}\sqrt{2V(0)}$, and 
\begin{align*}
	  \cal{K}=\left[
		\begin{array}{cc}
		2 a(0) & 2 L_a \sqrt{NW(0)}\\
		1 & 0
		\end{array}
		\right]
	\end{align*}
be as in the statement of the theorem. Then, rearranging the previous calculations in vector form and integrating from $0$ to $t$,  we get {the inequality} 
	\begin{align*}
		\left[
		\begin{array}{c}
		\cal{E}_2^v(t) \\
		\cal{E}_2^x(t)
		\end{array}
		\right] \leq
		\left[
		\begin{array}{c}
		\cal{E}_2^v(0) + (\eps K_1+\delta K_2) t+ \eps K_3 t^2\\
		\cal{E}_2^x(0) 
		\end{array}
		\right] + \int_{0}^t \cal{K} \cdot
		\left[
		\begin{array}{c}
		\cal{E}^v(s) \\
		\cal{E}^x(s)
		\end{array}
		\right] 
		ds,
	\end{align*}
	Notice that
	\begin{align*}
	  \left\| \cal{K} \right\|_{\ell_1 \rightarrow \ell_1} = \max \left( 2a(0) + 1, 2 L_a \sqrt{NW(0)} \right) \geq 1.
	\end{align*}
	Now we apply the $\ell_1$-norm to the inequality and we use Gronwall's Lemma,  
see Lemma \ref{le:class_gronwall}, to deduce
	\begin{align*}
		\left\|
		\begin{array}{c}
		\cal{E}_2^v(t) \\
		\cal{E}_2^x(t)
		\end{array}
		\right\|_{\ell_1} \leq
		\left\|
		\begin{array}{c}
		\cal{E}_2^v(0) + (\eps K_1+\delta K_2) t+ \eps K_3 t^2 \\
		\cal{E}_2^x(0) 
		\end{array}
		\right\|_{\ell_1} \cdot e^{t \left\| \cal{K} \right\|_{\ell_1 \rightarrow \ell_1} },
	\end{align*}
thus
	\begin{align}
	\label{eq:start_error}
	\begin{split}
		\cal{E}_2^v(t) + \cal{E}_2^x(t) &\leq \left[\cal{E}_2^v(0) + \cal{E}_2^x(0) + (\eps K_1+\delta K_2) t+ \eps K_3 t^2\right] \cdot e^{t \left\| \cal{K} \right\|_{\ell_1 \rightarrow \ell_1} }\\
		&=((\eps K_1+\delta K_2) t+ \eps K_3 t^2) \cdot e^{t \left\| \cal{K} \right\|_{\ell_1 \rightarrow \ell_1} },
	\end{split}
	\end{align}
	since $\cal{E}^v(0) = \cal{E}^x(0) = 0$ by definition.
	
Moreover, since {$\overline{v}(t) = \overline{v}(0)$ and $\overline{w}(t) = \overline{w}(0)$ for every $t \geq 0,$} it holds
\begin{align*}
 \|Mv_i(t) - w_i(t)\| & = \|Mv_i(t) - \overline{w}(0) + \overline{w}(0) - w_i(t)\| \\
& \leq \|Mv_i(t) - M\overline{v}(0)\| + \|\overline{w}(0) - w_i(t)\| \\
& \leq \|M\|\|v_i(t) - \overline{v}(t)\| + \|\overline{w}(t) - w_i(t)\| 
\end{align*}
and hence we have
\begin{align*}
\cal{E}_2^v(t) \leq \|M\| \left({\frac{1}{N}}\sum_{i=1}^N \|v_i(t) - \overline{v}(t)\|^2\right)^{1/2} + \left({\frac{1}{N}}\sum_{i=1}^N \|w_i(t) - \overline{w}(t)\|^2\right)^{1/2}.
\end{align*}
Together with \eqref{eq:start_error}, we deduce the upper bound
\begin{align*}
\cal{E}_2^v(t) & \leq \min \left\{ ((\eps K_1+\delta K_2) t+ \eps K_3 t^2) \cdot e^{t \left\| \cal{K} \right\|_{\ell_1 \rightarrow \ell_1}}, \left(\|M\|\sqrt{V(t)} + \sqrt{W(t)}\right) \right\}.
\end{align*}
Using the trivial estimate of the $\ell_{\infty}$-norm by the $\ell_2$-norm we conclude as well the estimate
\begin{align*}
\cal{E}^v(t) & \leq \sqrt{N} \min \left\{ ((\eps K_1+\delta K_2) t+ \eps K_3 t^2) \cdot e^{t \left\| \cal{K} \right\|_{\ell_1 \rightarrow \ell_1}}, \left(\|M\|\sqrt{V(t)} + \sqrt{W(t)}\right) \right\}.
\end{align*}
and
\begin{align*}
\cal{E}^v(t) + \cal{E}^x(t)  \leq \sqrt{N}((\eps K_1+\delta K_2) t+ \eps K_3 t^2) \cdot e^{t \left\| \cal{K} \right\|_{\ell_1 \rightarrow \ell_1} }.
\end{align*}
\end{proof}

\begin{Remark}
\label{rem:Vreplace}
In the proof we used that $V$ and $W$ are decreasing. When we consider controlled systems below, we even have a better estimate on the integral of $V$. In particular we use the following: \textit{Assume additionally that
\begin{align*}
\int_0^t \sqrt{2V(s)} \ ds \leq \alpha \text{ for all } t\leq T,
\end{align*}
for a fixed $\alpha > 0$. Then for all $t\leq T$ we have 
\begin{align*}
\cal{E}^x(t)+ \cal{E}^v(t) \leq \sqrt{N} ((\eps (K_1+K_4)+\delta K_2) t) \cdot e^{t \left\| \cal{K} \right\|_{\ell_1 \rightarrow \ell_1} }
\end{align*}
with $K_1,K_2, \cal{K}$ as in Theorem \ref{th:Reduction_uncontrolled}, and $K_4= L_a\sqrt{NW(0)} \alpha$.}

{To verify the latter estimate, just consider the boundedness of $\int_0^t \sqrt{2V(s)} \ ds$ within the inequality \eqref{eq:CSU_derivativeerrorv_partial} in the proof of Theorem \ref{th:Reduction_uncontrolled}, and then proceed further as before.}
\end{Remark}

\begin{Remark}
\label{eq:numb_points}
Among the hypotheses of Theorem \ref{th:Reduction_uncontrolled}, we assumed the existence of a matrix $M \in \R^{k \times d}$ fulfilling the weak Johnson--Lindenstrauss property for all curves of the form $x_i(t) - x_j(t)$, where $i,j \in \{1,\ldots,N\}$ and $t \in [0,T]$. We show now that $M$ is such a matrix provided that it fulfills the (strong) Johnson--Lindenstrauss property for all the (finite) vectors of the form $x_i(t_m) - x_j(t_m)$, where $i,j \in \{1,\ldots,N\}$, $m = 0, \ldots, \lceil T\cdot \cal{N}'\rceil-1$, $t_m = m/\lceil T\cdot \cal{N}'\rceil$, for
\begin{align}\label{eq:cal_multi}
\cal{N}' \geq 4 \cdot \frac{\sqrt{2NV(0)} \cdot (\sqrt{d}+2)}{\delta \eps},
\end{align}
and that the target dimension $k$ is sufficiently large.

Indeed, we can adapt the proof of Lemma \ref{le:JLL_cont} in order to obtain a result valid \emph{simultaneously} for all the curves $\varphi_{ij}: [0,T] \rightarrow \R^d$ given by $\varphi_{ij}(t) = x_i(t) - x_j(t)$. For each of these curves we have the Lipschitz estimate
\begin{align} \label{eq:lipschitzX}
\frac{\|(x_i(t_1)-x_j(t_1))-(x_i(t_2)-x_j(t_2))\|}{|t_1-t_2|} \leq \sup_{t \in [0,T]} \|(x_i-x_j)'(t)\| = \sup_{t \in [0,T]} \|(v_i-v_j)(t)\| \leq \sqrt{2NV(0)},
\end{align}
thus $L_{\varphi_{ij}}(0,T) \leq \sqrt{2NV(0)}$. In order for the argument of the proof to work, for each curve $\varphi_{ij}$ we need $\cal{N}\cdot T$ points (where $\cal{N}$ is as in \eqref{eq:cal_multi}, and the factor $T$ is due to stretching the dynamics from a reference time domain $[0,1]$ to $[0,T]$) at which the (strong) Johnson--Lindenstrauss property must hold, bringing the total number of points $\cal{N}$ at which that property must be true to $\cal{N}'\cdot T\cdot N^2$. So it holds
\begin{align*}
\cal{N}  \sim \sqrt{NV(0)}  \cdot \frac{\sqrt{d}}{\delta\eps} \cdot T \cdot N^2.
\end{align*}
Thus, if $M$ is a $k \times d$ matrix fulfilling the (strong) Johnson--Lindenstrauss property of parameter $\eps$ at these $\cal{N}$ points, where $k \geq k_0$ with
\begin{align*}
k_0 \lesssim \log(\cal{N}) \cdot \eps^{-2} \lesssim \log\left( \sqrt{NV(0)}  \cdot \frac{\sqrt{d}}{\delta\eps} \cdot T \cdot N^2\right) \cdot \eps^{-2} \sim \left(\log(T\cdot N\cdot d\cdot V(0))+|\log{\delta\eps}|\right) \cdot \eps^{-2} ,
\end{align*}
then $M$ satisfies also the hypothesis of Theorem \ref{th:Reduction_uncontrolled} for any $\delta>0$.

\end{Remark}

\begin{Remark}
\label{rem:noDinvolved}
In the remark above we calculated the necessary minimal dimension $k_0$ for a matrix $M$ to sa\-tisfy the weak Johnson--Lindenstrauss property for all curves of the form $x_i(t) - x_j(t)$, where $i,j \in \{1,\ldots,N\}$ and $t \in [0,T]$. The dependency of $k_0$ on $N$ and $\eps$ is quite natural, but the dependency on the dimension $d$, even only logarithmically, is perhaps not desirable. But one can circumvent the dependence on the dimension using certain direct estimates within the proof of Theorem \ref{th:Reduction_uncontrolled}. In analogy to what we did before, take $t_m=m/ \cal{N}'$ with $m=0,\ldots, \lceil T\cdot \cal{N}'\rceil - 1$ and $\cal{N}'$ -- the number of sampling points -- is to be chosen large enough later on. Furthermore, we assume that the matrix $M$ fulfills the (strong) Johnson--Lindenstrauss property at $t_m$, i.e., we require that $M$ satisfies
\begin{align*}
\begin{split}
(1-\eps) \|x_i(t_m)-x_j(t_m)\| \leq \|M(x_i(t_m)-x_j(t_m))\| \leq (1+\eps) \|x_i(t_m)-x_j(t_m)\| \\
(1-\eps) \|v_i(t_m)-v_j(t_m)\| \leq \|M(v_i(t_m)-v_j(t_m))\| \leq (1+\eps) \|v_i(t_m)-v_j(t_m)\|
\end{split}
\end{align*} 
for all $i,j \in \{1,\ldots,N\}$ and $m=0,\ldots, \lceil T\cdot \cal{N}'\rceil -1$. Now, for any $t \in [0,T]$ choose $m \in \{0,\ldots, \lceil T\cdot \cal{N}'\rceil-1 \}$ such that $t \in [t_m,t_{m+1}]$. We start at the estimate \eqref{eq:xy_estimate}:
\begin{align}
\label{eq:xyM_estimate}
\begin{split}
&\left| \| x_i(t) - x_j(t) \| - \|y_i(t) - y_j(t) \| \right| \\
&\qquad\qquad\leq\left| \| x_i(t_m) - x_j(t_m) \| - \|y_i(t_m) - y_j(t_m) \| \right| + \left|\| x_i(t) - x_j(t) \|-\| x_i(t_m) - x_j(t_m) \| \right| \\
&\qquad\qquad\quad
+\left|\| y_i(t) - y_j(t) \|-\| y_i(t_m) - y_j(t_m) \| \right|  \\
&\qquad\qquad \leq \left| \| x_i(t_m) - x_j(t_m) \| - \|Mx_i(t_m) - Mx_j(t_m) \| \right|+ e_i^x(t_m)+e_j^x(t_m) + \frac{L_{x_i-x_j}}{\cal{N}'} +  \frac{L_{y_i-y_j}}{\cal{N}'},
\end{split}
\end{align}
where $L_{x_i-x_j}=L_{x_i-x_j}(0,T)$ and $L_{y_i-y_j}=L_{y_i-y_j}(0,T)$ are the Lipschitz constants of the functions $(x_i-x_j)(\cdot)$ and $(y_i-y_j)(\cdot)$ on $[0,T]$, respectively. Furthermore, using the (strong) Johnson--Lindenstrauss property of $M$ at $t_m$ we get the same estimates as in \eqref{eq:Mx_error}, only with $t_m$ instead of $t$:
\begin{align*}
\left| \| Mx_i(t_m) - Mx_j(t_m) \| - \|x_i(t_m) - x_j(t_m) \| \right| &\leq \eps \|x_i(t_m) - x_j(t_m) \| \\
&\leq \eps \left( \|x_i(0) - x_j(0)\|+  \int_{0}^{t_m} \|v_i(s)-v_j(s)\| \ ds \right).
\end{align*}
For the estimate of the last two terms in \eqref{eq:xyM_estimate} we choose $\cal{N}'$ large enough so that
\begin{align*}
\cal{N}' \sim \frac{\max\left(L_{x_i-x_j}, L_{y_i-y_j}\right)}{\delta}.
\end{align*}
Thus we arrive at
\begin{align*}
\left| \| x_i(t) - x_j(t) \| - \|y_i(t) - y_j(t) \| \right| \leq e_i^x(t_m)+e_j^x(t_m) + \eps \left( \|x_i(0) - x_j(0)\|+  \int_{0}^{t_m} \|v_i(s)-v_j(s)\| \ ds \right)+2\delta.
\end{align*}
Following the steps of the proof of Theorem \ref{th:Reduction_uncontrolled}, we can get an analogue of \eqref{eq:CSU_derivativeerrorv}:
\begin{align*}
\frac{d}{dt} \cal{E}_2^v(t) 	&\leq \eps  L_a \sqrt{NW(0)} \left(\sqrt{2X(0)} + t\sqrt{2V(0)}\right)+ 2 \delta L_a\sqrt{NW(0)} + 2 L_a\sqrt{NW(0)}\cal{E}_2^x(t_m) + 2 a(0) \cal{E}_2^v(t).
\end{align*}
So, the main difference is the replacement of $\cal{E}_2^x(t)$ with $\cal{E}_2^x(t_m)$ on the right-hand sides, with $t_m=m/ \cal{N}'$. At this point in the proof of Theorem \ref{th:Reduction_uncontrolled} we applied Gronwall's Lemma, see the estimates before \eqref{eq:start_error}. Now here we intend to use its discrete version, Lemma \ref{le:discr_gronwall}: let again $K_1=L_a\sqrt{NW(0)}\sqrt{2X(0)}$, $K_2=2 L_a\sqrt{NW(0)}$, and $K_3=1/2 \cdot L_a\sqrt{NW(0)}\sqrt{2V(0)}$. Integrating between $t_m$ and $t$ we get 
	\begin{align*}
		\left[
		\begin{array}{c}
		\cal{E}_2^v(t) \\
		\cal{E}_2^x(t)
		\end{array}
		\right] \leq
		\left[
		\begin{array}{c}
		\cal{E}_2^v(t_m) + (\eps K_1+\delta K_2) (t-t_m)+ \eps K_3 (t^2-t_m^2) + K_2 \cal{E}_2^x(t_m) (t-t_m)\\
		\cal{E}_2^x(t_m) 
		\end{array}
		\right] + \int_{t_m}^t \cal{K}' \cdot
		\left[
		\begin{array}{c}
		\cal{E}^v(s) \\
		\cal{E}^x(s)
		\end{array}
		\right] 
		ds,
	\end{align*}
	where
	\begin{align*}
	  \cal{K}'=\left[
		\begin{array}{cc}
		2 a(0) & 0\\
		1 & 0
		\end{array}
		\right].
	\end{align*}
Now applying the $\ell_1$-norm and Lemma \ref{le:discr_gronwall} we get
	\begin{align*}
		\left\|
		\begin{array}{c}
		\cal{E}_2^v(t) \\
		\cal{E}_2^x(t)
		\end{array}
		\right\|_{\ell_1} \leq
		\left\|
		\begin{array}{c}
		\cal{E}_2^v(0) + (\eps K_1+\delta K_2) t+ \eps K_3 t^2 \\
		\cal{E}_2^x(0) 
		\end{array}
		\right\|_{\ell_1} \cdot e^{t \left(\left\| \cal{K}' \right\|_{\ell_1 \rightarrow \ell_1} +K_2\right)}.
	\end{align*}
This is a slightly worse estimate than the original one of Theorem \ref{th:Reduction_uncontrolled} by a factor $2$ in the exponential, since 
\begin{align*}
	  \|\cal{K}\|_{\ell_1 \rightarrow \ell_1} &=\max\left(2a(0) + 1, 2 L_a \sqrt{NW(0)}\right) \\
	  &\leq 2 L_a \sqrt{NW(0)}+\left\|\left[
		\begin{array}{cc}
		2 a(0) & 0\\
		1 & 0
		\end{array}
		\right]\right\|_{\ell_1 \rightarrow \ell_1} \\
		&= K_2+\|\cal{K}'\|_{\ell_1 \rightarrow \ell_1}\\
		&\leq 2  \max\left(2a(0) + 1, 2 L_a \sqrt{NW(0)}\right) \\
		&= 2\|\cal{K}\|_{\ell_1 \rightarrow \ell_1}. 
\end{align*}
So eventually we obtain
	\begin{align*}
		\cal{E}_2^v(t) + \cal{E}_2^x(t) &\leq \left[\cal{E}_2^v(0) + \cal{E}_2^x(0) + (\eps K_1+\delta K_2) t+ \eps K_3 t^2\right] \cdot e^{2t \left\| \cal{K} \right\|_{\ell_1 \rightarrow \ell_1} }\\
		&=((\eps K_1+\delta K_2) t+ \eps K_3 t^2) \cdot e^{2t \left\| \cal{K} \right\|_{\ell_1 \rightarrow \ell_1} }.
\end{align*}
At the cost of a slightly worse estimate, we gain, however, that the admissible lower dimensionality $k$ of the matrix $M$ does not depend anymore on the higher dimension $d$: indeed we make use of the (strong) Johnson--Lindenstrauss property on $\cal{N}= 2 \cdot \lceil T\cdot \cal{N}'\rceil \cdot N^2$ points. Hence, it suffices to take the minimal target dimension $k_0$ such that $M \in \R^{k\times d}$ with $k \geq k_0$ for
\begin{align*}
k_0 \lesssim \log\left(T \cdot \cal{N}' \cdot N^2\right) \cdot \eps^{-2}.
\end{align*}
Actually, in oder to verify the independence of the dimension $d$, we have to estimate the number of sampling points $\cal{N}'$ independently of it. By \eqref{eq:lipschitzX} in Remark \ref{eq:numb_points}, we know that $L_{x_i - x_j} \leq \sqrt{2NV(0)}$. Analogously, for $L_{y_i-y_j}$ we have 
\begin{align*}
\frac{\|(y_i(t_1)-y_j(t_1))-(y_i(t_2)-y_j(t_2))\|}{|t_1-t_2|}  \leq \sqrt{2NW(0)} \leq \sqrt{2NV(0)(1+\eps)^2}\leq \sqrt{8NV(0)},
\end{align*}
since 
\begin{align*}
\|w_i(0)-w_j(0)\|=\|Mv_i(0)-Mv_j(0)\|\leq (1+\eps)\|v_i(0)-v_j(0)\|.
\end{align*}
Hence we obtain
\begin{align*}
k \lesssim \log\left(T \cdot \cal{N}' \cdot N\right) \cdot \eps^{-2} \lesssim \log\left(T \cdot \sqrt{NV(0)} \cdot \delta^{-1} \cdot N\right) \cdot \eps^{-2} \sim \left(\log\left(T \cdot N \cdot V(0) \right)+|\log \delta|\right) \cdot \eps^{-2},
\end{align*}
so that we confirmed that there is no asymptotic dependence on $d$.
\end{Remark}

\begin{Remark}
The  estimate of Theorem \ref{th:Reduction_uncontrolled}
\begin{align*}
\cal{E}^v(t) & \leq \sqrt{N} \min \left\{ ((\eps K_1+\delta K_2) t+ \eps K_3 t^2) \cdot e^{t \left\| \cal{K} \right\|_{\ell_1 \rightarrow \ell_1}}, \left(\|M\|\sqrt{V(t)} + \sqrt{W(t)}\right) \right\}.
\end{align*}
explains the plot  presented in \cite[Fig.\ 3.5]{FHV13}, where surprisingly the error for large time was shown  to decrease instead of exploding according to classical Gronwall's estimates. Indeed, since $V(t)$ and $W(t)$ are decreasing functions, there is a time when the bound swaps from the exponential Gronwall-type bound to the decreasing curve given by
\begin{align*}
\sqrt{N}\left(\|M\|\sqrt{V(t)} + \sqrt{W(t)}\right).
\end{align*}
Moreover, if both the high-dimensional and the low-dimensional trajectories entered the consensus region already, then $V(t)$ and $W(t)$ approach $0$ as $t$ tends to $+\infty$, forcing $\cal{E}^v(t)$ to tend to $0$.
The vanishing of the discrepancy between the low-dimensional trajectory $(w_k(t))_{k = 1}^N$ of the consensus parameters and the projected trajectory $(Mv_k(t))_{k = 1}^N$ is a remarkable property of the Cucker--Smale system \eqref{eq:freesystemhi} as the initial mean-consensus parameter $\overline{w}(0) = M\overline{v}(0)$ is actually a conserved quantity.
\end{Remark}



\begin{Remark}
In the theorem we can replace $\cal{E}^v(t)$ by $\cal{E}^{v^{\perp}}(t)$ because $0=\|M\overline{v}(0)-\overline{w}(0)\|=\|M\overline{v}(t)-\overline{w}(t)\|$.
\end{Remark}

\section{Dimension reduction of the Cucker--Smale model with control} \label{sec:withControl}

{It was proven in \cite{CFPT14} that a system of type \eqref{eq:freesystemhi} can be driven to the consensus region using a \emph{sparse control strategy}, i.e., a control acting at every instant only on  one agent, whose consensus parameter is the farthest away from the mean consensus parameter. However, if the dimension $d$ of each agent is very large, the numerical simulation of such a dynamical system and its sparse control becomes computationally demanding.

In this section we consider a $k$-dimensional Cucker--Smale system, where $k \ll d$, having as initial conditions the projection of the initial configuration of the original $d$-dimensional system. The projection will be done by a matrix $M \in \R^{k\times d}$ fulfilling the (strong) Johnson--Lindenstrauss property for a certain amount of points. We shall show that the solution of the $k$-dimensional system obtained in this way will stay close to the projected dynamics of the original $d$-dimensional system via the matrix $M$. This, in turn, shall allow us to prove our main result: If we gather the information of which is the farthest agent away from consensus in the $k$-dimensional system and we control this agent in the original high-dimensional system by the sparse strategy presented in \cite{CFPT14}, then we will still be able to drive the high-dimensional system to the consensus region in finite time and with a near-optimal rate.}

One of the main consequences of this fact is that simulations following this strategy will save a relevant amount of computational time with respect to approaching directly the problem in high dimension: indeed, we present in Section \ref{sec:numerics} numerical examples, which show that we can take $k$ even conspicuously smaller than $d$ and still be able to implement a successful sparse control strategy steering the dynamics to the consensus region nearly optimally.

Formally, let us now consider a controlled version of the high-dimensional system
\begin{align} \label{eq:controledsystemhi}
\left \{\begin{array}{lll}
\dot{x}_{i}(t) & = & v_{i}(t)\\
\dot{v}_{i}(t) & = & \frac{1}{N}\underset{j=1}{\overset{N}{\sum}}a\left(\left\Vert x_{i}(t)-x_{j}(t)\right\Vert\right)\left(v_{j}(t)-v_{i}(t)\right) + u^h_i(t), \quad i=1,\ldots,N
\end{array}
\right .
\end{align}
with initial datum $(x(0), v(0)) \in (\mathbb{R}^d)^N \times (\mathbb{R}^d)^N$, and of the associated low-dimensional system
\begin{align} \label{eq:controledsystemlow}
\left \{\begin{array}{lll}
\dot{y}_{i}(t) & = & w_{i}(t), \\	
\dot{w}_{i}(t) & = & \frac{1}{N}\underset{j=1}{\overset{N}{\sum}}a\left(\left\Vert y_{i}(t)-y_{j}(t)\right\Vert\right)\left(w_{j}(t)-w_{i}(t)\right) + u^{\ell}_i(t), \quad i=1,\ldots,N
\end{array}
\right.
\end{align}
with initial condition $(y(0), w(0)) \in (\mathbb{R}^k)^N \times (\mathbb{R}^k)^N$, where $y_i(0) = Mx_i(0)$ and $w_i(0) = Mv_i(0)$ for every $i=1,\ldots,N$, and $M \in \R^{k\times d}$ is a matrix fulfilling the (strong) Johnson--Lindenstrauss property at certain points of the high-dimensional trajectories.


We have already stated that the control  $u^h$ in high dimension shall depend on $u^{\ell}$, the low-dimensional one. Since the latter control is a function of the low-dimensional dynamics determined by the initial datum $(y(0), w(0))$, which in turn depends on $M$, the trajectories of the high-dimensional dynamics depend on $M$ as well. 

As already stated before, given a set of $N$ points, not necessarily explicitly, a random matrix generated by one of the constructions reported in Remark \ref{rm:numberofpoints} fulfills the Johnson--Lindenstrauss property at these $N$ points with a certain high probability. Unfortunately, in the current situation and differently from the one encountered in Section \ref{sec:withoutControl},  the points on the trajectories at which the Johnson--Lindenstrauss property has to hold seem depending on the matrix $M$ that we have generated!

As we shall see in detail in Section \ref{sec:JLL_matrix}, we can resolve this dependency of the high-dimensional trajectories on the generated matrix M, by observing that the realization of the trajectories depends actually on a finite number of control switchings. Hence, for the moment, we just {\it assume} that the Johnson--Lindenstrauss property holds at certain points of the trajectory and we postpone to Section \ref{sec:JLL_matrix} the explanation of how this assumption can in fact hold true.

In what follows, we shall always indicate with $\theta>0$ the maximal amount of resources that the external policy maker is allowed to spend at every instant to keep the system confined. This means that our controls $u^h$ and $u^{\ell}$ will satisfy -- respectively -- the $\ell^N_1(\ell^d_2)$-constraint and the $\ell^N_1(\ell^k_2)$-constraints
\begin{align*}
\sum^N_{i = 1} \| u^h_i\|_{\ell^d_2} \leq \theta, \quad \sum^N_{i = 1} \| u^{\ell}_i\|_{\ell^k_2} \leq \theta.
\end{align*}

\begin{Definition}
\label{def:control}
Let $T > 0$, $(x(t), v(t)) \in (\R^d)^N \times (\R^d)^N$ and $(y(t), w(t)) \in (\R^k)^N \times (\R^k)^N$ be continuous functions defined on the interval $[0,T]$. Let $V(t)$ and $W(t)$ be as in \eqref{Def:XY} and \eqref{eq:Wintro}, respectively. Let us fix a $\Gamma\geq 0$ and define $T^c_0:=\inf \left\{t \in [0,T]: W(t) \leq \Gamma\right\}$ if the set is non-empty, otherwise set $T^c_0:=T$. We define the componentwise feedback controls $u^{h}$ and $u^{\ell}$ as follows: 
\begin{itemize}
	\item if $t \leq T^c_0$, let $\maxindex(t) \in \{1, \ldots, N\}$ be the smallest index such that
	\begin{align}\label{indexch}
	\|w^{\perp}_{\maxindex}(t)\| = \max_{1 \leq i \leq N} \|w^{\perp}_i(t)\|,
	\end{align}
define
\begin{align*}
u^{\ell}_i (t)= \begin{cases} -\theta \frac{w^{\perp}_i(t)}{\|w^{\perp}_i(t)\|}&, \text{ if } i=\maxindex(t) \\
\hspace{0.28cm} 0&, \text{ if } i \not =  \maxindex(t)
\end{cases}
\end{align*}
and
\begin{align}\label{contrrl}
u^{h}_i (t)= \begin{cases} -\theta \frac{v^{\perp}_i(t)}{\|v^{\perp}_i(t)\|}&, \text{ if } i=\maxindex(t) \\
\hspace{0.28cm} 0&, \text{ if } i \not =  \maxindex(t).
\end{cases} 
\end{align}
\item if $t > T^c_0$, then $u^{h}(t)=0$ and $u^{\ell}(t)=0$. We set $\maxindex(t):=0$.
\end{itemize}
We say that the trajectory in low dimension \emph{has entered the consensus region given by the threshold $\Gamma$} if $t \in [T^c_0,T)$.
\end{Definition}
Let us stress now the following observation.
\begin{Remark} \label{rem:onlyLow}
Notice that the control $u^{h}$ is {\it sparse} (all the components are zero except one) and defined exclusively through the following information: the index $\maxindex$ which is computed from the low-dimensional control problem according to \eqref{indexch}, the consensus parameter $v_{\maxindex}$, which is actually the only information to be observed in high-dimension and enters the definition \eqref{contrrl}, and the mean consensus parameter $\overline{v}(t) = \overline{v}(0) + \frac{1}{N} \sum_{i=1}^N\int_{0}^t u_i^h(s) ds$, which one easily computes by integration and sum of previous controls, and is also used in \eqref{contrrl}.
\end{Remark}

There are situations where the computation of the controls $u^h$ and $u^{\ell}$ from Definition \ref{def:control} turns out to be problematic. For instance, if there are only three agents and in the low dimensional system their consensus parameters form an equiangular and equinormal set of vectors at a certain time $t$, then $u^{\ell}$ (and thus $u^h$) are not pointwise computable after $t$ because of chattering effects. A method to avoid chattering in such situations is the use of sample solutions, as defined in \cite{CLSS97}.

\begin{Definition}
\label{def:Sampling_Control}
Let $U \subseteq \mathbb{R}^m$, $f: \mathbb{R}^m \times U \mapsto f(x, u)$ be continuous in $x$ and $u$ as well as locally Lipschitz in $x$ uniformly on every compact subset of $\mathbb{R}^m \times U$. Given a feedback control function $u: \mathbb{R}^m \rightarrow U$, $\tau > 0$, and $x_0 \in \mathbb{R}^m$ we define the \emph{sampling solution associated with the sampling time} $\tau$ of the differential system
\begin{align*}
\dot{x} = f(x, u(x)), \quad x(0) = x_0
\end{align*}
as the piecewise $C^1$-function $x: [0, T] \rightarrow \mathbb{R}^m$ solving the system 
\begin{align*}
\dot{x} = f(x(t), \widetilde{u}(t))
\end{align*}
in the interval $t \in [n \tau, (n+1)\tau]$ recursively for $n \in \N$, where $\widetilde{u}(t)=u(x(n\tau))$ is constant for $t \in [n\tau, (n+1)\tau]$. As the initial value $x(n\tau)$ we use the endpoint of the solution of the preceding interval and start with $x(0) = x_0$. 
\end{Definition}

Let us fix a sampling time $\tau>0$. In the following we shall consider $d$-dimensional and $k$-dimensional Cucker--Smale systems for $k \ll d$ and feedback controls $u^{h}$ and $u^{\ell}$, respectively, as introduced in Definition \ref{def:control}. We shall focus on their sampling solutions $(x,v)$ and $(y,w)$ associated with $\tau$ as defined in Definition \ref{def:Sampling_Control}, hence
\begin{align*}
\widetilde{u}^{\ell} (t) = u^{\ell}(n\tau), \quad \widetilde{u}^{h} (t) = u^{h}(n\tau)
\end{align*}
for $t \in [n\tau,(n+1)\tau)$. Since we are only able to change the control at times which are multiples of $\tau$, we define the \emph{switch-off time of the sampled control associated with the threshold $\Gamma$} as
\begin{align}
\label{eq:T0s}
T^s_0 := \inf_{n\in \N_0} \left\{n\tau: W(n\tau) \leq \Gamma\right\},
\end{align}
otherwise set $T^s_0:=T$ if the set whose infimum is taken is empty. Because in the rest of the paper we shall deal only with sampled control, we will refer to $T^s_0$ with $T_0$, omitting the superscript.


In the following, we shall show an estimate of the error between the projection of the sampled controlled high-dimensional system and the sampled controlled low-dimensional system, under the crucial assumption of the validity of the weak Johnson--Lindenstrauss property for $M$ for the differences of trajectories of the system.
 
This result is the controlled counterpart of Theorem \ref{th:Reduction_uncontrolled}.

\begin{Proposition}
\label{prop:diff_eq_est}
Let $T>0$, $\Delta>0$ and $k\in \N_0$ with $k\leq d $. Let $\tau>0$ be a sampling time, $\hat{T}>0$ such that $\hat{T}+\tau\leq T$ and let $M \in \R^{k\times d}$. 

Let $(x(t),v(t))$ be the sampling solution of the $d$-dimensional Cucker--Smale system \eqref{eq:controledsystemhi} with initial conditions $(x(0),v(0))$ and $(y(t),w(t))$ be the sampling solution of the $\R^k$-projected Cucker--Smale system \eqref{eq:controledsystemlow} with initial conditions $y(0)=\left(Mx_1(0),\ldots,Mx_N(0)\right) \in (\R^k)^N$ and $w(0)=\left(Mv_1(0),\ldots,Mv_N(0)\right) \in (\R^k)^N$ as defined in Definition \ref{def:Sampling_Control}, where $u_i^h$ and $u_i^\ell$ are the controls from Definition \ref{def:control} with threshold $\Gamma=(2\Delta)^2$. Moreover, let $T_0$ be as in \eqref{eq:T0s}.

Suppose that $W$ is non-increasing in time and that there exists a constant $\alpha \geq 0$ such that
\begin{align*}
\int_0^{t} \sqrt{2V(s)} \ ds \leq \alpha \text{ for all } t \in [0,\min(\hat{T}+\tau,T_0)].
\end{align*}
Let $\eps' \in (0,1)$ be so small that
\begin{align*}
 \eps' \sqrt{N}\left(4 L_a \sqrt{NV(0)} \left(\sqrt{2X(0)} +  \alpha \right) + \frac{\theta}{\sqrt{N}}\right)\cdot (\hat{T}+\tau) e^{(\hat{T}+\tau) \left(\max\left(2a(0) + 1, 4 L_a \sqrt{NV(0)}\right)+ \frac{8\theta}{\Delta}\right)} \leq \frac{\Delta}{2}.
 \end{align*}
Assume that the matrix $M$ 
\begin{enumerate}[label=(JL\arabic*)]
\item \label{item:JLL1} has the weak Johnson--Lindenstrauss property of parameter $\eps=\eps'$ and $\delta = \min\left(\eps'\frac{\sqrt{2X(0)}+\alpha}{2},1/2\right)$ at all points $x_i(t)-x_j(t)$  for $i,j \in \{1,\ldots,N\}$, $t \in [0,\hat{T}+\tau]$,
\item \label{item:JLL2} has the weak Johnson--Lindenstrauss property of parameter $\eps=\eps'$ and $\delta=\Delta$ at all points $v_i(n\tau)-\overline{v}(n\tau)$ for $i \in \{1,\ldots,N\}$, $n=0,\ldots,\lfloor \frac{\hat{T}}{\tau} \rfloor+1 $, and
\item \label{item:JLL3} has the (strong) Johnson--Lindenstrauss property of parameter $\eps=1/2$ at the points $v_i(0)-\overline{v}(0)$ and $x_i(0)-\overline{x}(0)$ for $i \in \{1,\ldots,N\}$ .
\end{enumerate}
Define the following errors:
\begin{align*}
	e^x_i(t) &= \| y_i(t) - Mx_i(t) \|, \quad e^v_i(t) = \| w_i(t) - Mv_i(t) \|, \\
\cal{E}^x (t) &= \max_{i = 1, \ldots, N} e^x_i(t), \quad\phantom{'''''''} \cal{E}^v (t) = \max_{i = 1, \ldots, N} e^v_i(t).
\end{align*}
Then it holds
\begin{align}
\label{eq:est_contr}
 \cal{E}^{v}(t) + \cal{E}^x(t) & \leq \eps'\sqrt{N}\left(4 L_a \sqrt{NV(0)} \left(\sqrt{2X(0)} +  \alpha \right) +  \frac{\theta}{\sqrt{N}}\right) \cdot te^{t \left(\max\left(2a(0) + 1, 4 L_a \sqrt{NV(0)}\right)+ \frac{8\theta}{\Delta}\right)} {\leq \frac{\Delta}{2}}
\end{align}
for all $t \in [0,\min(\hat{T}+\tau,T_0)]$.
\end{Proposition}
\begin{proof}
We argue by induction: We want to show that if \eqref{eq:est_contr} holds true at $t \in \{0,\tau,\ldots,n\tau\}$, then it is also true for $ t \in [n,(n+1)\tau]$, in particular at $t=(n+1)\tau$, as long as $n\tau \leq \hat{T}$ and $n\tau < T_0$, i.e., the control is not switched off before $(n+1)\tau$. 
Obviously, \eqref{eq:est_contr} holds for $n=0$, this means at $t=0$, and actually arguing in the same way as in the following inductive step, the base
step is verified. 

So, let $t \in [n\tau, (n+1)\tau]$ for $n \in \N_0$. First, we consider the estimate on the agent on which the control is acting. We shall estimate the decay in order to use Gronwall Lemma as in Theorem \ref{th:Reduction_uncontrolled}. We have
\begin{align}
	\begin{split}
	\label{eq:est_der}
		\frac{d}{dt} e_{\maxindex}^{v}(t) & \leq \left\| \frac{d}{dt} \left( w_{\maxindex}(t)- Mv_{\maxindex}(t)\right)\right\| \\
		&\leq   \frac{1}{N} \sum^N_{j = 1} \left\| a(\|y_{\maxindex}(t) - y_j(t) \|)(w_j(t) - w_{\maxindex}(t)) - a(\|x_{\maxindex}(t) - x_j(t) \|)(Mv_j(t) - Mv_{\maxindex}(t)) \right\|  \\
		& \quad + \theta  \left\| \frac{w^{\perp}_{\maxindex}(n\tau)}{\|w^{\perp}_{\maxindex}(n\tau)\|} - \frac{Mv^{\perp}_{\maxindex}(n\tau)}{\|v^{\perp}_{\maxindex}(n\tau)\|} \right\|.
		\end{split}
		\end{align}
	       For $i \in \{1\ldots,N\}$ and $i\neq \maxindex$ we have
	       \begin{align}
	       	\label{eq:est_der2}
	       \frac{d}{dt} e_{i}^{v}(t) &\leq   \frac{1}{N} \sum^N_{j = 1} \left\| a(\|y_i(t) - y_j(t) \|)(w_j(t) - w_i(t)) - a(\|x_i(t) - x_j(t) \|)(Mv_j(t) - Mv_i(t)) \right\|.
		\end{align}
		We now focus on the control term:
		\begin{align}
		\label{eq:moon}
		\begin{split}
		&\left\| \frac{w^{\perp}_{\maxindex}(n\tau)}{\|w^{\perp}_{\maxindex}(n\tau)\|} - \frac{Mv^{\perp}_{\maxindex}(n\tau)}{\|v^{\perp}_{\maxindex}(n\tau)\|} \right\| \\
		&\quad = \frac{1}{\|w^{\perp}_{\maxindex}(n\tau)\|\|v^{\perp}_{\maxindex}(n\tau)\|} \left\|\|v^{\perp}_{\maxindex}(n\tau)\|w^{\perp}_{\maxindex}(n\tau) - \|w^{\perp}_{\maxindex}(n\tau)\|Mv^{\perp}_{\maxindex}(n\tau) \right\| \\
		& \quad= \frac{1}{\|w^{\perp}_{\maxindex}(n\tau)\|\|v^{\perp}_{\maxindex}(n\tau)\|} \left\|\left(\|v^{\perp}_{\maxindex}(n\tau)\|-\|w^{\perp}_{\maxindex}(n\tau)\|\right)w^{\perp}_{\maxindex}(n\tau) - \|w^{\perp}_{\maxindex}(n\tau)\|\left(Mv^{\perp}_{\maxindex}(n\tau) - w^{\perp}_{\maxindex}(n\tau) \right) \right\| \\
		&\quad \leq \frac{1}{\|v^{\perp}_{\maxindex}(n\tau)\|}\left|\|v^{\perp}_{\maxindex}(n\tau)\|-\|w^{\perp}_{\maxindex}(n\tau)\|\right| + \frac{1}{\|v^{\perp}_{\maxindex}(n\tau)\|}
	\left\|Mv^{\perp}_{\maxindex}(n\tau) - w^{\perp}_{\maxindex}(n\tau) \right\|.
		\end{split}
		\end{align}
		Since by assumption $n\tau<T_0$ and \eqref{eq:est_contr} holds at $n\tau$ by the inductive hypothesis, it follows
		\begin{align}
		\label{eq:est_wi}
		\begin{split}
		\|w^{\perp}_{\maxindex}(n\tau)\| &\geq \sqrt{W(n\tau)} > 2 \Delta, \\
		\|Mv^{\perp}_{\maxindex}(n\tau) - w^{\perp}_{\maxindex}(n\tau)\| &\leq \|Mv_{\maxindex}(n\tau) - w_{\maxindex}(n\tau)\|+\frac{1}{N} \sum_{j=1}^N \|Mv_{j}(n\tau) - w_{j}(n\tau)\|\leq 2\cal{E}^{v}(n\tau) \leq \Delta.
		\end{split}
		\end{align}
Hence
		\begin{align}
		\label{eq:largerDelta}
		\|Mv^{\perp}_{\maxindex}(n\tau)\| >  \Delta.
		\end{align}
		From assumption \ref{item:JLL2} and \eqref{eq:largerDelta} it follows that the (strong) Johnson--Lindenstrauss property with parameter $\eps=\eps'$ holds at $v^{\perp}_{\maxindex}$. 		
		Hence
		\begin{align*}
		\left|\|v^{\perp}_{\maxindex}(n\tau)\|-\|w^{\perp}_{\maxindex}(n\tau)\|\right| & \leq \left|\|v^{\perp}_{\maxindex}(n\tau)\| - \|Mv^{\perp}_{\maxindex}(n\tau)\|\right| + \left|\|Mv^{\perp}_{\maxindex}(n\tau)\| -\|w^{\perp}_{\maxindex}(n\tau)\|\right| \\
		& \leq \left|\|v^{\perp}_{\maxindex}(n\tau)\| - \|Mv^{\perp}_{\maxindex}(n\tau)\|\right| + \|Mv^{\perp}_{\maxindex}(n\tau) -w^{\perp}_{\maxindex}(n\tau)\| \\
		& \leq \eps'\|v^{\perp}_{\maxindex}(n\tau)\| + \|Mv^{\perp}_{\maxindex}(n\tau) -w^{\perp}_{\maxindex}(n\tau)\|
		\end{align*}
		and 
			\begin{align*}
		\|v^{\perp}_{\maxindex}(n\tau)\| \geq \frac{1}{1+\eps'}  \|Mv^{\perp}_{\maxindex}(n\tau)\| \geq \frac{\Delta}{2}. 
\end{align*}
Inserting these estimates into \eqref{eq:moon} and using \eqref{eq:est_wi} we get
		\begin{align*}
		\begin{split}
		\left\| \frac{w^{\perp}_{\maxindex}(n\tau)}{\|w^{\perp}_{\maxindex}(n\tau)\|} - \frac{Mv^{\perp}_{\maxindex}(n\tau)}{\|v^{\perp}_{\maxindex}(n\tau)\|} \right\| & \leq \eps' + 2 \frac{ \|Mv^{\perp}_{\maxindex}(n\tau) -w^{\perp}_{\maxindex}(n\tau)\|}{\|v^{\perp}_{\maxindex}(n\tau)\|} \\
		& \leq \eps' + \frac{8\cal{E}^{v}(n\tau)}{\Delta}.
	\end{split}
	\end{align*}
Now we add the estimates for the derivatives of $e_{\maxindex}^{v}$ in \eqref{eq:est_der} and $e_{i}^{v}$ for $i\neq \maxindex$ in \eqref{eq:est_der2}. By \ref{item:JLL1} the weak Johnson--Lindenstrauss property of parameter $\eps=\eps'$ and $\delta=\min\left(\eps'\frac{\sqrt{2X(0)}+\alpha}{2},1/2\right)$ holds at $x_i(t)-x_j(t)$ for $t \in [0,(n+1)\tau)$. Hence, for the first (uncontrolled) part of \eqref{eq:est_der} and \eqref{eq:est_der2} we can use the same estimate as in \eqref{eq:CSU_derivativeerrorv_partial}. Thus, setting 
\begin{align*}
\cal{E}_2^v(t)=\left(\frac{1}{N}\sum_{i=1}^N (e^v_i(t))^2\right)^{1/2}, \quad \cal{E}_2^x (t) = \left(\frac{1}{N}\sum_{i=1}^N (e^x_i(t))^2\right)^{1/2}
\end{align*}
we obtain the bound
	\begin{align*}
	\begin{split}
	\frac{d}{dt}\cal{E}_2^{v}(t)&\leq \left(\frac{1}{N}\sum_{i=1}^N\left(\frac{d}{dt}  e^{v}_i(t)\right)^2\right)^{1/2}\\
	&\leq  \eps' L_a \sqrt{NW(0)} \left(\sqrt{2X(0)} +  \int_0^t \sqrt{2V(s)} \ ds \right)+   2 L_a\sqrt{NW(0)} (\delta+ \cal{E}_2^x(t)) + 2 a(0) \cal{E}_2^{v}(t) \\
	& \quad + \frac{\theta}{\sqrt{N}} \left\| \frac{w^{\perp}_{\maxindex}(n\tau)}{\|w^{\perp}_{\maxindex}(n\tau)\|} - \frac{Mv^{\perp}_{\maxindex}(n\tau)}{\|v^{\perp}_{\maxindex}(n\tau)\|} \right\| \\
	&\leq \eps' \left(2 L_a \sqrt{NW(0)} \left(\sqrt{2X(0)} +  \alpha \right) + \frac{\theta}{\sqrt{N}}\right) + 2 L_a\sqrt{NW(0)}\cal{E}_2^x(t) + 2 a(0) \cal{E}_2^{v}(t)+ \frac{8\theta}{\Delta}\cal{E}_2^{v}(n\tau)
	\end{split}
\end{align*}
using $\cal{E}^{v}(n\tau)\leq \sqrt{N} \cal{E}_2^{v}(n\tau)$ and the definition of $\delta$. For $\cal{E}_2^x$ we have
\begin{align*}
	\begin{split}
	\frac{d}{dt}\cal{E}_2^x(t)& \leq \left(\frac{1}{N}\sum_{i=1}^N\left(\frac{d}{dt}  e^x_i(t)\right)^2\right)^{1/2} \\
	&\leq\left(\frac{1}{N}\sum_{i=1}^N\left(e^v_i(t)\right)^2\right)^{1/2} \\
	&= \cal{E}_2^{v}(t).
	\end{split}
\end{align*}
By integrating the estimates for $\frac{d}{dt}\cal{E}_2^{v}(t)$ and $\frac{d}{dt}\cal{E}_2^x(t)$ between $n\tau$ and $t$  we get 
\begin{align*}
		\left[
		\begin{array}{c}
		\cal{E}_2^{v}(t) \\
		\cal{E}_2^x(t)
		\end{array}
		\right] \leq
		\left[
		\begin{array}{c}
		\left(1+ \frac{8\theta}{\Delta} (t - n\tau) \right)\cal{E}_2^{v}(n\tau)+\eps'\left(K_1+\frac{\theta}{\sqrt{N}}\right)   (t - n\tau) \\
		\cal{E}_2^x(n\tau)
		\end{array}
		\right] + \int_{n\tau}^t \cal{K} \cdot
		\left[
		\begin{array}{c}
		\cal{E}_2^{v}(s) \\
		\cal{E}_2^x(s)
		\end{array}
		\right] 
		ds
	\end{align*}
	with 
	\begin{align*}
	  \cal{K}=\left[
		\begin{array}{cc}
		2 a(0) & 2 L_a\sqrt{NW(0)} \\
		1 & 0
		\end{array}
		\right]
	\end{align*}
		and
	\begin{align*}
K_1=2 L_a \sqrt{2NW(0)} \left(\sqrt{2X(0)} +  \alpha \right).
\end{align*}
Now we apply the $\ell_1$-norm to the inequality and obtain 
	\begin{align*}
		\cal{E}_2^{v}(t) + \cal{E}_2^x(t) &\leq \eps' \left(K_1 + \frac{\theta}{\sqrt{N}}\right)  (t - n\tau)+\left(1+ \frac{8\theta}{\Delta} (t - n\tau)\right)\left(\cal{E}_2^{v}(n\tau)+ \cal{E}_2^x(n\tau)\right) \\
		& \quad +\int_{n\tau}^{t} \left\|\cal{K}\right\|_{\ell_1 \rightarrow \ell_1} \cdot (\cal{E}_2^{v}(s) + \cal{E}_2^x(s)) \ ds. 
	\end{align*}	
 The discrete Gronwall Lemma \ref{le:discr_gronwall} applied for
 \begin{align*}
\rho(t):=\left(K_1 + \frac{\theta}{\sqrt{N}}\right) \eps' t, \quad \beta_1(t):=\frac{8\theta}{\Delta}t, \quad \beta_2(t)\equiv  \left\| \cal{K} \right\|_{\ell_1 \rightarrow \ell_1}, \quad u(t):=\cal{E}_2^{v}(t) + \cal{E}_2^x(t)\geq 0
\end{align*}
yields
\begin{align*}
		\cal{E}_2^{v}(t) + \cal{E}_2^x(t) &  \leq \left[\cal{E}_2^{v}(0) + \cal{E}^x(0)\right]\cdot e^{t \left(\left\| \cal{K} \right\|_{\ell_1 \rightarrow \ell_1}+ \frac{8\theta}{\Delta}\right)} +\eps'\left(K_1 +  \frac{\theta}{\sqrt{N}}\right) t e^{t\left(\left\| \cal{K} \right\|_{\ell_1 \rightarrow \ell_1}+ \frac{8\theta}{\Delta}\right)} \\
	&=\eps' \left(K_1 + \frac{\theta}{\sqrt{N}}\right) t e^{t\left(\left\| \cal{K} \right\|_{\ell_1 \rightarrow \ell_1}+ \frac{8\theta}{\Delta}\right)}
	\end{align*}
	because the initial time errors are $0$ by definition of the low-dimensional system.
	Hence using a trivial estimate of the $\ell_2$-norm by the $\ell_{\infty}$-norm we conclude the induction and also the proof:
	\begin{align*}
		\cal{E}^{v}(t) + \cal{E}^x(t)&\leq \eps' \sqrt{N}\left(K_1 + \frac{\theta}{\sqrt{N}}\right) t e^{t\left(\left\| \cal{K} \right\|_{\ell_1 \rightarrow \ell_1}+ \frac{8\theta}{\Delta}\right)}\\ 
		&\leq  \eps' \sqrt{N}\left(4 L_a \sqrt{NV(0)} \left(\sqrt{2X(0)} +  \alpha \right) + \frac{\theta}{\sqrt{N}}\right)\cdot t e^{t \left(\max\left(2a(0) + 1, 4 L_a \sqrt{NV(0)}\right)+ \frac{8\theta}{\Delta}\right)}
\end{align*}
using that $\sqrt{W(0)}\leq (1+1/2) \sqrt{V(0)} \leq 2 \sqrt{V(0)}$ by \ref{item:JLL3}.
\end{proof}

Now we are in the position of  showing that we can steer both the low- and high-dimensional systems simultaneously to the consensus region using the control defined in Definition \ref{def:control} and Definition \ref{def:Sampling_Control}. We repeat that this means that we choose the index of the agent on which the sparse control acts from the low-dimensional system and use the same index for the control in the high-dimensional system. The challenge here is ensuring that the control coming out of this procedure drives the high-dimensional system to consensus as well. For this we need the estimates from Proposition \ref{prop:diff_eq_est} to show that the error of the projection of the high-dimensional system and the low-dimensional system stay near to each other. Additionally, from \cite{CFPT14} it is known that the low-dimensional system will be steered {\it optimally} to the consensus region in finite time using the sampled version of the control introduced in Definition \ref{def:control}.

\begin{Theorem}
\label{eq:convergence}
Let $x(0)=\left(x_1(0),\ldots,x_N(0)\right) \in (\R^d)^N $ and $v(0)=\left(v_1(0),\ldots,v_N(0)\right) \in (\R^d)^N$ be given and let  $k\leq d$. Let $\gamma$, $X(0)$ as well as $V(0)$ be defined as before and let $c,C$ be the constants from Lemma \ref{le:CSCS_technical}. Let
\begin{align}
\label{eq:Xbar}
	\overline{X} := 2X(0) + \frac{2 N^2}{c^2 \theta^2}V(0)^2,
\end{align}
as well as
\begin{align}
\label{eq:Delta}
\Delta:=\min\left(\frac{\gamma(\overline{X})}{C},\frac{1}{2} \gamma\left(4\overline{X}\right)\right),
\end{align}
and
\begin{align}
\label{eq:That}
\hat{T} := \frac{2N}{\theta} \left( 2\sqrt{V(0)} -2 \Delta\right).
\end{align}
Let $\tau_0>0$ be so small that
\begin{align}
		\begin{split}\label{eq:CSCS_tausmall}
			\tau_0 \left(a(0)\sqrt{N}\sqrt{V(0)}+\theta\right) + \tau_0^2 a(0)\theta \leq \frac{\Delta}{4},
			\end{split}
		\end{align}
and fix $\tau\in (0,\tau_0]$. Furthermore, let $M \in \R^{k\times d}$. Let $u^{h}$ and $u^{\ell}$ be the controls as in Definition \ref{def:control} with threshold $\Gamma= (2\Delta)^2$. Let $(x(t),v(t))$ be the sampling solution of the $d$-dimensional Cucker--Smale system
\begin{align*} 
\left \{\begin{array}{lll}
\dot{x}_{i}(t) & = & v_{i}(t)\\
\dot{v}_{i}(t) & = & \frac{1}{N}\underset{j=1}{\overset{N}{\sum}}a\left(\left\Vert x_{i}(t)-x_{j}(t)\right\Vert \right)\left(v_{j}(t)-v_{i}(t)\right) + u^h_i(t), \quad i=1,\ldots,N
\end{array}
\right .
\end{align*}
with initial conditions $(x(0),v(0))$ associated with the sampling time $\tau$ and $(y(t),w(t))$ be the sampling solution of the $\R^k$-projected Cucker--Smale system
\begin{align*}
\left \{\begin{array}{lll}
\dot{y}_{i}(t) & = & w_{i}(t), \\	
\dot{w}_{i}(t) & = & \frac{1}{N}\underset{j=1}{\overset{N}{\sum}}a\left(\left\Vert y_{i}(t)-y_{j}(t)\right\Vert \right)\left(w_{j}(t)-w_{i}(t)\right) + u^{\ell}_i(t), \quad i=1,\ldots,N
\end{array}
\right.
\end{align*}
with initial conditions $y(0)=\left(Mx_1(0),\ldots,Mx_N(0)\right) \in (\R^k)^N$ and $w(0)=\left(Mv_1(0),\ldots,Mv_N(0)\right) \in (\R^k)^N$ associated with the sampling time $\tau$.

Let $\alpha=\frac{\sqrt{2}N}{c \theta}$ and $\eps' \in (0,1)$ be so small that 
\begin{align}
\label{eq:epssuperest}
 \eps' \sqrt{N} \left(4 L_a \sqrt{NV(0)} \left(\sqrt{2X(0)} +  {\alpha} \right) + \frac{\theta}{\sqrt{N}}\right)\cdot (\hat{T}+\tau)e^{(\hat{T}+\tau)\left(\max\left(2a(0) + 1, 4 L_a \sqrt{NV(0)}\right)+ \frac{8\theta}{\Delta}\right)} \leq \frac{\Delta}{2}.
 \end{align}
Assume that the matrix $M$ 
\begin{enumerate}[label=(JL\arabic*)]
\item \label{item:JL1} has the weak Johnson--Lindenstrauss property of parameter $\eps=\eps'$ and $\delta =\eps'\frac{\sqrt{2X(0)}+\alpha}{2}$ at all points $x_i(t)-x_j(t)$  for $i,j \in \{1,\ldots,N\}$, $t \in [0,\hat{T}+\tau]$,
\item \label{item:JL2} has the weak Johnson--Lindenstrauss property of parameter $\eps=\eps'$ and $\delta=\Delta$ at all points $v_i(n\tau)-\overline{v}(n\tau)$ for $i \in \{1,\ldots,N\}$, $n=0,\ldots,\lfloor \frac{\hat{T}}{\tau} \rfloor+1 $, and
\item \label{item:JL3} has the (strong) Johnson--Lindenstrauss property of parameter $\eps=1/2$ at the points $v_i(0)-\overline{v}(0)$ and $x_i(0)-\overline{x}(0)$ for $i \in \{1,\ldots,N\}$ .
\end{enumerate}
Then there exists an $n \in \N_0$ such that $\sqrt{W(n\tau)}\leq 2\Delta$. Moreover, setting  
\begin{align*}
T_0=n^*\tau, \text{ where } n^*:=\min\left \{ n \in \N_0: \sqrt{W(n\tau)}\leq 2\Delta\right\},
\end{align*}
it holds that at $T_0$ both the high-dimensional and the projected low-dimensional systems are in the consensus region defined by Lemma \ref{le:HaHaKim}. Furthermore, we have the estimates
\begin{align*}
 		T_0\leq \frac{2N}{\theta} \left( \sqrt{W(0)} - 2\Delta\right)+\tau\leq \hat{T}+\tau
	\end{align*}
as well as
\begin{align}
\label{eq:M}
\begin{split}
\max_{t \in [0, T_0]} \max_{i, j} \| x_i(t) - x_j(t) \| \leq 2\sqrt{N \overline{X}}, \\
\max_{t \in [0, T_0]} \max_{i, j} \| v_i(t) - v_j(t) \| \leq 2\sqrt{N V(0)}.
\end{split}
\end{align}
\end{Theorem}

\begin{proof}

\textit{First step:} Let
		\begin{align}
		\label{eq:Ybar}
		 \overline{Y} = 2Y(0) + \frac{2 N^2}{\theta^2}W(0)^2.
		\end{align}
We shall  prove the following implication for every $n \in \N$ such that $n\tau\leq \hat{T}$: if $\sqrt{W(m\tau)} > 2 \Delta$ for every $m = 0, \ldots, n$ and the subsequent assumptions $P_1(n)$, $P_2(n)$, and $P_3(n)$ depending on $n$ hold
\begin{enumerate}[label=$P_\arabic*(n)$]
\item \label{P1}
: For $t \in [0, n\tau)$ it holds
	\begin{align*}
			W^{\prime}(t) & \leq - \frac{ \theta}{N} \sqrt{W(t)} < 0, \\
			V^{\prime}(t) & \leq - \frac{c \theta}{N} \sqrt{V(t)} < 0;
		\end{align*}
\item \label{P2} : $Y(t) \leq \overline{Y}$ and $X(t) \leq \overline{X}$ hold in $[0,n\tau]$;
\item \label{P3} : It holds 
\begin{align*}
\int_0^{n\tau} \sqrt{2V(s)} \ ds \leq \alpha,
\end{align*}
\end{enumerate}
then also $P_1(n+1)$, $P_2(n+1)$, and $P_3(n+1)$ hold true.

So let us assume $\sqrt{W(m\tau)} > 2 \Delta$ for every $m = 0, \ldots, n$, which means that $T_0\geq (n+1)\tau$ by definition of $T_0$, and assume \ref{P1}, \ref{P2}, and \ref{P3}. We begin by computing the derivative of $V$ and $W$ for $t \in [n\tau, (n+1)\tau]$:
		\begin{align*}
			V^{\prime}(t) & = \frac{d}{dt}B(v(t), v(t)) \\
			& = 2 B(\dot{v}(t), v(t)) \\
			& \leq 2 B(u^{h}(n\tau), v(t)).
		\end{align*}
		The same computation yields
		\begin{align*}
			W^{\prime}(t) \leq 2 B(u^{\ell}(n\tau), w(t)).
		\end{align*}
		By definition, $u_{\maxindex}^{\ell}(n\tau) = - \theta w^{\perp}_{\maxindex}(n\tau)/\|w^{\perp}_{\maxindex}(n\tau)\|$ where $\maxindex$ is the smallest index such that $\|w^{\perp}_{\maxindex}(n\tau)\| \geq \|w^{\perp}_j(n\tau)\|$ for all $j = 1, \ldots, N$, and $u_j^{\ell}(n\tau) = 0$ for every $j \not = {\maxindex}$. Then $u_{\maxindex}^{h}(n\tau) = - \theta v^{\perp}_{\maxindex}(n\tau)/\|v^{\perp}_{\maxindex}(n\tau)\|$ and $u_j^{h}(n\tau) = 0$ for every $j \not = i$. So we have
		\begin{align}
		\begin{split} \label{eq:CSCS_estimatederivativeVWn}
			V^{\prime}(t) \leq - \frac{2 \theta}{N} \phi^{h}(t), \\
			W^{\prime}(t) \leq - \frac{2 \theta}{N} \phi^{\ell}(t),
		\end{split}
		\end{align}
		where
		\begin{align*}
			\phi^{h}(t) = \frac{\langle v^{\perp}_{\maxindex}(n\tau), v^{\perp}_{\maxindex}(t) \rangle}{\|v^{\perp}_{\maxindex}(n\tau)\|}, \\
			\phi^{\ell}(t) = \frac{\langle w^{\perp}_{\maxindex}(n\tau), w^{\perp}_{\maxindex}(t) \rangle}{\|w^{\perp}_{\maxindex}(n\tau)\|}.
		\end{align*}
As we want to prove that $P_1(n+1)$ holds, we need to deduce suitable lower bounds on $\phi^{h}(t)$ and $\phi^{\ell}(t)$ to estimate the right-hand side of \eqref{eq:CSCS_estimatederivativeVWn}. To this purpose we need first do derive auxiliary bounds on the growth of $\sqrt{V(t)}$ and $\sqrt{W(t)}$, see formula \eqref{eq:CSCS_estimateVWn} below: The general estimate
		\begin{align*} 
			\frac{\langle a, b_i \rangle}{\|a\|} \leq \|b_i\| \leq \sqrt{N} \sqrt{B}
		\end{align*}
		with arbitrary vectors $a$ and $b_1, \ldots, b_N$, and $B = \frac{1}{N} \sum^N_{i = 1} \|b_i\|^2$ yields
		\begin{align*}
			|\phi^{\ell}(s)| \leq \sqrt{N} \sqrt{W(s)}, \quad |\phi^{h}(s)| \leq \sqrt{N} \sqrt{V(s)}
		\end{align*}
		for all $s \in [n\tau,(n+1)\tau]$. We use these bounds to estimate the right-hand side of \eqref{eq:CSCS_estimatederivativeVWn} 
		\begin{align*}
			\left(\sqrt{V}\right)'(s)&= \frac{V'(s)}{2\sqrt{V(s)}} \leq  \frac{\theta}{\sqrt{N}}, \\
\left(\sqrt{W}\right)'(s)&= \frac{W'(s)}{2\sqrt{W(s)}} \leq  \frac{\theta}{\sqrt{N}}.
		\end{align*}
		An integration between $n\tau$ and $s \in [n\tau,(n+1)\tau]$ yields
		\begin{align}
		\begin{split} \label{eq:CSCS_estimateVWn}
			\sqrt{W(s)} \leq \sqrt{W(n\tau)} + (s - n\tau)\frac{\theta}{\sqrt{N}} \leq \sqrt{W(n\tau)} + \tau \frac{\theta}{\sqrt{N}}, \\
			\sqrt{V(s)} \leq \sqrt{V(n\tau)} + (s - n\tau)\frac{\theta}{\sqrt{N}} \leq \sqrt{V(n\tau)} + \tau \frac{\theta}{\sqrt{N}}.
		\end{split}
		\end{align}
		With the help of \eqref{eq:CSCS_estimateVWn} we now work out lower bounds for $\phi^{\ell}$ and $\phi^{h}$. It holds
		\begin{align}
		\begin{split} \label{eq:CSCS_philown}
			\phi^{\ell}(t) & = \frac{\langle w^{\perp}_{\maxindex}(n\tau), w^{\perp}_{\maxindex}(t) \rangle}{\|w^{\perp}_{\maxindex}(n\tau)\|} \\
			& = \frac{\langle w^{\perp}_{\maxindex}(n\tau), w^{\perp}_{\maxindex}(n\tau) \rangle}{\|w^{\perp}_{\maxindex}(n\tau)\|} - \frac{\langle w^{\perp}_{\maxindex}(n\tau), w^{\perp}_{\maxindex}(n\tau) - w^{\perp}_{\maxindex}(t) \rangle}{\|w^{\perp}_{\maxindex}(n\tau)\|} \\
			& \geq \|w^{\perp}_{\maxindex}(n\tau)\| - \|w^{\perp}_{\maxindex}(n\tau) - w^{\perp}_{\maxindex}(t)\| \\
			& \geq \|w^{\perp}_{\maxindex}(n\tau)\| - \int^t_{n\tau}\|\dot{w}^{\perp}_{\maxindex}(s)\| \ ds.
			\end{split}
			\end{align}
			We now estimate the integrand. From 
			\begin{align*}
				\dot{w}^{\perp}_{\maxindex}(t)&=\frac{1}{N}\sum_{j=1}^N a(\|y_j(t)-y_{\maxindex}(t)\|)(w^{\perp}_{j}(t)-w^{\perp}_{\maxindex}(t))+ u_{\maxindex}^{\ell}(n\tau)-\frac{1}{N}\sum_{j=1}^N u_j^{\ell}(n\tau) \\
				&=\frac{1}{N}\sum_{j=1}^N a(\|y_j(t)-y_{\maxindex}(t)\|)(w^{\perp}_{j}(t)-w^{\perp}_{\maxindex}(t))- \theta \cdot \frac{N-1}{N} \cdot \frac{w^{\perp}_{\maxindex}(n\tau)}{\|w^{\perp}_{\maxindex}(n\tau)\|}
			\end{align*}		
			and the inequality			
			\begin{align*}
				\frac{1}{N}\sum_{j =1}^N \|w^{\perp}_{j}-w^{\perp}_{\maxindex}\|&\leq \left(\frac{1}{N}\sum_{j=1}^N \|w^{\perp}_{k}-w^{\perp}_{\maxindex}\|^2\right)^{\frac{1}{2}}	\leq  \left(\frac{1}{N}\sum_{j=1}^N \sum_{j'=1}^{j-1} \|w^{\perp}_{j'}-w^{\perp}_{j}\|^2\right)^{\frac{1}{2}}=\left(\frac{1}{2N}\sum_{j,j'=1}^N \|w^{\perp}_{j}-w^{\perp}_{j'}\|^2\right)^{\frac{1}{2}} \\
				&= \sqrt{N}\sqrt{W}
			\end{align*}										
			it follows
			\begin{align*}
			\|\dot{w}^{\perp}_{\maxindex}(s)\| & \leq a(0)\sqrt{N}\sqrt{W(s)} + \theta \frac{N-1}{N} \quad\text{for all}\quad s \in [n\tau,(n+1)\tau).
		\end{align*}
		Using \eqref{eq:CSCS_estimateVWn} we get
		\begin{align*}
			\|\dot{w}^{\perp}_{\maxindex}(s)\| & \leq a(0)\sqrt{N}\left(\sqrt{W(n\tau)} + \tau \frac{\theta}{\sqrt{N}}\right) + \theta.
		\end{align*}
		Plugging the last inequality into \eqref{eq:CSCS_philown} we deduce
		\begin{align*}
		\phi^{\ell}(t) & \geq \|w^{\perp}_{\maxindex}(n\tau)\| - \tau \left(a(0)\sqrt{N}\left(\sqrt{W(n\tau)} + \tau \frac{\theta}{\sqrt{N}}\right) + \theta\right) \\
		& \geq \|w^{\perp}_{\maxindex}(n\tau)\| - \tau \left(a(0)\sqrt{N}\sqrt{W(n\tau)} + \theta \right) - \tau^2 a(0)\theta.
		\end{align*}
		The same calculations give us
		\begin{align*}
		\phi^{h}(t)  & \geq \|v^{\perp}_{\maxindex}(n\tau)\| - \tau \left(a(0)\sqrt{N}\sqrt{V(n\tau)} + \theta \right) - \tau^2 a(0)\theta.
		\end{align*}
		Together with \eqref{eq:CSCS_estimatederivativeVWn} this yields
		\begin{align}		
			\label{eq:West} W^{\prime}(t) & \leq - \frac{2 \theta}{N} \left( \|w^{\perp}_{\maxindex}(n\tau)\| - \tau \left(a(0)\sqrt{N}\sqrt{W(n\tau)} + \theta \right) - \tau^2 a(0)\theta\right) \\		
			\label{eq:Vest} 	V^{\prime}(t) & \leq - \frac{2 \theta}{N} \left( \|v^{\perp}_{\maxindex}(n\tau)\| - \tau \left(a(0)\sqrt{N}\sqrt{V(n\tau)} + \theta \right) - \tau^2 a(0)\theta \right).
		\end{align}
		By the assumption on $\tau\leq \tau_0$ in \eqref{eq:CSCS_tausmall} and by  assumption \ref{item:JL3} we have
		\begin{align*}
		\tau \left(a(0)\sqrt{N}\sqrt{W(0)}+\theta\right) + \tau^2 a(0)\theta &\leq \tau \left(2a(0)\sqrt{N}\sqrt{V(0)}+\theta\right) + \tau^2 a(0)\theta \leq \frac{\Delta}{2}\leq \Delta. 
		\end{align*}	
		Applying this and the fact that $W$ is decreasing in $[0, n\tau]$, which follows from \ref{P1}, we use \eqref{eq:West} to deduce the following upper bound
		\begin{align*}
			W^{\prime}(t) & \leq - \frac{2 \theta}{N} \left(\|w^{\perp}_{\maxindex}(n\tau)\| - \tau \left(a(0)\sqrt{N}\sqrt{W(n\tau)} + \theta \right) - \tau^2 2a(0)\theta \right)\\
			& \leq - \frac{2 \theta}{N} \left(\sqrt{W(n\tau)} - \tau \left(a(0)\sqrt{N}\sqrt{W(0)} + \theta \right) - \tau^2 2a(0)\theta \right)\\
			& \leq - \frac{2 \theta}{N} \left(\sqrt{W(n\tau)} - \Delta \right).
		\end{align*}		
		Since we assumed that $\sqrt{W(n\tau)} > 2 \Delta$, this shows that $W$ is decreasing on $[n\tau,(n+1)\tau]$. Additionally, using this former assumption we also can estimate
 		\begin{align*}			
			W^{\prime}(t) & \leq - \frac{2 \theta}{N} \left(\sqrt{W(n\tau)} - \Delta \right) \\
			&\leq - \frac{\theta}{N} \sqrt{W(n\tau)}\\
			& \leq - \frac{\theta}{N} \sqrt{W(t)}
		\end{align*}
		for all $t \in [n\tau,(n+1)\tau]$. Together with \ref{P1} this shows the stated assertion for $W'(t)$ in $P_1(n+1)$. 
		
		In order to conclude the statement of \ref{P1} for $V'(t)$  we need to take advantage of the estimates of the lower dimensional dynamics, of Proposition \ref{prop:diff_eq_est}, and Lemma \ref{le:CSCS_technical}. By assumption \ref{P3} it holds that
	\begin{align*}
\int_0^{n\tau} \sqrt{2V(s)} \ ds \leq \alpha.
\end{align*}		
Thus, by the choice of $\eps'$ in \eqref{eq:epssuperest} and the assumptions \ref{item:JL1}, \ref{item:JL2}, and \ref{item:JL3}, the hypotheses of Proposition \ref{prop:diff_eq_est} are fulfilled in the interval $[0, n\tau]$ - since $n\tau \leq T_0$ by definition of $T_0$ as the time where we switch the control to $0$. Hence \eqref{eq:est_contr} holds and it follows
\begin{align*} 
		\|Mv^{\perp}_{i}(n\tau) - w^{\perp}_{i}(n\tau)\| &\leq \|Mv_{i}(n\tau) - w_{i}(n\tau)\|+\frac{1}{N} \sum_{j=1}^N \|Mv_{j}(n\tau) - w_{j}(n\tau)\|\leq 2\cal{E}^{v}(n\tau) \leq \Delta.
\end{align*}
This estimate and  assumption \ref{item:JL2} allow us to use Lemma \ref{le:CSCS_technical} for the vectors $a_i = v_i^{\perp}(n\tau)$ and $b_i = w^{\perp}_i(n\tau)$. Together with $\sqrt{W(n\tau)} > 2\Delta$ this results in
		\begin{align*}
		  \|v^{\perp}_{\maxindex}(n\tau)\| \geq \frac{1}{4}\|w^{\perp}_{\maxindex}(n\tau)\|, \qquad \|v^{\perp}_{\maxindex}(n\tau)\| \geq \frac{1}{4} \sqrt{W(n\tau)} \geq \frac{\Delta}{2}, \qquad 
		  \|v^{\perp}_{\maxindex}(n\tau)\| \geq c \sqrt{V(n\tau)}.
		\end{align*}
By assumption \ref{P1} we know that $V$ is decreasing in $[0,n\tau]$. Using the estimate \eqref{eq:Vest} together with the choice of $\tau\leq \tau_0$ in \eqref{eq:CSCS_tausmall} we obtain 
			\begin{align*}
			\begin{split}
			V^{\prime}(t) & \leq - \frac{2 \theta}{N} \left( \|v^{\perp}_{\maxindex}(n\tau)\| - \tau \left(a(0)\sqrt{N}\sqrt{V(n\tau)} + \theta \right) - \tau^2 a(0)\theta \right) \\
			& \leq - \frac{2 \theta}{N} \left( \|v^{\perp}_{\maxindex}(n\tau)\| - \tau \left(a(0)\sqrt{N})\sqrt{V(0)} + \theta \right) - \tau^2 a(0)\theta \right) \\
			& \leq - \frac{2\theta}{N} \left( \|v^{\perp}_{\maxindex}(n\tau)\|  - \frac{\Delta}{4} \right) \\
			& \leq - \frac{\theta}{N}  \|v^{\perp}_{\maxindex}(n\tau)\| \\
			& \leq - \frac{c \theta}{N} \sqrt{V(n\tau)}
			\end{split}
		\end{align*}
		for all $t \in [n\tau,(n+1)\tau]$. This shows that also $V$ is decreasing in $[n\tau,(n+1)\tau]$ and hence
		\begin{align*}	
			V^{\prime}(t) \leq - \frac{c \theta}{N} \sqrt{V(n\tau)}\leq - \frac{c \theta}{N} \sqrt{V(t)}
		\end{align*}
		for all $t\in [n\tau,(n+1)\tau]$. Together with \ref{P1} this finishes the proof of $P_1(n+1)$.
		
		We can now use Lemma \ref{le:tech_up} with $\eta= \frac{\theta}{N}$ and \ $\eta= \frac{c \theta}{N}$ to get the following estimates for $Y(t)$ and \ $X(t)$, respectively,
		\begin{align*}
			Y(t) \leq \overline{Y} \mbox{ and } X(t) \leq \overline{X} \text{ for all } t \in [0, (n+1)\tau]
		\end{align*}
	     with $\overline{X}$ as defined in \eqref{eq:Xbar} and $\overline{Y}$ as defined in \eqref{eq:Ybar}. This shows $P_2(n+1)$. Furthermore, $P_1(n+1)$ yields by integration 
	       \begin{align*}
\int_0^{(n+1)\tau} \sqrt{2V(s)} \ ds \leq -\frac{\sqrt{2}N}{c\theta} \int_0^{(n+1)\tau} V'(s) \ ds = \frac{\sqrt{2}N}{c\theta}  \left(V(0)-V((n+1)\tau)\right) \leq \alpha V(0)
\end{align*}
with $\alpha=\frac{\sqrt{2}N}{c\theta}$. Hence, under the assumptions \ref{P1}, \ref{P2}, and \ref{P3} we have shown $P_1(n+1)$, $P_2(n+1)$ as well as $P_3(n+1)$, provided that $\sqrt{W(m\tau)} > 2 \Delta$ for every $m = 0, \ldots, n$, and thus completed the first step.

\textit{Second step:} In the second step we shall prove that there exists an $n^* \in \N_0$ such that $n^*\tau \leq \hat{T}+\tau$ and $\sqrt{W(n^*\tau)} \leq 2 \Delta$ holds, where $\hat{T}$ is defined as in \eqref{eq:That}. By definition of the threshold $\Gamma=(2 \Delta)^2$, this implies the switching of the control to $0$ at time $n^*\tau$.  Assume on the contrary that 
\begin{align}
\label{eq:contra}
 \sqrt{W((n+1)\tau)} > 2 \Delta
\end{align}
for all $n \in \N_0$ with $n\tau \leq \hat{T}$. In the first step we showed that this yields in particular for $t \in [0,(n+1)\tau)$ the estimates
\begin{align*}
	W^{\prime}(t) \leq - \frac{\theta}{N} \sqrt{W(t)}< 0 \text{ and } \sqrt{W(t)}>2\Delta.
\end{align*}
Hence for all $t \in [0,(n+1)\tau)$ it holds
\begin{align*}
\sqrt{W(t)} &\leq \sqrt{W(0)} + t \cdot \sup_{\xi \in (0,(n+1)\tau)} \left(\sqrt{W}\right)'(\xi)=\sqrt{W(0)} + t \cdot \sup_{\xi \in (0,(n+1)\tau)} \frac{W'(\xi)}{2\sqrt{W(\xi)}} \\
&\leq \sqrt{W(0)} - t \cdot \frac{\theta}{2N}.
\end{align*}
Taking $n_0 \in  \N_0$ such that $n_0\tau \leq \hat{T}<(n_0+1)\tau$ and using \ref{item:JL3} we have
\begin{align}
\label{eq:WTime}
\begin{split}
\sqrt{W((n_0+1)\tau)} \leq \sqrt{W(\hat{T})} \leq \sqrt{W(0)}- \hat{T}\cdot \frac{\theta}{2N}  &= \sqrt{W(0)}-  \frac{2N}{\theta} \left( 2\sqrt{V(0)} -2 \Delta\right) \frac{\theta}{2N} \\
&\leq \sqrt{W(0)}-  \frac{2N}{\theta} \left(\sqrt{W(0)} -2 \Delta\right) \frac{\theta}{2N} \\
&\leq 2\Delta.
\end{split}
\end{align}
This contradicts assumption \eqref{eq:contra}. Thus there exists an $n^* \in \N_0$ such that $n^*\tau \leq \hat{T}+\tau$ and it holds
\begin{align}
\label{eq:consesusRegionBoth}
			\sqrt{W(n^*\tau)} \leq 2 \Delta.
		\end{align}
		
\textit{Third step:} We shall show that \eqref{eq:consesusRegionBoth} implies that the trajectories of both the low- and high-dimensional systems are in the consensus region identified by Lemma \ref{le:HaHaKim} at time $n^*\tau$, i.e.,
	\begin{align*}
		\sqrt{W(n^*\tau)} \leq \gamma(Y(n^*\tau)) \quad \text{ and } \quad \sqrt{V(n^*\tau)} \leq \gamma(X(n^*\tau)).
	\end{align*}
We shall start considering the low-dimensional system. Since by  $\ref{item:JL3}$ it holds
\begin{align*}
Y(0) \leq (1+1/2)^2 \cdot X(0) \leq 4X(0), \quad W(0) \leq 4 V(0)
\end{align*}
and by the fact that the constant $c$ from Lemma \ref{le:CSCS_technical} is smaller than $1$, we can estimate $\overline{Y}=2Y(0) + \frac{2 N^2}{\theta^2}W(0)^2$ from below by $4\overline{X}$, where $\overline{X}= 2X(0) + \frac{2 N^2}{c^2 \theta^2}V(0)^2$. This together with \eqref{eq:consesusRegionBoth}, the definition of $\Delta$ in \eqref{eq:Delta}, and $P_2(n^*)$ lead to
	\begin{align*}
		\sqrt{W(n^*\tau)} \leq 2\Delta \leq \gamma\left(4\overline{X}\right) \leq \gamma(\overline{Y}) \leq \gamma(Y(n^*\tau)).
		\end{align*}
It remains to prove that the high-dimensional system is in the consensus region identified by Lemma \ref{le:HaHaKim}. Again, the conditions of Lemma \ref{le:CSCS_technical} for the vectors $a_i = v_i^{\perp}(n^*\tau)$ and $b_i = w^{\perp}_i(n^*\tau)$ are fulfilled: as in the first step we have by Proposition \ref{prop:diff_eq_est}
\begin{align*}
\|Mv^{\perp}_{i}(n^*\tau) - w^{\perp}_{i}(n^*\tau)\| &\leq \Delta,  
\end{align*}
and  property $\ref{item:JL2}$ holds at $n^*\tau$. Thus, an application of Lemma \ref{le:CSCS_technical} shows
\begin{align*}
\sqrt{V(n^*\tau)} \leq C \Delta.
\end{align*}		
Hence the definition of $\Delta$ in \eqref{eq:Delta} and $P_2(n^*)$ yield
\begin{align*}
\sqrt{V(n^*\tau)} \leq C \Delta \leq \gamma(\overline{X}) \leq \gamma(X(n^*\tau)).
\end{align*}
We conclude that both the trajectories of the systems are in the consensus region at time $n^*\tau$. By Lemma \ref{le:HaHaKim} we are allowed to switch the control to $0$ and both systems tend to consensus autonomously.

\textit{Fourth step:} In the second and third steps we have proven that both systems enters the consensus region at time $T_0=n^*\tau$, where $n^*\tau \leq \hat{T}+\tau$. By the computations in \eqref{eq:WTime}, we have the following estimate 
\begin{align*}
 		T_0\leq \frac{2N}{\theta} \left( \sqrt{W(0)} - 2\Delta\right)+\tau\leq \hat{T}+\tau.
\end{align*}
Moreover, by $P_2(n^*)$ we have 
\begin{align*}
\|x_i(t) - x_j(t)\|^2 & \leq 2\|x_i(t) - \overline{x}(t)\|^2 + 2\|x_j(t) - \overline{x}(t)\|^2  \leq 4 N X(t) \leq 4 N \overline{X} \quad \text{for} \quad t \in [0,n^*\tau],
\end{align*}
and from $P_1(n^*)$ it follows
\begin{align*}
\|v_i(t) - v_j(t)\|^2 & \leq 2\|v_i(t) - \overline{v}(t)\|^2 + 2\|v_j(t) - \overline{v}(t)\|^2  \leq 4 N V(t) \leq 4 N V(0) \quad \text{for} \quad t \in [0,n^*\tau].
\end{align*}
This shows \eqref{eq:M} and the proof is concluded.
\end{proof}
	

\section{How to find a Johnson--Lindenstrauss matrix}
\label{sec:JLL_matrix}

The main ingredient of Proposition \ref{prop:diff_eq_est} and Theorem \ref{eq:convergence} is the existence of a Johnson--Lindenstrauss matrix $M\in \R^{k\times d}$ for the trajectories. Let $\Delta$ and $\eps'$ be as in Theorem \ref{eq:convergence} and let us recall what we explicitly needed. Assume that $\hat{T}$ is an upper estimate for $T_0$, the time to switch off the control. Then we need to define a matrix $M\in \R^{k\times d}$ such that the following properties hold:
\begin{enumerate}[label=(JL\arabic*)]
\item
\label{it:JLLatt1}
Let $\eps=\eps'$ and $\delta = \frac{\sqrt{2}}{2} \eps'\left(\sqrt{X(0)}+\frac{N}{c\theta}\right)$. For all $t \in [0,\hat{T}+\tau]$ and $i,j\in \{1,\ldots,N\}$ either we have
\begin{align*} 
(1-\eps) \|x_i(t)-x_j(t)\| \leq \|M(x_i(t)-x_j(t))\| \leq (1+\eps) \|x_i(t)-x_j(t)\|
\end{align*} 
or
\begin{align*}
\|x_i(t)-x_j(t)\| \leq \delta \text{ and } \|M(x_i(t)-x_j(t))\| \leq \delta.
\end{align*}
\item
\label{it:JLLatt2}
Let $\eps=\eps'$ and $\delta=\Delta$. For all $n=0,\ldots,\lfloor \frac{T}{\tau}\rfloor+1 $ and $i\in \{1,\ldots,N\}$ either we have
\begin{align*} 
(1-\eps) \|v_i(n\tau)-\overline{v}(n\tau)\| \leq \|M(v_i(n\tau)-\overline{v}(n\tau))\| \leq (1+\eps) \|v_i(n\tau)-\overline{v}(n\tau)\|
\end{align*} 
or
\begin{align*}
\|v_i(n\tau)-\overline{v}(n\tau)\| \leq \delta \text{ and } \|M(v_i(n\tau)-\overline{v}(n\tau))\| \leq \delta.
\end{align*}
\item
\label{it:JLLat0}
Let $\eps=1/2$. Then for all $i \in \{1,\ldots,N\}$ we have
\begin{align*}
(1-\eps) \|v_i(0)-\overline{v}(0)\| \leq \|M(v_i(0)-\overline{v}(0))\| \leq (1+\eps) \|v_i(0)-\overline{v}(0)\|
\end{align*}
and 
\begin{align*}
(1-\eps) \|x_i(0)-\overline{x}(0)\| \leq \|M(x_i(0)-\overline{x}(0))\| \leq (1+\eps) \|x_i(0)-\overline{x}(0)\|.
\end{align*}
\end{enumerate}
In order to prove conditions \ref{it:JLLatt2} and \ref{it:JLLat0} one can directly invoke the Johnson--Lindenstrauss Lemma as discussed in Remark \ref{rm:numberofpoints} while for \ref{it:JLLatt1} one can use its continuous version, Lemma \ref{le:JLL_cont}, which boils down again to the application of the Johnson--Lindenstrauss Lemma on points sampled from the trajectories.

However, the Johnson--Lindenstrauss Lemma applies on points which are fixed a priori before generating randomly the matrix $M \in \R^{k \times d}$. At a first look, due to the fact that the high-dimensional controls depend on the low-dimensional ones, which depend on the matrix $M$, the points on which we apply the Johnson--Lindenstrauss Lemma may be seen as directly depending on $M$ as well. 

In order to resolve this apparent paradox, we want to clarify that actually, due to the finite number of sampling times of the control and the finite number of agents, the number of possible realizable trajectories, and consequently the number of possible sampling points for the Johnson--Lindenstrauss Lemma, is finite and, actually, \emph{independent} of the choice of the matrix $M$. Hence we are now left with the tasks of counting the number of such trajectories and of verifying that they fulfill the necessary Lipschitz continuity assumptions for applying Lemma \ref{le:JLL_cont}.

Let us state again that the lower dimension $k$ of $M \in \R^{k \times d}$ scales as
\begin{align}\label{kJL}
k \sim \eps^{-2} \log(\cal{N}),
\end{align}
where $\eps \in (0,1)$ is the allowed distortion and $\cal{N}$ is the number of sampling points on all possible trajectories.

We focus first in \eqref{kJL} on the dependence of $\eps = \min \{ \eps', \frac{1}{2} \}$ on $N$, the number of agents, and the dimension $d$. According to \eqref{eq:epssuperest} in Theorem \ref{eq:convergence} the estimate on $\eps'$ scales exponentially with $N$, i.e., $\eps' \lesssim e^{-N}$, since $\hat{T}$ scales (at least) linearly with $N$, see \eqref{eq:That} (for $\theta$ independent of $N$ and $d$).

The positive aspect is that the estimate for $\eps'$ does not involve the dimension $d$.

\vspace{0.5cm}

In order to compute $\cal{N}$ in \eqref{kJL} we need first of all to estimate the number of realizable trajectories. Since we are insisting on {\it sparse} controls acting at most on {\it one} agent at the time, at every switching time $n\tau$ with $n\tau \leq T_0$, i.e., as long as the control is not switched off, there are precisely only $N$ possible controls and hence $N$ possible branches of future developments of the trajectories. By Theorem \ref{eq:convergence} it holds $T_0\leq \hat{T}+\tau$ and thus we can estimate the number $P$ of possible paths by
\begin{align*}
 P \leq N^{\lfloor \frac{T}{\tau}\rfloor+1}.
\end{align*}
Surprisingly, accounting for all the possible future branching is  sufficient to show that actually we can already deterministically fix points a priory on which later apply an independently randomly drawn matrix!

\begin{enumerate}
\item In order to fulfill \ref{it:JLLatt1} for every possible trajectory, an application of Lemma \ref{le:JLL_cont} yields an estimate of the number of necessary sampling points 
\begin{align}
\label{eq:count}
\cal{N}_1 = P \cdot (\hat{T}+\tau) \cdot \binom{N}{2} \cdot 4 \cdot \frac{L_x \cdot (\sqrt{d}+2)}{\delta \eps},
\end{align}
where the factor $P \cdot (\hat{T}+\tau)$ accounts for the number of trajectories and their time length, the factor $\binom{N}{2}$ accounts for the number of space trajectory differences $x_i - x_j$, and $L_x$ is an upper estimate for the individual Lipschitz constant, given by an estimate similar to \eqref{eq:lipschitzX} and the result from Theorem \ref{eq:convergence} that $V$ is decreasing until $T_0$ as follows:
\begin{align*}
L_x=\max\left( L_{x_i-x_j}\left(0,T_0\right): i,j\in \{1,\ldots,N\}\right) \leq \sup_{t \in \left(0,T_0\right)} \sqrt{2NV(t)} \leq \sqrt{2NV(0)}.
\end{align*}
\item To fulfill \ref{it:JLLatt2} we shall now count the necessary sampling points at every switching time $n\tau$. For $n = 0$ we have to consider $N$ sampling points. For $n = 1$ there are already $N$ possible paths to take into account and hence we need to take $N^2=N \cdot N$ sampling points. 
Going on in this way, at time $n\tau$ we have $N^n$ possible outcomes of the dynamical system and hence we have to take $N^{n+1}=N^n \cdot N$ sampling points, as long as $n\tau \leq \hat{T}+\tau$. Summing up the number of sampling points, we conclude
\begin{align*}
\cal{N}_2=\sum_{n=0}^{\lfloor \frac{T}{\tau}\rfloor+1} N^{n+1} \leq N^{\lfloor \frac{T}{\tau}\rfloor+3}
\end{align*}
\item To fulfill \ref{it:JLLat0} we need only $\cal{N}_3=2N$ sampling points.
\end{enumerate}

Hence we can eventually estimate $\cal{N}$ from above by
\begin{align*}
\cal{N} \leq \cal{N}_1+\cal{N}_2+\cal{N}_3 = N^{\lfloor \frac{T}{\tau}\rfloor+1} \cdot (\hat{T}+\tau) \cdot \binom{N}{2} \cdot 4 \cdot \frac{\sqrt{2NV(0)} \cdot (\sqrt{d}+2)}{\delta \eps} + N^{\lfloor \frac{T}{\tau}\rfloor+3} +  2N  
\end{align*} 

Thus, we can choose the dimension $k$ of a Johnson--Lindenstrauss matrix $M \in \R^{k \times d}$ as
\begin{align*}
k \sim \eps^{-2} \cdot \log(\cal{N}) &\sim \eps^{-2} \left[ \left(\frac{\hat{T}}{\tau}+1\right) \cdot \log{N} + \log(\hat{T}+\tau) + \log d + \log V(0) + |\log(\delta \eps)| \right],
\end{align*}
where
\begin{align*}
\eps=\min\left(\eps',1/2\right) \quad \text{and} \quad \delta=\min\left(\Delta,\frac{\sqrt{2}}{2} \eps'\left(\sqrt{X(0)}+\frac{N}{c\theta}\right)\right).
\end{align*}
Since the estimate on $\eps'$ scales exponentially in $N$, i.e., $\eps' \lesssim e^{-N}$, the dimension $k$ grows exponentially in $N$. However, the positive aspect is that the estimate of $k$ only scales logarithmically with the dimension $d$. Hence we have shown that at least for very large dimension $d \gg1$ and relatively small number of agents $N$ our dimensionality reduction approach will pay-off. As we show in Section \ref{sec:numerics}, these theoretical bounds turn out to be by far over-pessimistic and, surprisingly, this method of dimensionality reduction for computing optimal controls can work effectively with lower dimensions $k$ conspicuously smaller than $d$. Moreover, we show below ways to circumvent the exponential dependency of $k$ with respect to $N$ at the cost of using sequences of Johnson--Lindenstrauss matrices, see Remark \ref{1wout} and  Remark \ref{2wout}.

\begin{Remark}
The $\log(d)$-dependency only comes into play when we derive \ref{it:JLLatt1} from Lemma \ref{le:JLL_cont}. One can actually use a similar argument as in Remark \ref{rem:noDinvolved} in order to get rid of this logarithmic dependency. We do not elaborate further on this issue which appears to us just a mere and perhaps unnecessary technicality at this point.
\end{Remark}

\begin{Remark}\label{1wout} 
We observed that at least in the worst-case scenario here considered, the dimension $k$ of the Johnson--Lindenstrauss matrix is blowing up exponentially with the number of agents $N$. A practical approach to circumvent this problem is to use not only one but a whole family of matrices $M_0,\ldots,M_{\ell}$. The matrix $M_0$ is used from time $0$ up to a certain time $t_0$ and thus only needs to fulfill the Johnson--Lindenstrauss property in this short time interval. At time $t_0$ a new matrix $M_1$ is chosen. We have to observe the positions as well as the consensus parameters in high-dimension and project the system  to low-dimension again, using $M_1$, at $t_0$. Then we use the new low-dimensional system to calculate the index of the control for the high-dimensional system from time $t_0$ up to time $t_1$, eventually we again repeat the procedure with a new matrix $M_2$ etc.

This approach has the advantage that it requires the Johnson--Lindenstrauss properties for $M_i$, $i=1,\ldots,\ell$, only for a short time interval. The disadvantage is that we have to observe the high-dimensional system and project it to low-dimension again at every time $t_i$, $i=0,\ldots,\ell-1$.
\end{Remark}

\begin{Remark}\label{2wout} 
There is additionally the possibility to get rid of the mutual dependency of the matrix and the points of the trajectories using another family of matrices.

First, we take a matrix $M_0$ having the Johnson--Lindenstrauss properties \ref{it:JLLatt2} at $t=0$ and \ref{it:JLLat0}. We compute the index $i_0$ of the control (as defined in Definition \ref{def:control}) at $t=0$ using the projection $M_0$. 

Then we choose a matrix $M_1$ having the Johnson--Lindenstrauss properties \ref{it:JLLatt1} for all $t \in [0,\tau)$, \ref{it:JLLatt2} at $t=\tau$, and \ref{it:JLLat0}. We compute the low-dimensional system using the projection $M_1$ in $[0,\tau]$ and let the control act on the agent $i_0$ calculated by $M_0$. This is the main trick of the procedure: The points of the high-dimensional system in $[0,\tau]$ are not influenced by the matrix $M_1$ and hence the mutual dependency is removed, which means that there is no need of considering all trajectories $P$ anymore, in contrast to \eqref{eq:count}. 

Now, from the low-dimensional system, computed by $M_1$ with control acting on $i_0$ in $[0,\tau]$, we choose the agent $i_1$ at $\tau$ on which the control will act in the next interval $[\tau,2\tau]$. 

This procedure can be carried on using a family of matrices $\{M_p: p=0,\ldots,\ell\}$ fulfilling the Johnson--Lindenstrauss properties \ref{it:JLLatt1} for all $t \in [0,p\tau)$, \ref{it:JLLatt2} at $t=p\tau$, and \ref{it:JLLat0}. The agent $i_p$ on which the control shall act in the interval $[p\tau,(p+1)\tau)$ is computed at $p\tau$ using the low-dimensional system projected by $M_p$, while the control acts on $j_{q}$ in $[q\tau,(q+1)\tau)$ for $q=0,\ldots p-1$. Therefore, in $[0,p\tau]$  the index of the controlled agent and hence the trajectories of the high-dimensional system are independent of $M_p$.
\end{Remark}

\section{Numerical experiments} \label{sec:numerics}
In the following section we shall present some numerical experiments to confirm the theoretical observation of the interplay between the Cucker--Smale system, the dimension reduction by a Johnson--Lindenstrauss matrix and the quality of the control chosen from the low-dimensional (projected) system as defined in Definition \ref{def:control}.

For every $\ell=0,1,\ldots$ we recursively solve the Cucker--Smale system with $x_1,\ldots,x_N \in \R^d$ and $v_1,\ldots,v_N \in \R^d$
\begin{align*} 
\left \{\begin{array}{lll}
\dot{x}_{i}(t) & = & v_{i}(t)\\
\dot{v}_{i}(t) & = & \frac{1}{N}\underset{j=1}{\overset{N}{\sum}}\frac{v_{j}(t)-v_{i}(t)}{\left(1+\left\| x_{i}(t)-x_{j}(t)\right\|^{2}\right)^{\beta}} + u_i(\ell\tau), \quad 
\end{array} t \in [\ell\tau,(\ell+1)\tau], \quad i=\{1,\ldots,N\} 
\right .
\end{align*}
numerically, using as the initial value $(x(\ell\tau),v(\ell\tau))$ the solution of the preceding interval $[(\ell-1)\tau,\ell\tau]$ and as the starting value for $\ell=0$ the given values $x(0)=x_0$ and $v(0)=v_0$. The experiments are implemented by using Matlab applying a Runge-Kutta method of order 4 solving the systems of ODEs with step width $\tau$. The following are the different control strategies $u_i$ we compare in our experiments:

\begin{itemize}
\item[(SP)] Sparse control implemented in the high-dimensional system: this is the sparse control strategy outlined in \cite[Definition 4]{CFPT14}. The control acts on the agent with consensus parameter farthest away from the mean consensus parameter as long as we are not in the consensus region given by Lemma \ref{le:HaHaKim}: for every $\ell \in \mathbb{N}_0$ let $\maxindex \in \{1,\ldots,N\}$ be the smallest index such that
\begin{align*}
\|v^{\perp}_{\maxindex}(\ell\tau)\| = \max_{1 \leq i \leq N} \|v^{\perp}_i(\ell\tau)\|,
\end{align*}
and define the control as
\begin{align*}
u_{\maxindex}(\ell\tau)= -\theta \frac{v^{\perp}_{\maxindex}(\ell\tau)}{\|v^{\perp}_{\maxindex}(\ell\tau)\|} \quad \text{ and } \quad u_i(\ell\tau) = 0 \quad\text{for every}\quad i \not = {\maxindex} \ 
\end{align*}
as long as $V(\ell\tau)>\gamma(X(\ell\tau))^2$. As soon as $V(n\tau)\leq\gamma(X(n\tau))^2$ is satisfied for some $n \in \mathbb{N}_0$, we set $T_0 := n\tau$ and the control to zero.

This control was shown to be {\it optimal} in the work \cite[Section 4]{CFPT14} in terms of maximizing the rate of convergence to the consensus region, and shall be therefore employed as a benchmark to test the effectiveness of the other controls.
\item[(U)] Uniform control: this control strategy acts on every agent simultaneously using a control pointing towards the mean consensus parameter with norm equal to $\theta/N$ as long as $V(\ell\tau)>\gamma(X(\ell\tau))^2$ . This means
 \begin{align*}
u_j(\ell\tau)= - \frac{\theta}{N} \frac{v^{\perp}_j(\ell\tau)}{\|v^{\perp}_j(\ell\tau)\|} \quad\text{for all}\quad j \in \{1,\ldots,N\}. 
\end{align*} 
Again, as soon as $V(n\tau)\leq\gamma(X(n\tau))^2$ is satisfied for some $n \in \mathbb{N}_0$, we set $T_0 := n\tau$ and the control to zero.
\item[(R)] Random sparse control: as long as $V(\ell\tau)>\gamma(X(\ell\tau))^2$, at every sampling time $\ell\tau$ we choose an index $j \in \{1,\ldots,N\}$ at random following a uniform distribution. Then we define the control as
\begin{align*}
u_j(\ell\tau)= -\theta \frac{v^{\perp}_j(\ell\tau)}{\|v^{\perp}_j(\ell\tau)\|} \quad \text{and} \quad u_i(\ell\tau) = 0 \quad\text{for every}\quad i \not = j .
\end{align*}
As in the above controls, as soon as $V(n\tau)\leq\gamma(X(n\tau))^2$ is satisfied for some $n \in \mathbb{N}_0$, we set $T_0 := n\tau$ and the control to zero.
\item[(DR)] Dimension reduction sparse control chosen by the low-dimensional projected system: here $u_i(\ell\tau)=u_i^h(\ell\tau)$ is defined as in Definition \ref{def:control}. In order to test the performance of this control, and to avoid the stability complications arising from finite precision approximation, we calculate the trajectories of both the high- and the low-dimensional system: if the high-dimensional system enters the consensus region first (i.e., $V(n\tau)\leq \gamma(X(n\tau))^2$ for some $n \in \mathbb{N}_0$), then we set the control to zero and $T_0 := n \tau$. Instead, if the system in low dimension reaches the consensus region first (i.e., $W(\ell\tau)\leq \gamma(Y(\ell\tau))^2$ for some $\ell \in \mathbb{N}_0$), then we switch the control for the high-dimensional system to the random sparse control strategy (R) until $V(n\tau)\leq \gamma(X(n\tau))^2$ is eventually satisfied for some $n \in \mathbb{N}_0$.
\end{itemize}
Notice that all the controls above are time sparse, and only the uniform control strategy (U) is not componentwise sparse.

\begin{Remark}
\label{rem:simplify}
The reasons for using random sparse control at the end phase of (DR) in the case that the low-dimensional system reaches the consensus region before the high-dimensional one are of numerical and computational nature. In fact, the step width $\tau$ computed in Theorem \ref{eq:convergence} to ensure convergence to the consensus region in finite time is often way too small, and in our numerical experiments we need to exceed it. Moreover, as soon as the high-dimensional system enters the consensus region, the difference between consensus parameters becomes so small to render, for such a large time step, the choice of the sparse control highly inaccurate, leading to inefficient chattering phenomena, without steering the high-dimensional system to consensus.

As an alternative, we employ the random sparse control as soon as the low-dimensional system has reached the consensus region (if this happens before the high-dimensional system does). This procedure has the advantage of always steering the system to the consensus region, and it only slightly affects the time that the high-dimensional system takes to reach the consensus region, since it is usually necessary for a very short time (provided that the dimension of the Johnson--Lindenstrauss matrix is sufficiently large).

\end{Remark}

\subsection{Content of the Numerics}
The following are the driving issues concerning the controls introduced above:
\begin{enumerate}
\item Does the control steer the system to the consensus region as defined in Definition \ref{def:consensus} in finite time?
\item How long does it take to steer the system to the consensus region? 
\end{enumerate}

In the following, for every experiment we fix the number of agents $N$, the dimension $d$, the control strength $\theta$, the power of the interaction kernel $\beta$ as in \eqref{csmod}, the step width $\tau$, and in particular the configuration $(x_0, v_0)$ at the beginning. We report the maximal step width $\tau_0$ (theoretically) allowed by formula \eqref{eq:CSCS_tausmall}, and the estimate from above for the time to consensus $\hat{T}$ (taken from Theorem \ref{eq:convergence}). 
We also report the quantity $V(0)-\gamma(X(0))^2$, accounting for the discrepancy of the original configuration from the consensus region.

For every configuration we shall present 
a table containing the performances of the different controls, measured by the time employed by the high-dimensional system to reach the consensus region $T_0$, and the time  $T_{0.5}$ it takes to halve the ``distance'' to the consensus region: this means that $T_{0.5}$ is the minimal time satisfying
\begin{align*}
V(T_{0.5})-\gamma(X(T_{0.5}))^2 \leq 1/2 \cdot \left( V(0)-\gamma(X(0))^2\right).
\end{align*}

To test the performances of the control (DR) we shall use a variety of Bernoulli random matrices $M \in \R^{k \times d}$ for different dimensions $k$. For any of these dimensions, we also report the initial discrepancy $W(0)-\gamma(Y(0))^2$ from the consensus region of the projected system, and the switching time $T_S$ at which the random sparse control replaces the original dimension reduction control strategy (if the high-dimensional system enters the consensus region before the low-dimensional one, we set $T_S := T_0$).

Figure \ref{inconf} shows the first two coordinates of the initial configurations used in each section.

\begin{figure}[H]
\centering
\includegraphics[scale=0.23]{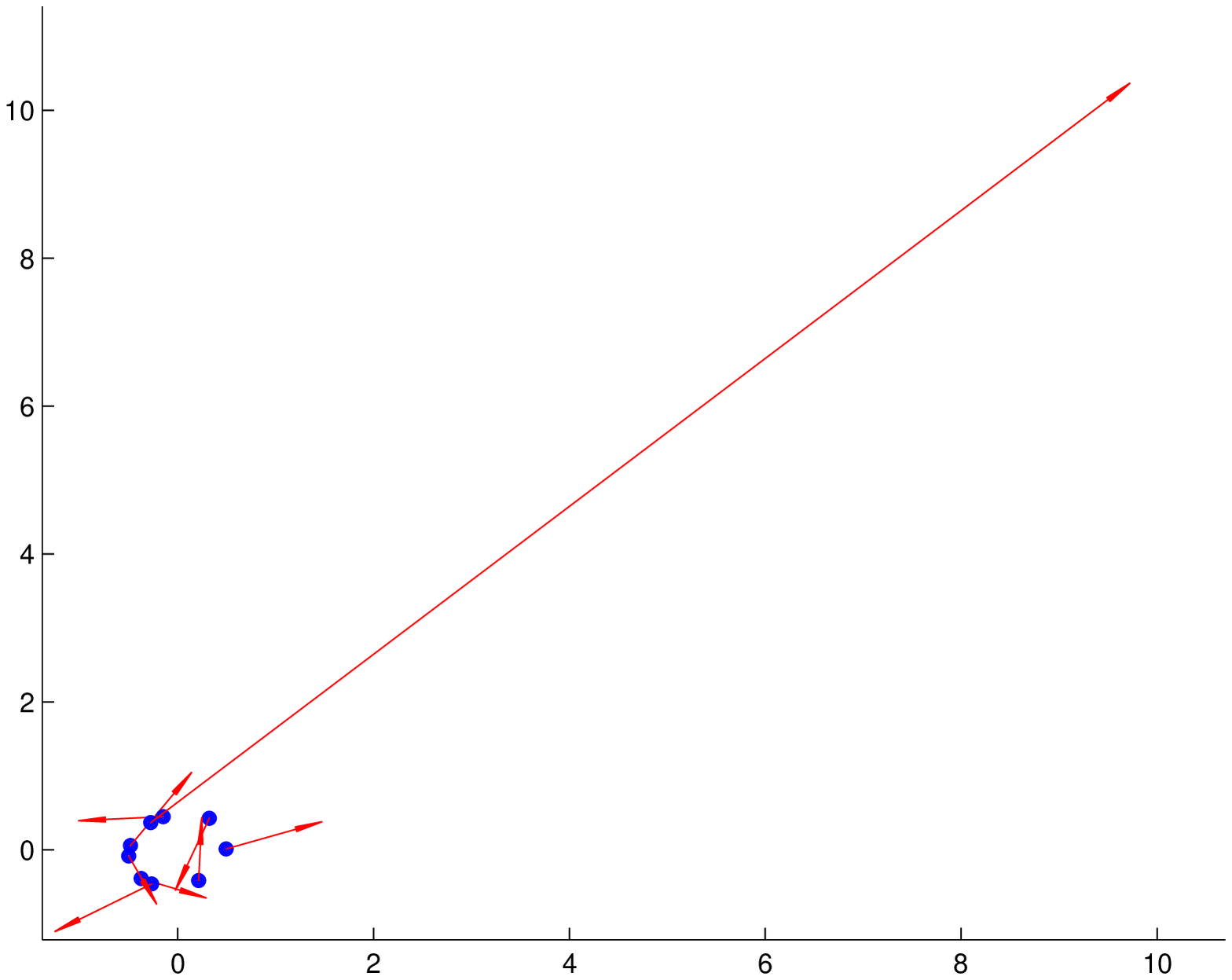}
\includegraphics[scale=0.23]{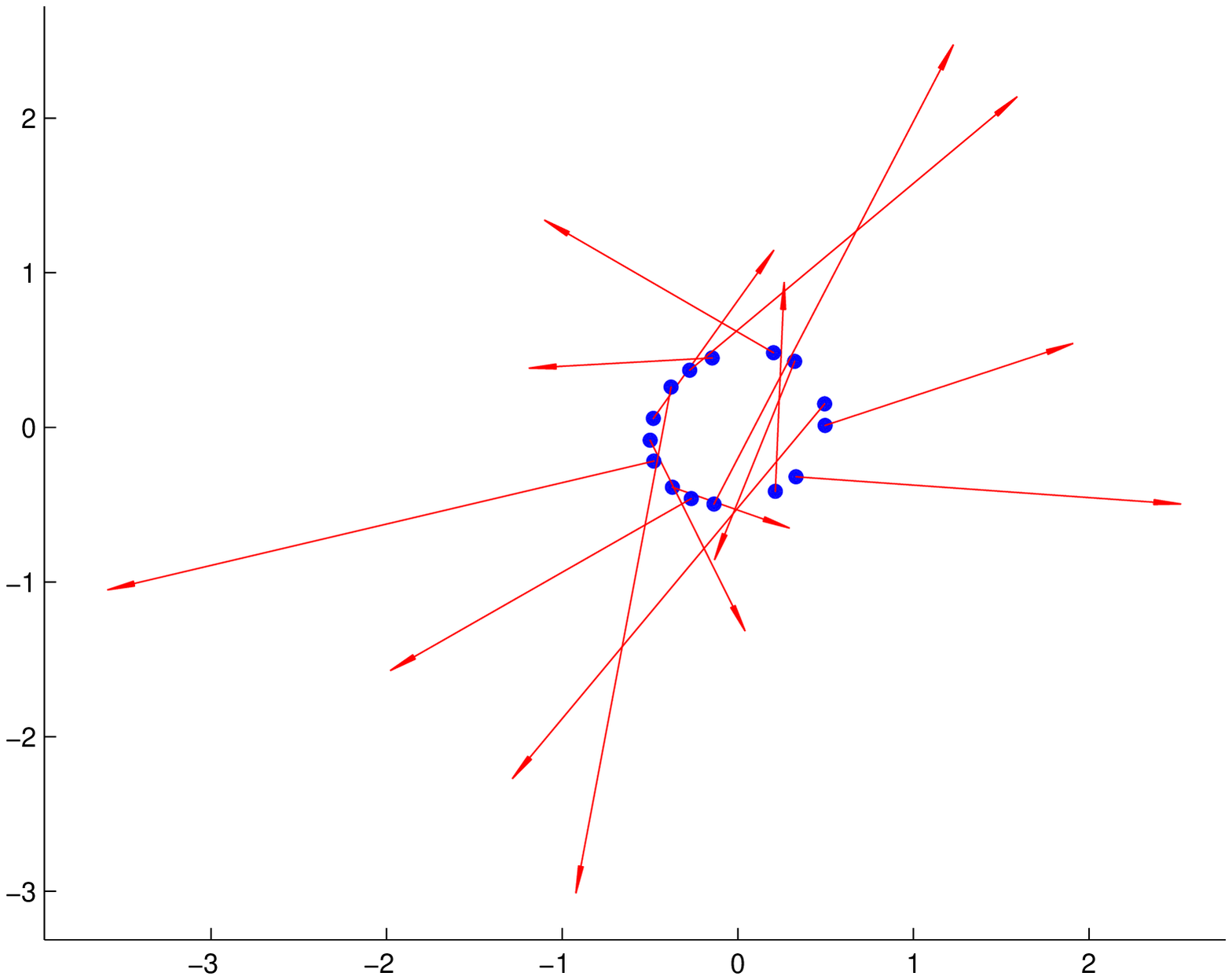}
\includegraphics[scale=0.23]{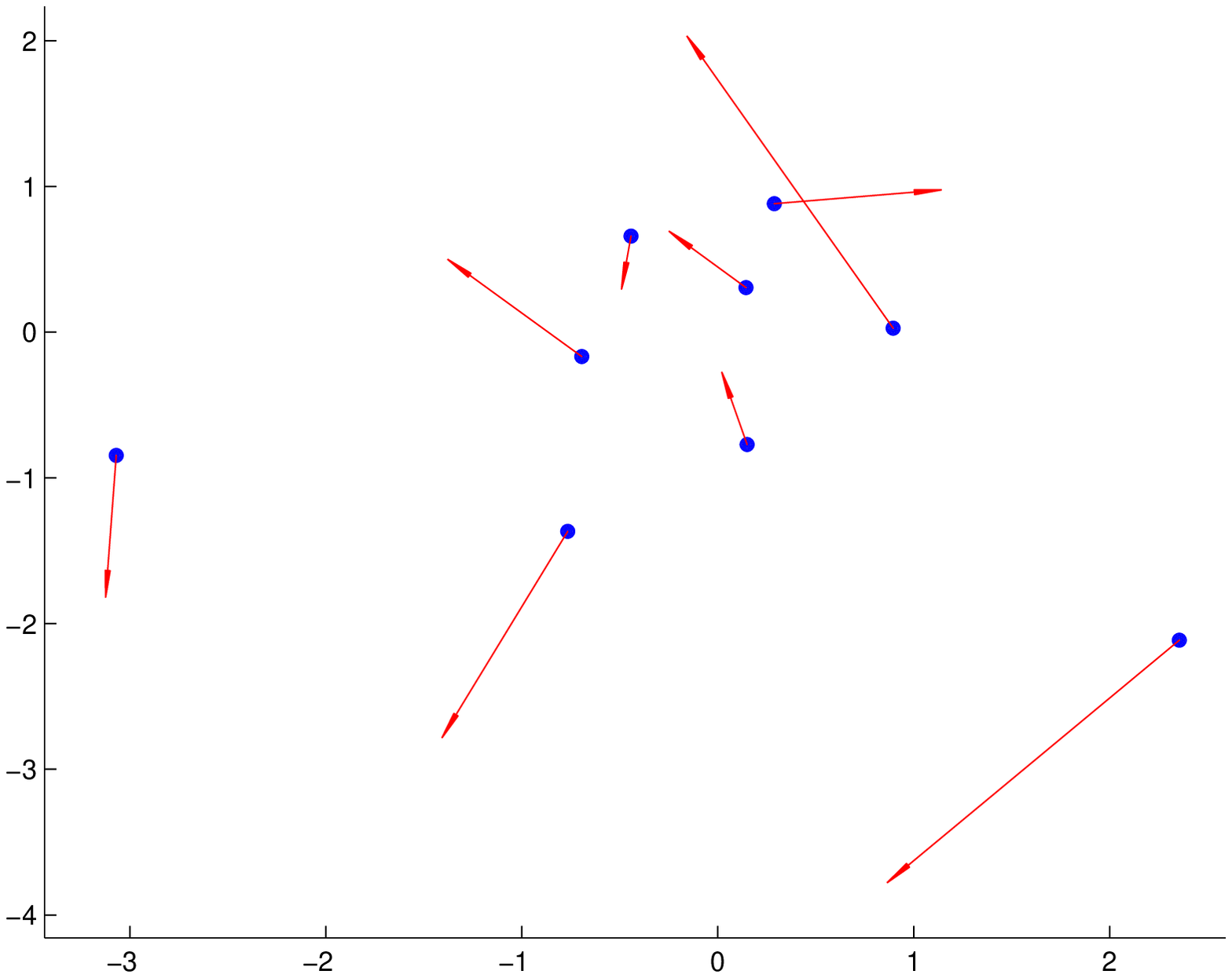}
\includegraphics[scale=0.23]{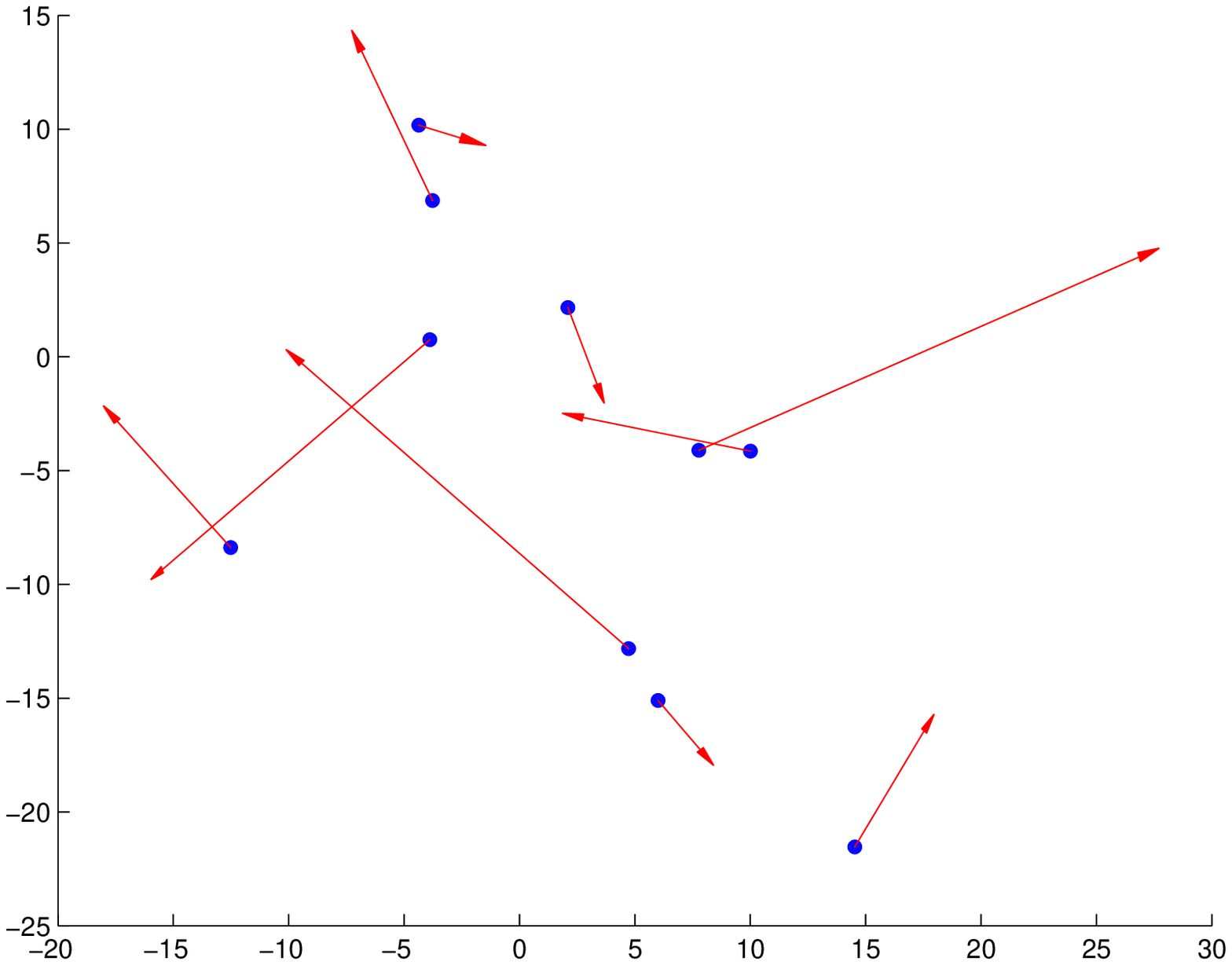}
\includegraphics[scale=0.23]{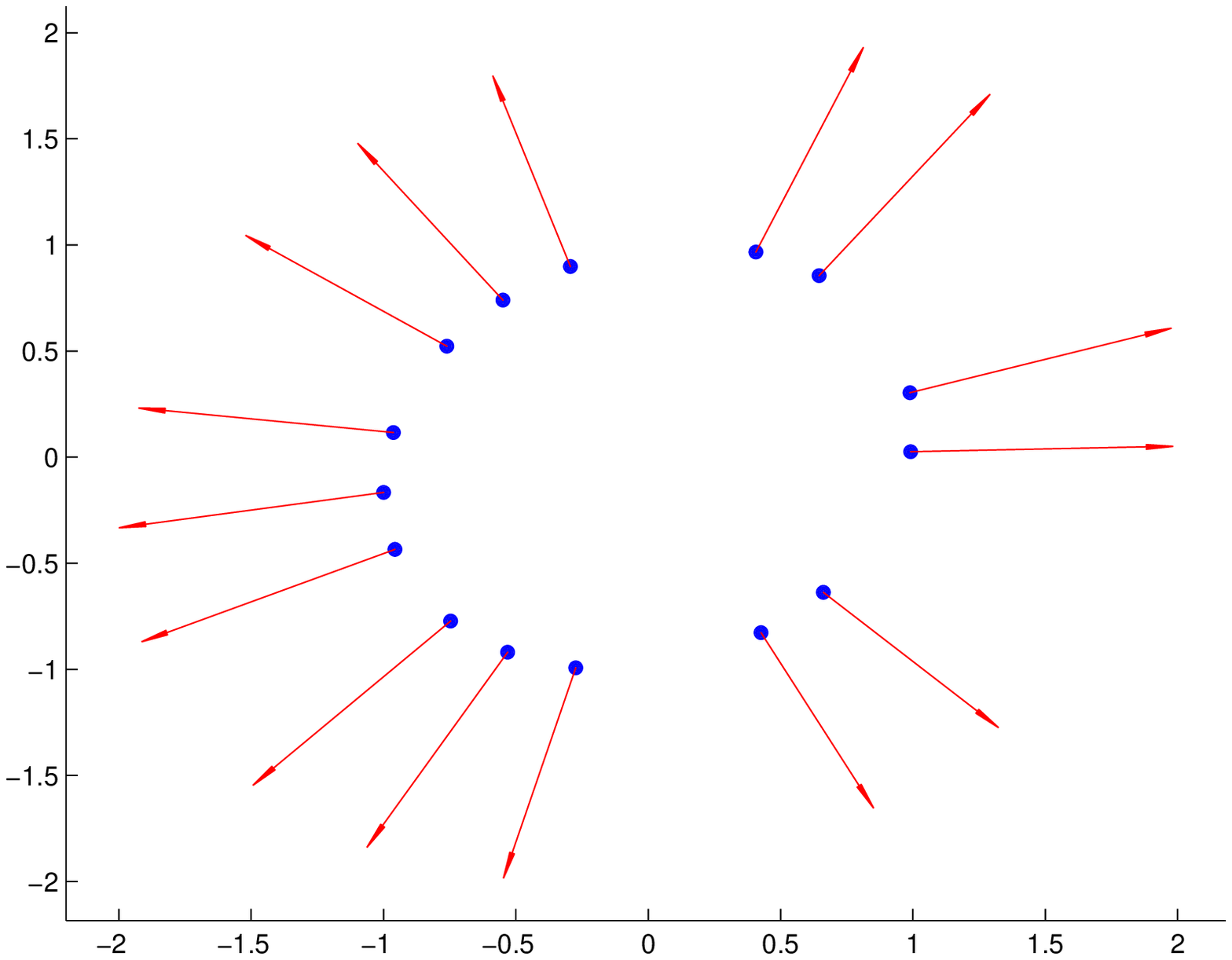}
\caption{From top-left to bottom-right: first two coordinates of the initial configurations of Sections \ref{sec:outlier}, \ref{sec:geomdis}, \ref{sec:cauchydis},  \ref{sec:gaussdis},  \ref{sec:unif}. The blue points are the positions of the agents, the red arrows their consensus parameters.}
\label{inconf}
\end{figure}

\subsection{Examples where (DR) performs second best after (SP)}
\subsubsection{Configuration with one outlier}	
\label{sec:outlier}
We take into account $N = 9$ agents in dimension $d = 100$ for which the $j$-th spatial component of the $i$-th agent is given by the formula
\begin{align*}
(x_i)_j = \frac{1}{2}\cos(i + j\sqrt{2}) \quad \text{for} \quad j=1,\ldots,d \quad \text{and} \quad i=1,\ldots,N.
\end{align*}
The result obtained is a set of points non-homogeneously distributed over an almost spherical configuration, which, projected in $\R^2$, resembles an ellipse. A similar configuration is used for the consensus parameter of each agent, for which we have
\begin{align*}
(v_i)_j = \sin(i \sqrt{3} - j) \quad \text{for} \quad j=1,\ldots,d \quad \text{and} \quad i=1,\ldots,N -1;
\end{align*}
the initial consensus parameter of the $N$-th agent is instead the vector with all entries set equal to $10$.


\vspace{0.3cm}
\begin{center}
\setlength{\tabcolsep}{10pt}
\renewcommand{\arraystretch}{2}
\begin{tabular}{|c|c|c|c|c|c|c|c|}
\hline $N$ & $\theta$ & $\beta$ & $d$ & $\tau_0$ & $\hat{T}$ & $\tau$ & $V(0)-\gamma(X(0))^2$ \\
\hline \hline $9$ & $5$ & $0.6$ & $100$ & $7.33\cdot 10^{-4}$& $115.17$ & $10^{-2}$ & $1031.3$ \\
\hline
\end{tabular}
\end{center}


\vspace{0.3cm}
The following table reports the performance of the different kinds of control taken into account:
\begin{center}
\setlength{\tabcolsep}{2pt}
\renewcommand{\arraystretch}{1.2}
\resizebox{\columnwidth}{!}{
\begin{tabular}{|c||c|c|c|c|c|c|c|c|c|c|c|c|c|c|c|c|}
\hline Control & (SP) & (U) & (R) & (R) & (DR) & (DR) & (DR) & (DR) & (DR) & (DR) & (DR) & (DR) & (DR) & (DR) & (DR) & (DR) \\
\hline \hline $k$ 			   & - & - & - & - & 1 & 1& 5& 5& 10& 10 & 25 & 25 & 32 & 40 & 55 & 55  \\
\hline $W(0)-\gamma(Y(0))^2$ & - & - & - & - & 202.0 & 1014.2 & 509.9 & 870.2& 1651.4 & 1072.3 & 1035.2 & 582.18 & 933.0 & 1273.1 & 1054.5 & 1046.5  \\
\hline $T_0$ 	  		  & 27.78 & 87.21 & 87.90 & 88.79 & 69.75 & 30.98 & 44.27 & 30.47 & 35.55& 29.65 & 27.8 & 44.75 & 30.65 & 32.49 & 28.20 & 28.19  \\
\hline $T_{0.5}$ 		  & 5.44 & 22.96 & 22.64 & 23.21 & 5.92 & 5.44 & 5.44 & 5.44& 5.44& 5.44 & 5.44 & 5.44 & 5.44 & 5.44 & 5.44 & 5.44  \\
\hline $T_{S}$  & - & - & - & - & 13.47 & 22.25 & 17.81 & 22.39& 32.62& 29.03 & 27.8 & 23.14 & 26.2 & 28.94 & 28.20 & 28.15  \\
\hline
\end{tabular}
}
\end{center}
\vspace{0.3cm}

We first observe that if the system is left alone, with no control acting on it, the quantity $V(t)-\gamma(X(t))^2$ decreases only from $1031.3$ to $946.2$ at time 100, from which we can infer that the system would not reach the consensus region without an external intervention. Notice that the Sparse Control (SP) is the fastest; this shall be a common feature of all our experiments, as expected by its optimality shown in \cite[Proposition 3]{CFPT14}. The uniform control (U) and the random control (R) perform similarly and both take more than three times longer to reach the consensus region as (SP). The control (DR) has comparable performances to (SP), and very surprisingly even when projecting to dimension $k = 1$ the system reaches the consensus region faster than with the controls (U) and (R).

\begin{figure}[H]
\centering
\begin{minipage}{.5\textwidth}
  \centering
  \includegraphics[scale=0.22]{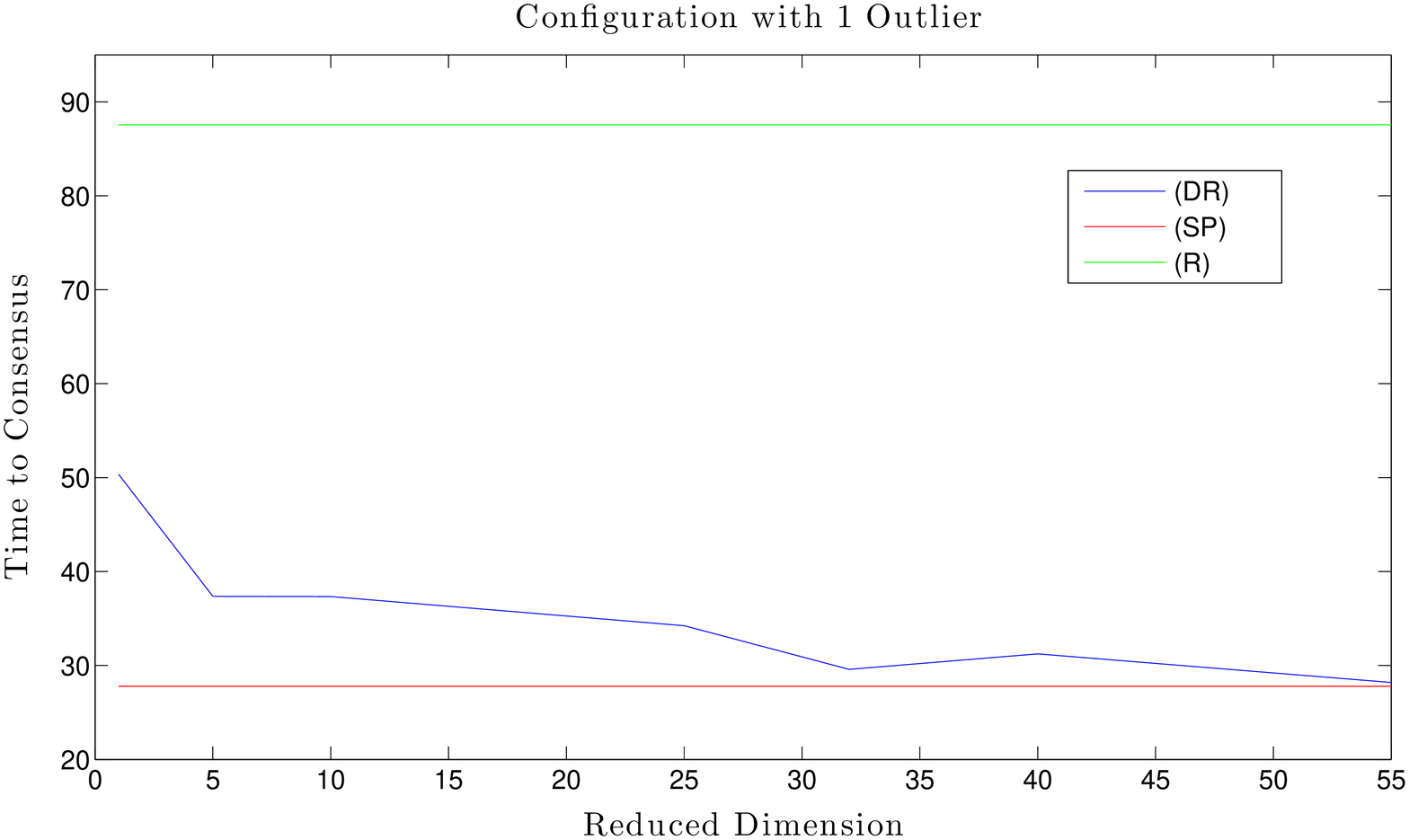}
  \caption{Time to consensus for (DR) in function of the projected dimension $k$, and comparison with (SP) and (R).}
  \label{time_inconf4}
\end{minipage}%
\begin{minipage}{.5\textwidth}
  \centering
  \includegraphics[scale=0.2]{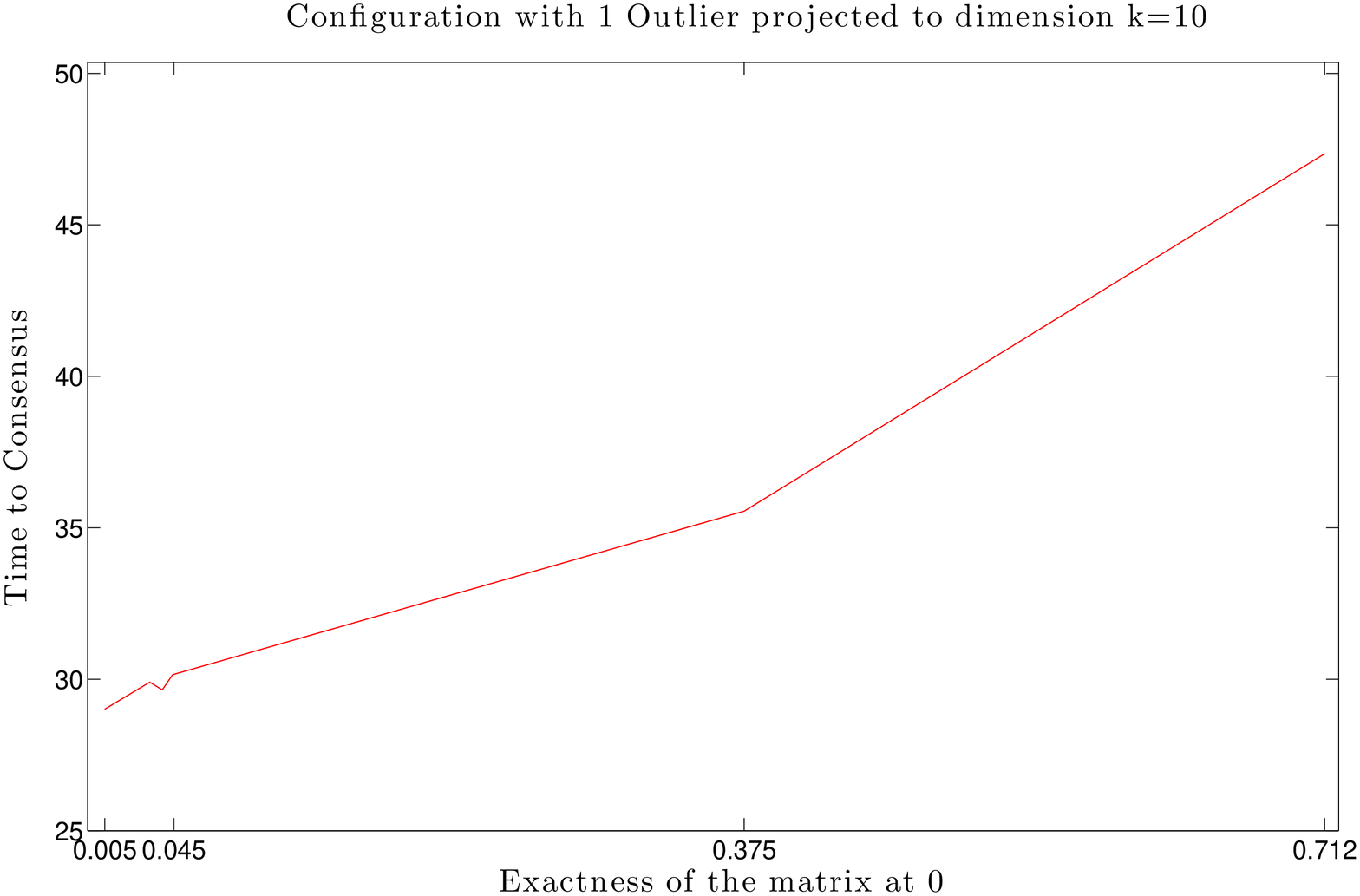}
  \caption{Time to consensus in function of the exactness of the matrix at 0 for the fixed dimension $k = 10$.}
  \label{exact_inconf4}
\end{minipage}%
\end{figure}

In Figure \ref{time_inconf4}, we illustrate the time $T_0$ the system takes to reach the consensus region as a function of  the projected dimension $k$ for the control (DR). If multiple tests are made with the same dimension $k$, we consider an average of the results. We also report, in different colors, the values of $T_0$ we obtain with the control (SP) and the control (R) (blue and green line, respectively). It can be seen how the performance of (DR) is basically the same as (SP) even if we reduce the dimensionality by 80\% .

Up to now, we don't have any procedure to test if the randomly generated matrix we use to implement the control (DR) satisfies the requested properties of Theorem \ref{eq:convergence}. Moreover, to get a precise answer, we would need to gather information which belongs to the high-dimensional system beyond time 0, something which we are not allowed to know in advance. We claim, however, that the quantity, which we call {\it the exactness of the matrices at 0},
\begin{align*}
E_M=\left|1-\frac{V(0)}{W(0)}\right| = \left|1-\frac{\sum_{i=1}^N \|v^{\perp}_i(0)\|^2}{\sum_{i=1}^N \|Mv^{\perp}_i(0)\|^2}\right|.
\end{align*}
is a measure of how good the matrix $M$ is. To show that, we have considered six different $M \in \R^{k\times d}$ for $k = 10$ and their respective time to the consensus region: we report in Figure \ref{exact_inconf4} the time to consensus for the system in function of the exactness of the matrices at 0. A correlation between how $E_M$ is close to zero and how effective is the control, is clearly visible.


\subsubsection{Configuration generated by a geometric distribution}	\label{sec:geomdis}
In this section we consider a system where the locations are distributed as in the example before, while the consensus parameters are given by the formula
\begin{align*}
(v_i)_j = (1.2)^{(i-1)/2} \cdot \sin(i \sqrt{3} - j) \quad \text{for} \quad j=1,\ldots,d \quad \text{and} \quad i=1,\ldots,N -1;
\end{align*}
This results in a more heterogeneous situation at the beginning.  We also increase the dimension $d$ to 500, the strength of the force $\theta$ to $20$ and $\beta$ to $0.65$.
\vspace{0.3cm}
\begin{center}
\setlength{\tabcolsep}{10pt}
\renewcommand{\arraystretch}{2}
\begin{tabular}{|c|c|c|c|c|c|c|c|}
\hline $N$ & $\theta$ & $\beta$ & $d$ & $\tau_0$ & $\hat{T}$ & $\tau$ & $V(0)-\gamma(X(0))^2$ \\
\hline
\hline 15 & 20 & $0.65$ & 500 & $1.26\cdot 10^{-4}$& $51.82$ & $10^{-2}$ & $1195.5$ \\
\hline
\end{tabular}
\end{center}

\vspace{0.3cm}
The following table summarizes the results of the experiments:
\begin{center}
\setlength{\tabcolsep}{2pt}
\renewcommand{\arraystretch}{1.2}
\resizebox{\columnwidth}{!}{
\begin{tabular}{|c||c|c|c|c|c|c|c|c|c|c|c|c|c|}
\hline Control & (SP) & (U) & (R) & (R) & (DR) & (DR) & (DR) & (DR) & (DR) & (DR) & (DR) & (DR) & (DR) \\
\hline
\hline $k$ 			   & - & - & - & - & 1 & 1 & 10 & 10 & 50 & 50 & 50 & 100 & 100  \\
\hline $W(0)-\gamma(Y(0))^2$ & - & - & - & - & 1194.2 & 1191.9& 1194.3 & 1197.5 & 1007.2 & 1199.7 & 1178.2 & 1079.1 & 1204.7  \\
\hline $T_0$ 	  		  & 23.45 & 38.02 & 38.10 & 39.82 & 40.96 & 45.41& 26.66 & 29.81 & 27.45 & 24.33 & 26.48 & 26.88 & 24.02  \\
\hline $T_{0.5}$ 		  & 5.49 & 7.60 & 7.68 & 7.66 & 7.455 & 9.04& 5.64 & 5.86 & 5.55 & 5.5 & 5.59 & 5.59 & 5.50  \\
\hline
\end{tabular}
}
\end{center}
\vspace{0.3cm}

If we let the system free to evolve, the quantity $V(t)-\gamma(X(t))^2$ decreases only from $1195.5$ to $1122.3$ at time 30. The slowness of the decay implies the necessity of a control. The uniform control (U) and the random control (R) perform similarly, as in the example before. However, the control (DR) overwhelms both when the projected dimension is large enough ($k \geq 10$). Figure \ref{time_inconf7} shows the performance of (DR) in function of $k$ and compares it with (R) and (SP).


\begin{figure}[H]
\centering
Time to consensus for (DR) in function of the projected dimension $k$, and comparison with (SP) and (R)
\begin{minipage}{.5\textwidth}
  \centering
  \vspace{0.2cm}
  \includegraphics[scale=0.22]{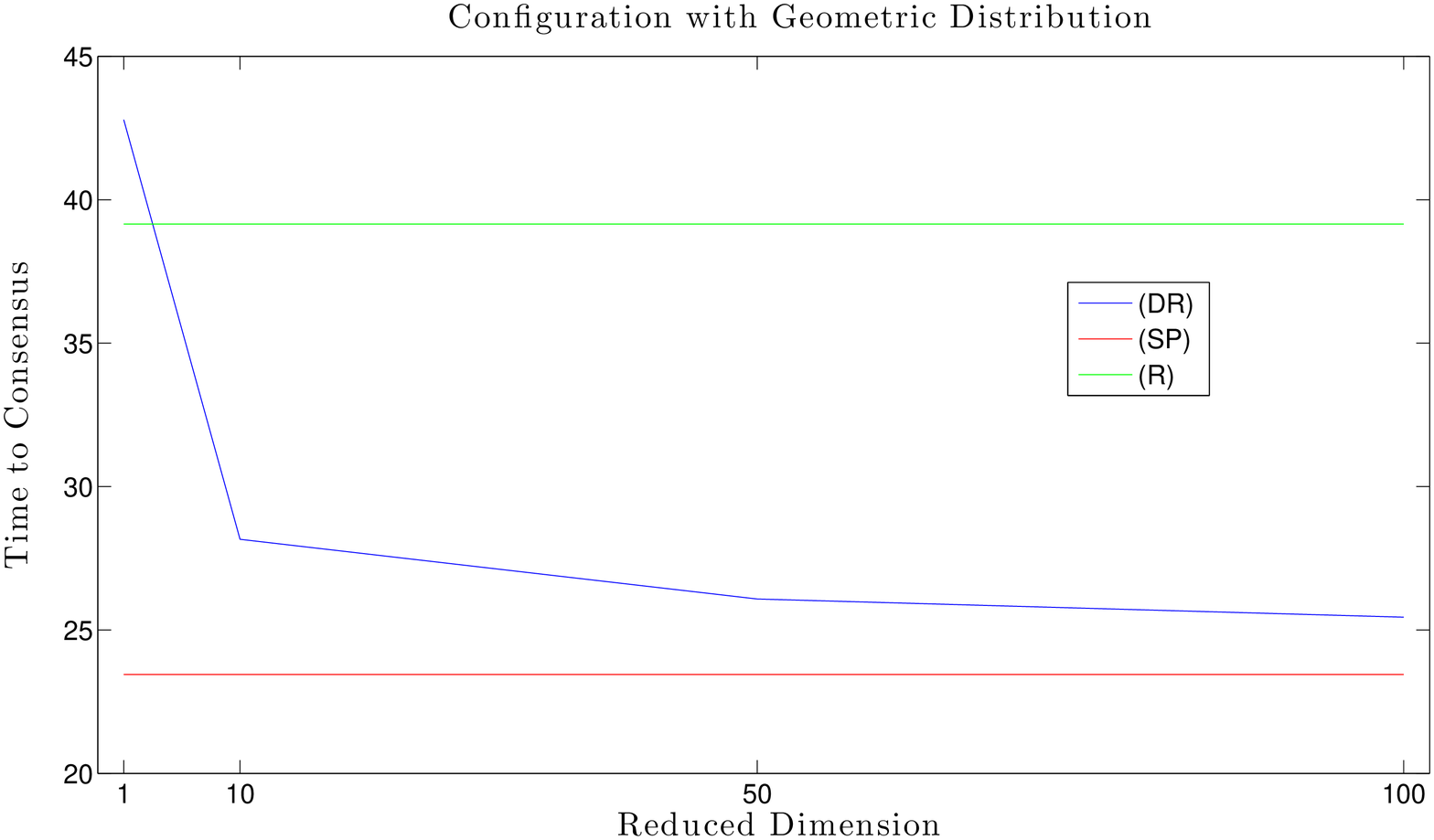}
  \caption{}
  \label{time_inconf7}
\end{minipage}%
\begin{minipage}{.5\textwidth}
  \centering
  \vspace{0.2cm}
  \includegraphics[scale=0.22]{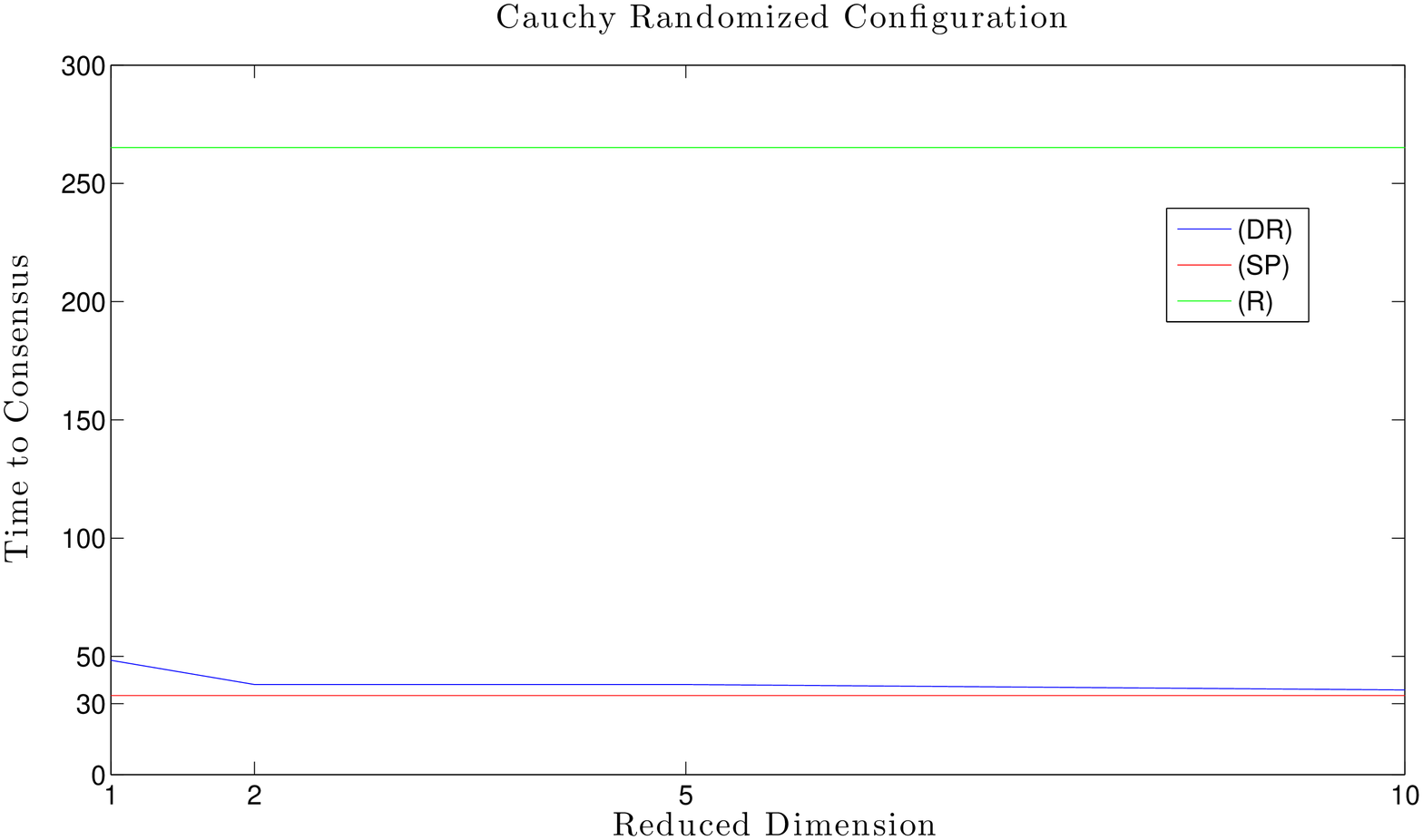}
  \caption{}
  \label{time_inconf9}
\end{minipage}%
\end{figure}

\subsubsection{Configuration generated by a Cauchy distribution}	\label{sec:cauchydis}
For the system considered in this section, the initial configuration is calculated as follows: the $j$-th spatial component of the $i$-th agent is the value of a normal distribution with expected value $0$ and standard deviation $1$, independently selected for different $i$ and $j$. The $j$-th component of the consensus parameter of the $i$-th agent is ruled by a Cauchy distribution, whose density is given by
\begin{align*}
f(x)=\frac{b}{\pi(b^2+x^2)}.
\end{align*}  
We choose the height to be $b=1/40$ (to get a reasonably large $V(0)$ in the computations). The initial configuration is generated once and then fixed for all the experiments with the different controls (SP), (R), (U) and (DR).


Below we list the parameters we fix for this section:
\vspace{0.3cm}
\begin{center}
\setlength{\tabcolsep}{10pt}
\renewcommand{\arraystretch}{2}
\begin{tabular}{|c|c|c|c|c|c|c|c|}
\hline $N$ & $\theta$ & $\beta$ & $d$ & $\tau_0$ & $\hat{T}$ & $\tau$ & $V(0)-\gamma(X(0))^2$ \\
\hline
\hline 25 & 5 & $0.6$ & 100 & $3.77\cdot 10^{-4}$& $214.76$ & $10^{-2}$ & $464.03$ \\
\hline
\end{tabular}
\end{center}

\vspace{0.3cm}
The following table reports the performances of the various controls:
\vspace{0.3cm}
\begin{center}
\setlength{\tabcolsep}{2pt}
\renewcommand{\arraystretch}{1.2}
\begin{tabular}{|c||c|c|c|c|c|c|c|c|c|c|c|}
\hline Control & (SP) & (U) & (R) & (DR) & (DR) & (DR) & (DR) & (DR) & (DR) & (DR) \\
\hline
\hline $k$ 			   & - & - & -  & 1 & 1 & 2 & 5 & 5 & 10 & 10  \\
\hline $W(0)-\gamma(Y(0))^2$ & - & - & -  & 461.04 & 461.04 & 475.48 & 464.39 & 464.39 & 465.00 & 465.00  \\
\hline $T_0$ 	  		  & 33.45 & 266.44 & 265.14  & 48.04 & 48.6 & 38.07 & 37.98 & 38.16 & 36.11 & 35.41  \\
\hline $T_{0.5}$ 		  & 6.1 & 70.55 & 68.54 & 6.1 & 6.1 & 6.1 & 6.1 & 6.1 & 6.1 & 6.1  \\
\hline
\end{tabular}
\end{center}
\vspace{0.3cm}

As in the examples before, the control (DR) clearly outperforms both (R) and (U), and in this case even for $k = 1$. Figure \ref{time_inconf9} compares the effectiveness of the controls (DR) (in function of $k$), (R) and (SP). We point out that, even in this situation, a control is necessary to steer the system to consensus since the quantity $V(t)-\gamma(X(t))^2$ decreases only from $464$ to $436.5$ at time 50 if no control is applied.


\subsection{Examples in which the performances of (R) and (U) are comparable to (DR)}

\subsubsection{Configuration generated by a normal distribution} \label{sec:gaussdis}
\label{seq:uniform}
In this example, the $j$-th spatial (resp., consensus parameter) component of the $i$-th agent is independently generated by a normal distribution with expected value $0$ and standard deviation $10$ (resp., $8$). As in Section \ref{sec:cauchydis}, we generate the initial configuration once and we use it for all the experiments with the controls.

The parameters used for this configuration are listed in the table below, and after it we report the performances of the various controls:
\vspace{0.3cm}
\begin{center}
\setlength{\tabcolsep}{10pt}
\renewcommand{\arraystretch}{2}
\begin{tabular}{|c|c|c|c|c|c|c|c|}
\hline $N$ & $\theta$ & $\beta$ & $d$ & $\tau_0$ & $\hat{T}$ & $\tau$ & $V(0)-\gamma(X(0))^2$ \\
\hline
\hline 10 & 20 & $0.65$ & 500 & $2.55\cdot 10^{-6}$& 165.68 & $5  \cdot 10^{-3}$ & $27458$ \\
\hline
\end{tabular}
\end{center}

\vspace{-0.1cm}
\begin{center}
\setlength{\tabcolsep}{2pt}
\renewcommand{\arraystretch}{1.2}
\resizebox{\columnwidth}{!}{
\begin{tabular}{|c||c|c|c|c|c|c|c|c|c|c|c|c|c|c|c|}
\hline Control & (SP) & (U) & (R) & (R) & (DR) & (DR) & (DR) & (DR) & (DR) & (DR) & (DR) & (DR) & (DR) & (DR) & (DR) \\
\hline
\hline dim. $k$ 			   & - & - & - & - & 1 & 1& 2 & 5 & 10 & 10 & 20 & 50 & 50 & 100 & 100 \\
\hline $W(0)-\gamma(Y(0))^2$ & - & - & - & - & 27496 & 27469 & 27421 & 27425& 27458 & 27493 & 27464 & 27482& 27481 & 27495 & 27498 \\
\hline $T_0$ 	  		  & 82.65 & 84.56 & 85.82 & 85.25 & 129.28 & 153.02 & 115.91 & 99.79& 95.31 & 100.18 & 96.7& 89.79 & 91.02 & 91.67 & 89.60 \\
\hline $T_{0.5}$ 		  & 24.09 & 24.13 & 24.15 & 24.13 & 36.195 & 42.76 & 29.28 & 26.47 &26.75 & 25.25 & 25.32 & 24.7& 24.74 & 24.44 & 24.32 \\
\hline $T_{S}$ 		 & - & - & - & - & 68.06 & 62.94 & 76.91 & 80.51& 80.55 & 77.30 & 81.57 & 82.25& 82.25 & 82.49 & 82.57 \\
\hline
\end{tabular}
}
\end{center}
\vspace{0.3cm}

This time the controls (R) and (U) are quasi-optimal, performing in almost the same way as the benchmark control (SP). Figure \ref{time_inconf8} shows that the control (DR) behaves similarly to (R) and (SP) up to a reduced dimension $k = 50$ (hence up to $10\%$ of the original dimension): from that point on the efficiency rapidly deteriorates, making the control unfeasible.



\begin{figure}[H]
\centering
Time to consensus for (DR) in function of the projected dimension $k$, and comparison with (SP) and (R)
\begin{minipage}{.5\textwidth}
  \centering
  \vspace{0.22cm}
  \includegraphics[scale=0.22]{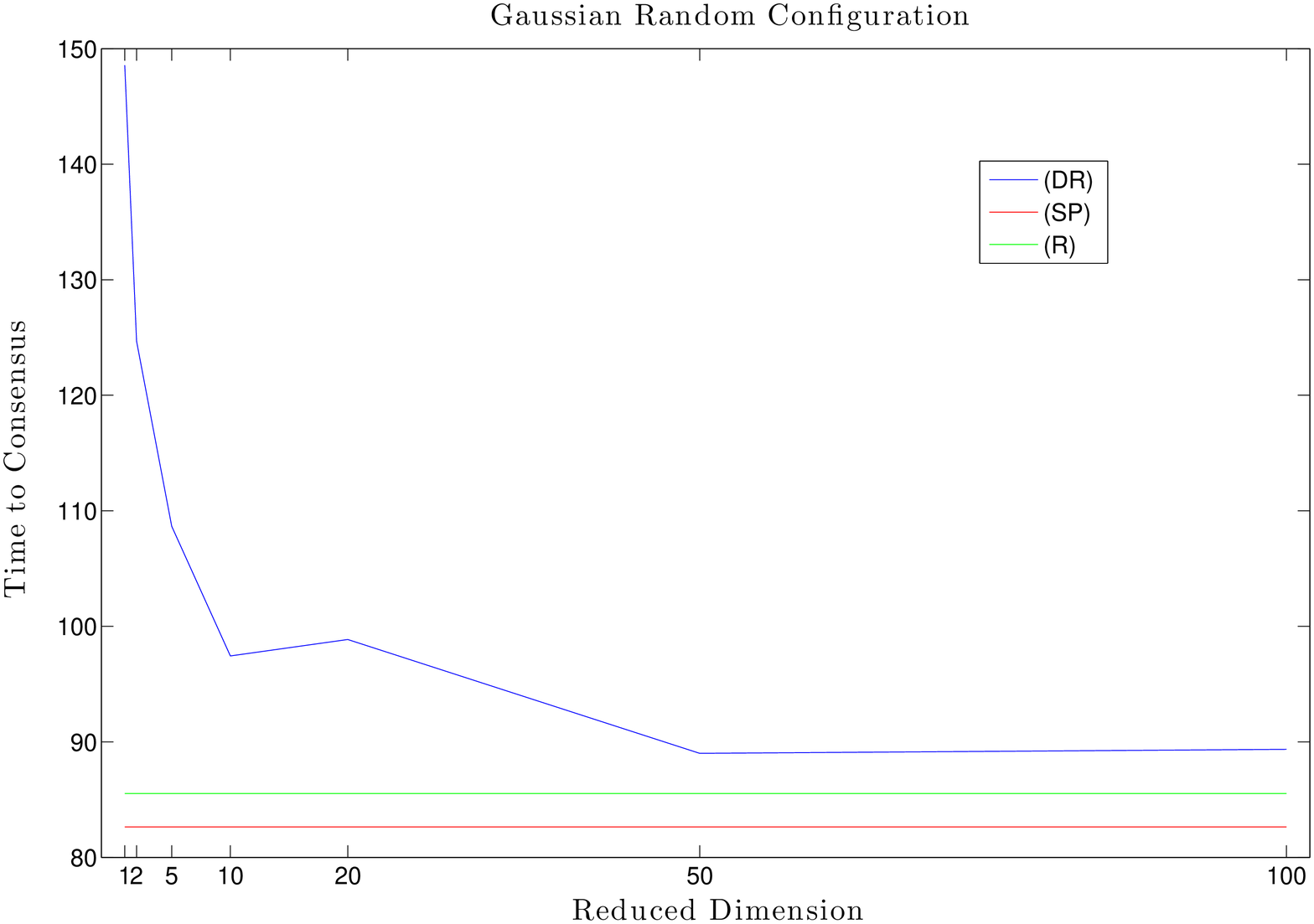}
  \caption{}
  \label{time_inconf8}
\end{minipage}%
\begin{minipage}{.5\textwidth}
  \centering
  \vspace{0.2cm}
  \includegraphics[scale=0.22]{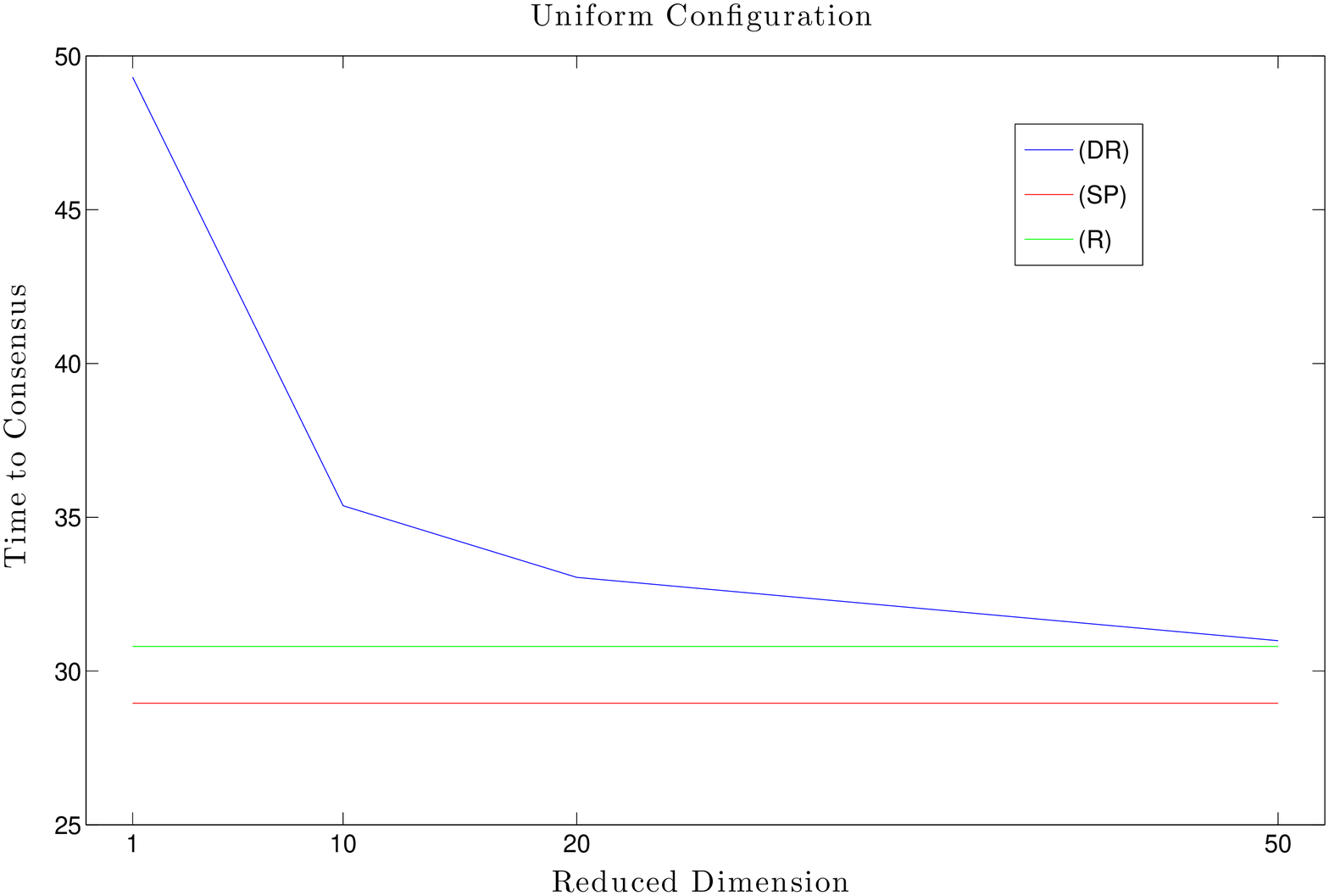}
  \caption{}
  \label{time_inconf1}
\end{minipage}%
\end{figure}

\subsubsection{Uniform configuration}  \label{sec:unif}
As last example we consider a configuration similar to the one of Section \ref{sec:outlier}: the $j$-th spatial and consensus parameter components of the $i$-th agent are both given by
\begin{align*}
(x_i)_j = (v_i)_j = \cos(i + j\sqrt{2}) \quad \text{for} \quad j=1,\ldots,d \quad \text{and} \quad i=1,\ldots,N.
\end{align*}
The following tables report the parameters of the configuration taken into account and the outcomes with the different controls:

\vspace{0.3cm}
\begin{center}
\setlength{\tabcolsep}{10pt}
\renewcommand{\arraystretch}{2}
\begin{tabular}{|c|c|c|c|c|c|c|c|}
\hline $N$ & $\theta$ & $\beta$ & $d$ & $\tau_0$ & $\hat{T}$ & $\tau$ & $V(0)-\gamma(X(0))^2$ \\
\hline
\hline 15 & 5 & $0.8$ & 200 & $1.91\cdot 10^{-5}$& $59.48$ & $10^{-3}$ & $98.30$ \\
\hline
\end{tabular}
\end{center}


\vspace{-0.1cm}
\begin{center}
\setlength{\tabcolsep}{2pt}
\renewcommand{\arraystretch}{1.2}
\resizebox{\columnwidth}{!}{
\begin{tabular}{|c||c|c|c|c|c|c|c|c|c|c|c|c|}
\hline Control & (SP) & (U) & (R) & (R) & (DR) & (DR) & (DR) & (DR) & (DR) & (DR) & (DR) & (DR) \\
\hline
\hline $k$ 			   & - & - & - & - & 1 & 1 & 10 & 10	 & 20 & 20 & 50 & 50 \\
\hline $W(0)-\gamma(Y(0))$ & - & - & - & - & 95.98 & 43.50 & 95.85 & 96.70 & 77.12 & 101.65 & 97.02 & 122.83 \\
\hline $T_0$ 	  		  & 28.95 & 29.95 & 30.86 & 30.74 & 53.97 & 44.47 & 38.13 & 32.35 & 33.08 & 32.73 & 29.41 & 32.45 \\
\hline $T_{0.5}$ 		  & 7.99 & 8.30 & 8.30 & 8.31 & 9.74 & 9.21 & 8.92 & 8.15	 & 8.21 & 8.17 & 7.99 & 8.14 \\
\hline
\end{tabular}
}
\end{center}
\vspace{0.3cm}

As before, (R) and (U) perform similarly to (SP); (DR) is able to compete up to a dimension reduction of $25\%$ of the original dimension ($k = 50$). From there on, its efficiency steadily declines. This phenomenon can be witnessed in Figure \ref{time_inconf1}.

	
\subsection{Conclusions from the experiments}
In this section we summarize the conclusions that can be drawn from the list of experiments reported in this numerical section.
\begin{enumerate} 
\item A common feature of all the experiments is that the control (DR) is highly competitive with respect to the benchmark control (SP) up to a reduced dimension which is $10\%$ of the original one. Indeed, in this case (DR) takes between 5 to 22\% more time than (SP) to steer the system to consensus. This suggests that the approach of dimension reduction works in general much better practically than theoretically, and that our analysis in Theorem \ref{eq:convergence} is quite conservative.
\item The dimension of the matrix is not the only necessary ingredient to obtain a competitive control: a matrix should also fulfill the Johnson--Lindenstrauss property for certain points of the high dimensional system. Since to check the latter condition we need information regarding the future development of the system, we need to design different criteria to distinguish ``good'' matrices versus ``bad'' ones. In Section \ref{sec:outlier}, we have seen that an efficient sieve is the notion of \emph{exactness of a matrix at $0$}: the smaller this value is, the better the control shall perform, according to the empirical data we have gathered.
\item There is no proof yet that random sparse control (R) forces the system to enter the consensus region almost surely for every configuration, but numerical experiments suggest this behavior. Furthermore, it is interesting to notice that the time to consensus obtained by the use of the uniform control (U) is always very close to the one we get by using the random sparse control strategy (R): this strongly hints that the expected value of the time to consensus of the random control (R) could be very near or even equal to the one of (U).
\item A common feature of the last two examples is the ``relative homogeneity'' of the consensus para\-meters with respect to the mean consensus para\-meter: by this we mean that the consensus parameters of all the agents compete to be the farthest away from it, and thus the sparse control will jump from one to another continuously, showing a chattering behavior. In contrast, all the first three experiments feature a relatively small subgroup of agents whose consensus parameters are the farthest away from the mean consensus parameter by a considerable margin. These are the case where the controls (SP) and (DR) are substantially more efficient than (R) and (U): by firmly acting on the most ``badly behaving'' agents, we are able to steer the system to consensus faster than employing control strategies which are blind to the structure of the group. It is thus advisable to use sparse strategies only when the consensus parameters of the agents are sufficiently ``asymmetric'' at the starting point.

\end{enumerate}

\section{Appendix}

We need a technical lemma which can be found also in \cite{CFPT14}. But with a slightly different argument, we could improve the inequalities there and get rid of an $N^2$, which is important for estimating the time of entrance in the consensus region of controlled Cucker--Smale systems, depending on $N$.

\begin{Lemma}
\label{le:tech_up}
If there exists $\eta>0$ and $T>0$ such that 
\begin{align*}
V'(t)\leq -\eta \sqrt{V(t)}
\end{align*}
for almost every $t \in [0,T]$, then
\begin{align*}
V(t) \leq \left(\sqrt{V(0)}-\frac{\eta}{2}t\right)^2
\end{align*}
and
\begin{align*}
X(t) \leq 2X(0) + \frac{2}{\eta^2}V(0)^2.
\end{align*}
\end{Lemma}
\begin{proof}
Integrating the first assumption one has
\begin{align*}
\int_0^t \frac{V'(s)}{\sqrt{V(s)}} \ ds \leq -\eta t
\end{align*}\
and hence 
\begin{align*}
\sqrt{V(t)}-\sqrt{V(0)}=\frac{1}{2}  \int_0^t \frac{V'(s)}{\sqrt{V(s)}} \ ds \leq -\frac{\eta}{2}t.
\end{align*}
Furthermore, to prove the second statement of the lemma we observe
\begin{align*}
 \int_0^t \sqrt{V(s)} \ ds \leq -\frac{1}{\eta} \int_0^t V'(s) \ ds = -\frac{1}{\eta} (V(t)-V(0)) \leq \frac{1}{\eta} V(0).
\end{align*}
On the other hand, using the (vector-valued) Minkowski inequality in the second step
\begin{align*}
\sqrt{X(t)}&=\left(\frac{1}{2N^2} \sum_{i,j} \|x_i(t)-x_j(t)\|^2\right)^{1/2} \\
&\leq {\left(\frac{1}{2N^2} \sum_{i,j} \|x_i(0)-x_j(0)\|^2\right)^{1/2}} + \left(\frac{1}{2N^2} \sum_{i,j} \left(\int_0^t \|v_i(s)-v_j(s)\| \ ds\right)^2 \right)^{1/2} \\
&\leq \sqrt{X(0)}+\int_0^t  \left(\frac{1}{2N^2} \sum_{i,j} \|v_i(s)-v_j(s)\|^2 \right)^{1/2}  \ ds\\
&= \sqrt{X(0)} + \int_0^t \sqrt{V(s)} \ ds \\
&\leq \sqrt{X(0)} -  \frac{1}{\eta} \int_0^t V'(s) \ ds \\
&\leq \sqrt{X(0)} + \frac{1}{\eta} V(0)
\end{align*}
and furthermore by $(x+y)^2 \leq 2x^2+2y^2$ it follows 
\begin{align*}
X(t) \leq 2X(0) + \frac{2}{\eta^2} V^2(0). 
\end{align*}
\end{proof}

\subsection{Gronwall's estimates and variations on the theme}
We need to employ at several places Gronwall's estimates. However, besides the classical one, we need to develop a variation for piecewise continuous evolutions. Both are reported as follows.
\begin{Lemma}[Classical Gronwall's Lemma]
\label{le:class_gronwall}
Let $I=[a,b]$ be an interval on the real line. Let $\rho,\beta$ and $u$ be real valued functions and furthermore assume that $\beta$ is non-negative as well as continuous, $u$ is continuous and $\rho$ is non-decreasing on $I$ and integrable on $I$.

Assume that we have
\begin{align*}
u(t)\leq \rho(t)+\int_a^t \beta(s) u(s) \ ds, \quad \forall t \in I,
\end{align*} 
then
\begin{align*}
u(t)\leq \rho(t) e^{\int_a^t \beta(s) \ ds}, \quad \forall t \in I.
\end{align*}
\end{Lemma}
\begin{Lemma}[Discrete Gronwall's Lemma]
\label{le:discr_gronwall}
Let $I=[0,T]$ be an interval on the real line and $\tau \in [0,T]$.  Let $\rho,\beta_1,\beta_2$, and $u$ be real functions on $I$ such that
\begin{enumerate}
\item $\rho$ is non-decreasing and bounded on $I$,
\item $\beta_1$ is non-decreasing and continuous on $I$,
\item $\beta_2$ is non-negative and continuous on $I$ and
\item  $u$ be non-negative and continuous on $I$.
\end{enumerate}
Assume that for every $t \in [0,T]$ it holds: Let $n \in \N_0$ such that $n\tau \leq t < (n+1)\tau$ and assume
\begin{align*}
u(t)\leq (\rho(t)-\rho(n\tau))+  (1+\beta_1(t) - \beta_1(n\tau))u(n\tau)+\int_{n\tau}^t \beta_2(s) u(s) \ ds \quad \text{for all}\quad t \in I.
\end{align*} 
Then
\begin{align*}
u(t)&\leq u(0)e^{\beta_1(t)-\beta_1(0)+\int_{0}^t \beta_2(s) \ ds}   + (\rho(t)-\rho(0)) e^{\beta_1(t)-\beta_1(0)+\int_{0}^t \beta_2(s) \ ds}.
\end{align*}
\end{Lemma}
\begin{proof}
The proof uses $n$ applications of Gronwall's Lemma for the intervals $[0,\tau],[\tau,2\tau],\ldots,[(n-1)\tau,n\tau]$, $[n\tau,t]$. The first application {over} $[n\tau,t]$ gives
\begin{align*}
u(t)\leq \left[(\rho(t)-\rho(n\tau))+  (1+\beta_1(t) - \beta_1(n\tau))u(n\tau)\right] e^{\int_{n\tau}^t \beta_2(s) \ ds}.
\end{align*}
The second application for the interval $[(n-1)\tau,n\tau]$ gives
\begin{align*}
u(n\tau)\leq \left[(\rho(n\tau)-\rho((n-1)\tau))+  (1+\beta_1(n\tau) - \beta_1((n-1)\tau))u((n-1)\tau)\right] e^{\int_{(n-1)\tau}^{n\tau} \beta_2(s) \ ds}.
\end{align*}
Plugging the last estimate into the first one we arrive at
\begin{align*}
u(t)&\leq (\rho(t)-\rho(n\tau))e^{\int_{n\tau}^t \beta_2(s) \ ds}+  \left[1+\beta_1(t) - \beta_1(n\tau)\right]\left[\rho(n\tau)-\rho((n-1)\tau)\right]e^{\int_{(n-1)\tau}^t \beta_2(s) \ ds} \\
&\quad+  \left[1+\beta_1(t) - \beta_1(n\tau)\right]\left[1+\beta_1(n\tau)-\beta_1((n-1)\tau)\right]u((n-1)\tau)e^{\int_{(n-1)\tau}^t \beta_2(s) \ ds}.
\end{align*}
Now, {by induction on this successive substitutions we obtain}
\begin{align*}
u(t)&\leq (\rho(t)-\rho(n\tau))e^{\int_{n\tau}^t \beta_2(s) \ ds} \\
 +&  \left[1+\beta_1(t) - \beta_1(n\tau)\right] \sum_{i=0}^{n-1} \left[\rho((n-i)\tau)-\rho((n-i-1)\tau)\right]e^{\int_{(n-i-1)\tau}^t \beta_2(s) \ ds}\hspace{-0.3cm}\prod_{j=n-i+1}^n \left[1+\beta_1(j\tau)-\beta_1((j-1)\tau)\right] \\
+&   u(0)e^{\int_{0}^t \beta_2(s) \ ds} \left[1+\beta_1(t) - \beta_1(n\tau)\right]\prod_{j=1}^n \left[1+\beta_1(j\tau)-\beta_1((j-1)\tau)\right].  
\end{align*}
Now we use
\begin{align*}
1+a\leq e^a, \quad \text{hence}\quad \prod_{i=1}^n (1+a_n) \leq e^{\sum_{i=1}^n a_n} 
\end{align*}
for $a,a_1,\ldots,a_n \in \R^+$ to get
\begin{align*}
u(t)&\leq (\rho(t)-\rho(n\tau))e^{\int_{n\tau}^t \beta_2(s) \ ds} \\
&\quad +  \sum_{i=0}^{n-1} \left[\rho((n-i)\tau)-\rho((n-i-1)\tau)\right]e^{\int_{(n-i-1)\tau}^t \beta_2(s) \ ds} e^{\beta_1(t)-\beta_1((n-i)\tau)} \\
&\quad +   u(0)e^{\int_{0}^t \beta_2(s) \ ds} e^{\beta_1(t)-\beta_1(0)} \\
&\leq u(0)e^{\beta_1(t)-\beta_1(0)+\int_{0}^t \beta_2(s) \ ds}   + (\rho(t)-\rho(0)) e^{\beta_1(t)-\beta_1(0)+\int_{0}^t \beta_2(s) \ ds}.
\end{align*}
\end{proof}

\section*{Acknowledgments}

Mattia Bongini, Massimo Fornasier, and Benjamin Scharf acknowledge the support of the ERC-Starting Grant project ``High-Dimensional Sparse Optimal Control''.

\bibliographystyle{plain}
\bibliography{JLbib}

\end{document}